\documentclass[reqno]{amsart}
\usepackage{fullpage}
\usepackage{amsmath}
\usepackage{amssymb}
\usepackage{amsaddr}
\usepackage{graphicx}
\newcommand{\myi}{\mathrm{i}}
\newcommand{\eps}{\epsilon}
\newcommand{\dbar}{\overline{\partial}}
\newcommand{\kr}{k_\mathrm{r}}
\newcommand{\ki}{k_\mathrm{i}}
\newcommand{\sgn}{\mathrm{sgn}}
\newcommand\xqed[1]{%
  \leavevmode\unskip\penalty9999 \hbox{}\nobreak\hfill
  \quad\hbox{#1}}
\newcommand\myendrmk{\xqed{$\triangleright$}}

\theoremstyle{plain}
\newtheorem{theorem}{Theorem}
\newtheorem{lemma}{Lemma}
\newtheorem{proposition}{Proposition}
\newtheorem{corollary}{Corollary}

\theoremstyle{definition}
\newtheorem{definition}{Definition}
\newtheorem{assumption}{Assumption}
\newtheorem{rhp}{Riemann-Hilbert Problem}
\newtheorem{dbarprob}[rhp]{$\dbar$ Problem}
\newtheorem{rhpdbar}[rhp]{Riemann-Hilbert-$\dbar$ Problem}
\newtheorem{remark}{$\triangleleft$ Remark}

\numberwithin{equation}{section}

\begin{document}
\title{Initial-boundary value problems for the defocusing nonlinear Schr\"odinger equation in the semiclassical limit}
\author{Peter D. Miller}
\address{Department of Mathematics, University of Michigan, East Hall, 530 Church St., Ann Arbor, MI 48109}
\email{millerpd@umich.edu}
\author{Zhenyun Qin}
\address{School of Mathematics and Key Laboratory of Mathematics for Nonlinear Science, Fudan University, Shanghai 200433, PR China}
\email{zyqin@fudan.edu.cn}
\date{\today}
\begin{abstract}
Initial-boundary value problems for integrable nonlinear partial differential equations have become tractable in recent years due to the development of so-called unified transform techniques.
The main obstruction to applying these methods in practice is that calculation of the spectral transforms of the initial and boundary data requires knowledge of too many boundary conditions, more than are required make the problem well-posed.  The elimination of the unknown boundary values is frequently addressed in the spectral domain via the so-called global relation, and types of boundary conditions for which the global relation can be solved are called \emph{linearizable}.  For the defocusing nonlinear Schr\"odinger equation, the global relation is only known to be explicitly solvable in rather restrictive situations, namely homogeneous boundary conditions of Dirichlet, Neumann, and Robin (mixed) type.  General nonhomogeneous boundary conditions are not known to be linearizable.  In this paper, we propose an explicit approximation for the nonlinear Dirichlet-to-Neumann map supplied by the defocusing nonlinear Schr\"odinger equation and use it to provide approximate solutions of general nonhomogeneous boundary value problems for this equation posed as an initial-boundary value problem on the half-line.  Our method sidesteps entirely the solution of the global relation.  The accuracy of our method is proven in the semiclassical limit, and we provide explicit asymptotics for the solution in the interior of the quarter-plane space-time domain.
\end{abstract}

\maketitle

\section{Introduction}
\label{sec:Introduction}
Consider the following initial-boundary value problem for the defocusing nonlinear Schr\"odinger equation on the positive half-line
\begin{equation}
\myi\eps\frac{\partial q}{\partial t} +\eps^2\frac{\partial^2q}{\partial x^2}-2|q|^2q=0,\quad q=q(x,t),\quad
x>0,\quad t>0,
\label{eq:NLS}
\end{equation}
with given initial data:
\begin{equation}
q(x,0)=q_0(x),\quad x>0,
\end{equation}
and with a given (generally nonhomogeneous) Dirichlet boundary condition at $x=0$:
\begin{equation}
q(0,t)=Q^\mathrm{D}(t),\quad t>0.
\end{equation}
Here $\epsilon>0$ is an arbitrary parameter.  Assuming that $q_0\in H^2(\mathbb{R}_+)$, $Q^\mathrm{D}\in C^2(\mathbb{R}_+)$, and that the compatibility condition $q_0(0)=Q^\mathrm{D}(0)$ holds, Carroll and Bu
\cite{CarrollB91} have established the existence of a unique classical global solution of this problem that is a continuously differentiable map from $t\in\mathbb{R}_+$ to $q\in L^2(\mathbb{R}_+)$ and that is a continuous map from $t\in \mathbb{R}_+$ to $q\in H^2(\mathbb{R}_+)$.

The defocusing nonlinear Schr\"odinger equation \eqref{eq:NLS} is an integrable equation, being the compatibility condition for the existence of a simultaneous general solution $\psi$ of the equation
\begin{equation}
\eps\frac{\partial\psi}{\partial x}=\mathbf{U}\psi,\quad \mathbf{U}:=\begin{bmatrix}-\myi k & q\\q^* & \myi k\end{bmatrix}
\label{eq:xproblem}
\end{equation}
and also of the equation
\begin{equation}
\eps\frac{\partial\psi}{\partial t}=\mathbf{V}\psi,\quad \mathbf{V}:=\begin{bmatrix}-2\myi k^2-\myi |q|^2 & 2kq + \myi\eps q_x\\
2kq^* - \myi\eps q^*_x & 2\myi k^2+\myi|q|^2\end{bmatrix},\quad q_x:=\frac{\partial q}{\partial x}.
\label{eq:tproblem}
\end{equation}
Here $k$ is a complex spectral parameter, and the compatibility condition is independent of $k$.
These two linear equations for $\psi$ comprise the \emph{Lax pair} for \eqref{eq:NLS}.
One of the earliest applications of the Lax pair representation of integrable equations was
the development of a transform technique based on the spectral theory of the spatial equation \eqref{eq:xproblem} of the Lax pair, the \emph{inverse-scattering transform}, for solving
initial-value problems posed for $x\in\mathbb{R}$ with initial data given at $t=0$; see
\cite{FaddeevT07} for a pedagogical description.  More recently, a \emph{unified transform method} has been developed involving the simultaneous use of both equations of the Lax pair to study mixed initial-boundary value problems of various types.  As a general reference for these methods that includes the specific details we will need in this paper, we refer to \cite{Fokas08}; there is also a website \cite{FokasWeb} that summarizes the salient features of the technique and has links to many original references.

For the defocusing nonlinear Schr\"odinger equation \eqref{eq:NLS} on the half-line $x>0$, the
unified transform method first advanced in \cite{FokasIS05} and also described in \cite{Fokas08} amounts to the following algorithm.  Recall the Pauli spin matrices
\begin{equation}
\sigma_1:=\begin{bmatrix}0 & 1\\1 & 0\end{bmatrix},\quad
\sigma_2:=\begin{bmatrix}0 & -\myi\\\myi & 0\end{bmatrix},\quad\text{and}\quad
\sigma_3:=\begin{bmatrix}1 & 0\\0 & -1\end{bmatrix},
\end{equation}
and let $Q^\mathrm{N}(t):=\eps q_x(0,t)$.
Firstly, define the following special solutions of the Lax pair:
\begin{equation}
\label{eq:X-problem}
\eps\frac{d\mathbf{X}}{dx}(x;k)=\begin{bmatrix}-\myi k & q_0(x)\\q_0(x)^* & \myi k\end{bmatrix}\mathbf{X}(x;k),\quad
\lim_{x\to +\infty}\mathbf{X}(x;k)e^{\myi kx\sigma_3/\eps}=\mathbb{I},\end{equation}
and
\begin{equation}
\label{eq:T-problem}
\eps\frac{d\mathbf{T}}{dt}(t;k)=\begin{bmatrix} -2\myi k^2-\myi|Q^\mathrm{D}(t)|^2 & 2kQ^\mathrm{D}(t)+\myi Q^\mathrm{N}(t)\\
2kQ^\mathrm{D}(t)^*-\myi Q^\mathrm{N}(t)^* & 2\myi k^2+\myi|Q^\mathrm{D}(t)|^2\end{bmatrix}\mathbf{T}(t;k),\quad
\lim_{t\to +\infty}\mathbf{T}(t;k)e^{2\myi k^2 t\sigma_3/\eps}=\mathbb{I}.
\end{equation}
The spectral transforms of $q_0$, $Q^\mathrm{D}$, and $Q^\mathrm{N}$ are then given by
\begin{equation}
a(k):=X_{22}(0;k),\quad b(k):=X_{12}(0;k),\quad A(k):=T_{22}(0;k),\quad B(k):=T_{12}(0;k),
\label{eq:basic-transforms}
\end{equation}
and by elementary symmetries one also has that
\begin{equation}
a(k^*)^*=X_{11}(0;k),\quad b(k^*)^*=X_{21}(0;k),\quad A(k^*)^*=T_{11}(0;k),\quad
B(k^*)^*=T_{21}(0;k).
\label{eq:conjugate-transforms}
\end{equation}
The second column of $\mathbf{X}(x;k)$ is analytic and bounded in $k$ for $\Im\{k\}>0$ whenever $x\ge 0$ (and hence the same is true of $a(k)$ and $b(k)$).  The second column of $\mathbf{T}(t;k)$ is analytic and bounded in $k$ for $\Im\{k^2\}>0$ whenever $t\ge 0$ (and hence the same is true of $A(k)$ and $B(k)$).

Next, given these functions of $k$, one formulates a Riemann-Hilbert problem.  Let $\Sigma$ denote the contour $\Im\{k^2\}=0$ with each of the four half-line arcs of $\Sigma\setminus\{0\}$ assigned an orientation such that the domain $\Im\{k^2\}>0$ lies on the left.  On each of the four arcs we define a \emph{jump matrix} as follows:
\begin{equation}
\mathbf{J}(k):=
\begin{bmatrix}1-|\gamma(k)|^2 & \gamma(k)e^{-2\myi\theta(k;x,t)/\eps}\\-\gamma(k)^*e^{2\myi\theta(k;x,t)/\eps} & 1\end{bmatrix},\quad \arg(k)=0,
\label{eq:exact-jump-positive}
\end{equation}
\begin{equation}
\mathbf{J}(k):=
\begin{bmatrix}1 & 0\\-\Gamma(k)e^{2\myi\theta(k;x,t)/\eps} & 1\end{bmatrix},\quad 
\arg(k)=\pi/2,
\label{eq:exact-jump-positive-imaginary}
\end{equation}
\begin{equation}
\mathbf{J}(k):=
\begin{bmatrix}1 & \Gamma(k^*)^*e^{-2\myi\theta(k;x,t)/\eps}\\0 & 1\end{bmatrix},\quad
\arg(k)=-\pi/2,\quad \text{and}
\label{eq:exact-jump-negative-imaginary}
\end{equation}
\begin{equation}
\mathbf{J}(k):=\begin{bmatrix} 1 & (\Gamma(k)^*-\gamma(k))e^{-2\myi\theta(k;x,t)/\eps}\\
(\gamma(k)^*-\Gamma(k))e^{2\myi\theta(k;x,t)/\eps} & 1-|\gamma(k)|^2-|\Gamma(k)|^2+\Gamma(k)\gamma(k)+\Gamma(k)^*\gamma(k)^*\end{bmatrix},\quad \arg(-k)=0.
\label{eq:exact-jump-negative}
\end{equation}
Here, the spectral coefficients in the jump matrix are
\begin{equation}\label{gamma}
\gamma(k):=\frac{b(k)}{a(k)^*}\quad\text{and}\quad\Gamma(k):=\frac{B(k^*)^*}{a(k)d(k)},\quad \text{where}\quad d(k):=a(k)A(k^*)^*-b(k)B(k^*)^*,
\end{equation}
and all of the dependence on $x$ and $t$ appears explicitly through the function
\begin{equation}
\theta(k;x,t):=kx+2k^2t.
\label{eq:theta-define}
\end{equation}
The Riemann-Hilbert problem is then the following.
\begin{rhp}
Find a $2\times 2$ matrix $\mathbf{M}(k)$ with the following properties:
\begin{itemize}
\item[]\textbf{Analyticity:}  $\mathbf{M}$ is analytic and uniformly bounded for $k\in\mathbb{C}\setminus\Sigma$, taking boundary values $\mathbf{M}_\pm(k)$ on each of the four rays of $\Sigma$ from the domain where $\pm\Im\{k^2\}>0$.  
\item[]\textbf{Jump Condition:}  The boundary values are related on each ray of $\Sigma$ by the jump condition 
\begin{equation}
\mathbf{M}_+(k)=\mathbf{M}_-(k)\mathbf{J}(k),\quad k\in\Sigma\setminus\{0\},
\end{equation}
where $\mathbf{J}:\Sigma\setminus\{0\}\to\mathrm{SL}(2,\mathbb{C})$ is defined by \eqref{eq:exact-jump-positive}--\eqref{eq:exact-jump-negative}.
\item[]\textbf{Normalization:}  $\mathbf{M}(k)\to\mathbb{I}$ as $k\to\infty$.
\end{itemize}
\label{rhp-original}
\end{rhp}
From the solution of Riemann-Hilbert Problem~\ref{rhp-original}, which depends parametrically on $x$, $t$, and $\eps$,  one obtains a solution of the defocusing nonlinear Schr\"odinger equation by taking the limit
\begin{equation}
q(x,t)=2\myi\lim_{k\to\infty}kM_{12}(k).
\label{eq:q-from-M}
\end{equation}

This procedure is derived assuming the existence of a solution $q(x,t)$ satisfying the initial and boundary conditions in addition to some other technical assumptions.  It produces the solution to the initial-boundary value problem under two conditions:
\begin{itemize}
\item The given boundary data $(Q^\mathrm{D},Q^\mathrm{N})$ used to compute the spectral transforms from \eqref{eq:X-problem}--\eqref{eq:conjugate-transforms} are consistent.  That is, $Q^\mathrm{N}(t)$ must agree with ($\eps$ times) the Neumann boundary value of the solution of the Dirichlet problem whose well-posedness was established by Carroll and Bu \cite{CarrollB91}.
\item The function $d(k)$ must have no zeros in the closed second quadrant of the complex $k$-plane.  This is a technical condition as otherwise Riemann-Hilbert Problem~\ref{rhp-original} must be formulated differently to allow $\mathbf{M}(k)$ to have poles at these points and their complex conjugates, with prescribed residue relations.  It is conjectured that in fact $d(k)$ is nonvanishing for consistent boundary data, but to our knowledge there is no proof\footnote{After this paper was accepted for publication, a preprint \cite{LenellsPreprint} was made public that evidently contains a proof of this conjecture.} of this in the literature.
\end{itemize}
Of course the problem is that if the boundary data functions $Q^\mathrm{D}$ and $Q^\mathrm{N}$ are both independently specified as is required to calculate the spectral transforms and hence the jump matrices, then the initial-boundary value problem is overdetermined and the solution of the equation produced by the method cannot generally satisfy the initial and boundary conditions (although it will solve the differential equation in the interior of the domain).  On the other hand, the procedure is sadly incomplete if only the Dirichlet data (the function $Q^\mathrm{D}$) is specified in which case the jump matrices \eqref{eq:exact-jump-positive}--\eqref{eq:exact-jump-negative} are indeterminate as $\Gamma(k)$ cannot be calculated at all from \eqref{eq:T-problem}--\eqref{eq:conjugate-transforms} and \eqref{gamma}.

A central role in the unified transform theory is therefore played by the \emph{global relation}, an identity satisfied by the spectral transforms of consistent boundary data $(Q^\mathrm{D},Q^\mathrm{N})$ that encodes  in the transform domain the Dirichlet-to-Neumann map giving $\eps q_x(0,t)$ in terms of $q(0,t)$ and $q(x,0)$.  Under certain conditions on the Dirichlet data, the global relation can be effectively solved, and hence the unknown Neumann data is eliminated.  The class of boundary conditions for which the global relation can be solved by symmetries in the complex $k$-plane is called the class of \emph{linearizable} boundary conditions.  Unfortunately, the only type of Dirichlet boundary condition known to be linearizable is the homogeneous boundary condition $Q^\mathrm{D}(t)\equiv 0$.  Of course, this special case can also be handled via the usual inverse scattering transform on the whole line $x\in\mathbb{R}$ simply by extending the initial data $q_0(x)$ to $x<0$ as an odd function.

Another approach to general nonhomogeneous Dirichlet boundary conditions that avoids the global relation entirely
may be
based on the observation that under mild conditions, given the spectral transforms $\{a(k),b(k),A(k),B(k)\}$,  Riemann-Hilbert Problem~\ref{rhp-original} 
has a unique solution for almost all $(x,t)\in\mathbb{R}^2$ by analytic Fredholm theory (see \cite[Proposition 4.3]{Zhou89}) combined with steepest descent asymptotics for large $x$ and $t$.  The exceptional set is the zero locus of an entire scalar function of $(x,t)\in\mathbb{C}^2$ that does not vanish identically, i.e., a complex curve in $\mathbb{C}^2$ that may or may not have real points but that would be at worst a closed and nowhere-dense union of analytic arcs in the real $(x,t)$-plane.
This in turn implies (by the standard arguments of the \emph{dressing method}, see also the proof of Proposition~\ref{prop:solution-exists} below) that 
for those $(x,t)$ for which a solution exists
the function $q(x,t)$ produced by taking the limit \eqref{eq:q-from-M} is necessarily \emph{some} solution of the defocusing nonlinear Schr\"odinger equation \eqref{eq:NLS}.  The question is whether this solution satisfies also the (three in total) initial and boundary conditions
that were used to generate the spectral transforms $\{a(k),b(k),A(k),B(k)\}$ in the first place.  This line of reasoning suggests an iteration procedure for solving the Dirichlet initial-boundary value problem for the defocusing nonlinear Schr\"odinger equation on the half-line with general nonhomogeneous data:  begin by making an initial guess for the (unknown) Neumann boundary data, say $Q^\mathrm{N}_0(t)$ for $t>0$.  Set $n=0$ and then:
\begin{enumerate}
\item[1.] Take $Q^\mathrm{N}(t)=Q^\mathrm{N}_{n}(t)$, and with the given Dirichlet data $Q^\mathrm{D}(t)$ and $q_0(x)$
calculate the spectral transforms $\{a(k),b(k),A(k),B(k)\}=\{a(k),b(k),A_{n}(k),B_{n}(k)\}$ from \eqref{eq:X-problem}--\eqref{eq:conjugate-transforms}.
\item[2.] Formulate Riemann-Hilbert Problem~\ref{rhp-original} with these spectral transforms and solve it.  Denote the function obtained from \eqref{eq:q-from-M} as $q_{n}(x,t)$.  It is necessarily a solution of the defocusing nonlinear Schr\"odinger equation \eqref{eq:NLS} on the quarter plane $x>0$ and $t>0$.
\item[3.] Define $Q^\mathrm{N}_{n+1}(t):=\eps \partial_x q_{n}(0,t)$ for $t>0$.
\item[4.] Set $n:=n+1$ and go to step 1.
\end{enumerate}
No doubt the reader can imagine various other iterative approaches like this one.  It is not the 
purpose of this paper to study the convergence of this algorithm, but in the spirit of the principle that finite truncations of a convergent iteration (or infinite series) can often provide accuracy in various asymptotic limits, we wish to explore
the possibility of using \emph{just one iteration} of the algorithm (actually a slightly modified version of the first iteration, see \S\ref{sec:Formulation} for details)
to provide an asymptotic approximation of the solution of the Dirichlet initial-boundary value problem in the \emph{semiclassical limit} $\epsilon\downarrow 0$.  The key to the success of this procedure is to make a very good initial guess for the unknown Neumann data, one that is asymptotically accurate in the semiclassical limit (as we will rigorously prove after the fact, see Theorem~\ref{theorem:boundary-condition-recover}).  That is, what we need is an explicit approximation of the Dirichlet-to-Neumann map for \eqref{eq:NLS}.

\subsection{The semiclassical Dirichlet-to-Neumann map}
Let us now explain the approximation of the Dirichlet-to-Neumann map for the defocusing nonlinear 
Schr\"odinger equation \eqref{eq:NLS} that we plan to study in this article.  Without loss of generality, we represent the complex field $q(x,t)$ in real phase-amplitude form:
\begin{equation}
q(x,t)=\eta(x,t)e^{\myi\sigma(x,t)/\eps},\quad \eta(x,t):=|q(x,t)|.
\label{eq:plane-wave-form}
\end{equation}
Substituting into \eqref{eq:NLS}, dividing by the common factor $e^{\myi\sigma(x,t)/\eps}$, and separating real and imaginary parts yields the following system of equations:
\begin{equation}
\begin{split}
\frac{\partial \eta}{\partial t} + 2\frac{\partial \sigma}{\partial x}\frac{\partial \eta}{\partial x} + \eta\frac{\partial^2\sigma}{\partial x^2}&=0\\
\frac{\partial \sigma}{\partial t} +\left(\frac{\partial\sigma}{\partial x}\right)^2+2\eta^2&=\frac{\eps^2}{\eta}\frac{\partial^2\eta}{\partial x^2}.
\end{split}
\label{eq:NLSWKB}
\end{equation}
This coupled system is equivalent to \eqref{eq:NLS}.  It is useful to intoduce notation for the phase gradient:
\begin{equation}
u(x,t):=\frac{\partial\sigma}{\partial x}(x,t).
\end{equation}
Now, in terms of $\eta$ and $u$, the exact ratio between the unknown Neumann data and the given Dirichlet data at $x=0$ takes the form
\begin{equation}
\begin{split}
-\myi\frac{Q^\mathrm{N}(t)}{Q^\mathrm{D}(t)}&=\frac{-\myi\eps}{q(0,t)}\frac{\partial q}{\partial x}(0,t)\\ {}&=u(0,t) -\frac{\myi\eps}{\eta(0,t)}
\frac{\partial\eta}{\partial x}(0,t).
\end{split}
\label{eq:NeumannWKB}
\end{equation}
Consider the second equation of the system \eqref{eq:NLSWKB} at $x=0$ along with  \eqref{eq:NeumannWKB} in the formal semiclassical limit $\epsilon\to 0$, assuming that $\eta(0,t)\neq 0$.  This means that we simply neglect the terms explicitly proportional to $\eps$ or $\eps^2$ in each case, yielding the formal approximations:
\begin{equation}
\frac{\partial\sigma}{\partial t}(0,t) + u(0,t)^2 + 2\eta(0,t)^2\approx 0\quad\text{and}\quad
-\myi\frac{Q^\mathrm{N}(t)}{Q^\mathrm{D}(t)}\approx u(0,t).
\label{eq:formalapproximations}
\end{equation}
Our approach is to assume that the known Dirichlet boundary data is specified in the form
\begin{equation}
Q^\mathrm{D}(t):=H(t)e^{\myi S(t)/\eps}
\label{eq:DirichletDataForm}
\end{equation}
with $H(\cdot)>0$ and $S(\cdot)$ are given real-valued functions independent of $\eps$.
Obviously we then have $\eta(0,t)=H(t)$ and $\sigma_t(0,t)=S'(t)$, so we may rewrite the approximate relations
\eqref{eq:formalapproximations} as
\begin{equation}
S'(t)+u(0,t)^2 + 2H(t)^2\approx 0\quad\text{and}\quad
-\myi\frac{Q^\mathrm{N}(t)}{Q^\mathrm{D}(t)}\approx u(0,t).
\label{eq:formalapproximationsII}
\end{equation}
Assuming further that
\begin{equation}
S'(t)<-2H(t)^2,\quad t>0,
\label{eq:Sprimebound}
\end{equation}
we solve the first of these relations for the unknown phase derivative $u(0,t)$ at the boundary:
\begin{equation}
u(0,t)\approx U(t):=\sqrt{-S'(t)-2H(t)^2}>0,
\label{eq:Udefine}
\end{equation}
that is, $U(t)$ is the (real valued) formal semiclassical approximation of the exact phase derivative $u(0,t)$.
Finally,
for Dirichlet boundary data \eqref{eq:DirichletDataForm} satisfying the condition \eqref{eq:Sprimebound} we use the second equation of \eqref{eq:formalapproximationsII} to 
approximate the Dirichlet-to-Neumann map as follows.
\begin{definition}
Suppose that the Dirichlet boundary data $Q^\mathrm{D}(t)$ of the form \eqref{eq:DirichletDataForm} satisfies $S'(t)+2H(t)^2<0$.  The \emph{semiclassical approximation of the Dirichlet-to-Neumann map} is defined by
\begin{equation}\label{5}
Q^\mathrm{N}_0(t):=\myi U(t)Q^\mathrm{D}(t),\quad t>0,
\end{equation}
where $U(t)$ is given in terms of the phase and amplitude of the known Dirichlet data by \eqref{eq:Udefine}.
\label{definition:D-to-N}
\end{definition}
The key point of our approach is that by neglecting the formally small dispersive terms in \eqref{eq:NLSWKB} we obtain a system that is first-order in $x$ and hence allows the unknown Neumann data to be \emph{explicitly eliminated} in favor of $t$-derivatives that may be computed along the boundary from the given Dirichlet data.  This approximation is a purely local relation between the two boundary values, and in particular the approximate Dirichlet-to-Neumann map is independent of initial data $q_0$.

We have selected the positive square root in \eqref{eq:Udefine} for a specific reason, which we now explain.
Differentiation of the second equation of \eqref{eq:NLSWKB} with respect to $x$ produces the equivalent system
\begin{equation}
\frac{\partial}{\partial t}\begin{bmatrix}\eta\\ u\end{bmatrix}+\mathbf{C}(\eta,u)
\frac{\partial}{\partial x}\begin{bmatrix}\eta\\ u\end{bmatrix} = \epsilon^2\frac{\partial}{\partial x}
\begin{bmatrix}0\\\eta^{-1}\eta_{xx}\end{bmatrix},\quad
\mathbf{C}(\eta,u):=\begin{bmatrix}2u &\eta\\4\eta & 2u\end{bmatrix}.
\label{eq:NLSWKB2}
\end{equation}
Obviously, \eqref{eq:NLSWKB2} is a formally small perturbation of a quasilinear system obtained by simply setting $\epsilon$ to zero.
The \emph{characteristic velocities} of the limiting system are the eigenvalues $c(\eta,u)$ of the coefficient matrix  $\mathbf{C}(\eta,u)$:
\begin{equation}
c(\eta,u):=2u\pm 2\eta.
\end{equation}
As the characteristic velocities are real and distinct (for $\eta\neq 0$), the limiting quasilinear system is of hyperbolic type.
Causality and local well-posedness for the Cauchy problem of the hyperbolic approximating system in the quarter plane $x>0$ and $t>0$ requires that the boundary $x=0$ be a space-like curve.  In other words, we require \emph{both} characteristic velocities to be strictly positive at the boundary.  This means that we will require that $U(t)>H(t)$ for all $t>0$.  
Since $H(t)\ge 0$ for $t>0$ it is clear that well-posedness of the limiting hyperbolic boundary-value problem requires in particular $U(t)>0$.  In fact, we will ensure the condition $U(t)>H(t)$ by imposing the stronger condition $U(t)>2H(t)$; the latter condition appears to be necessary to recover the Dirichlet boundary data at $x=0$ for all $t>0$ (see Remark~\ref{remark:kb-negative} below).

\subsection{Outline of the paper.  Description of main results}
For convenience we restrict our attention to the already nontrivial and physically interesting case of zero initial data:  $q_0(x)=0$ for all $x>0$.  However, we fix rather general nonhomogeneous Dirichlet boundary data of the form \eqref{eq:DirichletDataForm} for $t>0$ (and satisfying several additional conditions allowing our procedure to succeed, see Assumption~\ref{assumption:data} below), and attempt to solve the corresponding initial-boundary value problem.  The first step is the calculation of the spectral transforms $A_0(k)$ and $B_0(k)$ corresponding to the Dirichlet data $Q^\mathrm{D}(t)$ given by \eqref{eq:DirichletDataForm} and the formally approximate Neumann data $Q^\mathrm{N}_0(t)$ given by Definition~\ref{definition:D-to-N}.  The direct spectral analysis is made possible in practice because the parameter $\eps>0$ is presumed small, so the equations of the Lax pair become singularly perturbed differential equations that may be studied by classical methods.  The results of this analysis are summarized in \S\ref{sec:spectral-analysis}, with the corresponding proofs appearing in two appendices.  Despite the rigor of these results, there are certain difficulties that remain with directly formulating the inverse problem for
the exact scattering data, so
rather than calculate the solution of the defocusing nonlinear Schr\"odinger equation corresponding to the exact spectral transforms of the (generally incompatible) Dirichlet-Neumann pair $(Q^\mathrm{D},Q^\mathrm{N}_0)$ we modify the spectral functions in an ad-hoc fashion, but one inspired by the rigorous direct spectral analysis of the temporal problem of the Lax pair.  This allows us to formulate a simpler and completely explicit version of  Riemann-Hilbert Problem~\ref{rhp-original} for a matrix $\tilde{\mathbf{M}}(k)$; see \S\ref{sec:Formulation}.

The simpler Riemann-Hilbert problem explicitly encodes the given Dirichlet boundary data \eqref{eq:DirichletDataForm} through two integral transforms denoted $\tau$ and $\Phi$ (these are really the semiclassical analogues of the amplitude and phase of the spectral function $\Gamma$; see \eqref{eq:tau-define-1}--\eqref{eq:Phi-define}), and its solution produces, for each $\eps>0$, a solution $q=\tilde{q}^\eps(x,t)$ of the defocusing nonlinear Schr\"odinger equation \eqref{eq:NLS}.  The rest of the paper is concerned with analyzing this solution, paying particular attention to the semiclassical asymptotic behavior of $\tilde{q}^\eps(x,t)$  at the initial time $t=0$ for $x>0$ and at the boundary $x=0$ for $t>0$.  Our first result
is the following.
\begin{theorem}[approximation of the initial condition]
The solution $q=\tilde{q}^\eps(x,t)$ of the defocusing nonlinear Schr\"odinger equation \eqref{eq:NLS} obtained from Riemann-Hilbert problem~\ref{rhp-M-tilde} satisfies
\begin{equation}
\tilde{q}^\eps(x,0)=\mathcal{O}((\log(\eps^{-1}))^{-1/2}) \quad x>0,
\end{equation}
where the error term is uniform on $x\ge x_0$ for each $x_0>0$.
\label{theorem:initial-condition-small}
\end{theorem}
Thus the function $\tilde{q}^\eps(x,t)$ nearly satisfies the given homogeneous initial condition $q_0(x)= 0$ for $x>0$.  Our rigorous proof of this result is based on the 
steepest descent method for Riemann-Hilbert problems, combined with a natural generalization  of the method involving $\dbar$-problems \cite{McLaughlinM08}.  After establishing some preliminary results in \S\ref{sec:preliminary-results}, we give the proof of Theorem~\ref{theorem:initial-condition-small} in \S\ref{sec:initial-condition-small-proof}.

Our next main result is the following.  The points $t_\mathfrak{a}$ and $t_\mathfrak{b}$ are defined as part of Assumption~\ref{assumption:data} below.
\begin{theorem}[approximation of boundary conditions]
Suppose that $t>0$ and $t\neq t_\mathfrak{a}$, $t\neq t_\mathfrak{b}$.  The solution $q=\tilde{q}^\eps(x,t)$ of the defocusing nonlinear Schr\"odinger equation \eqref{eq:NLS} obtained from Riemann-Hilbert problem~\ref{rhp-M-tilde} satisfies
\begin{equation}
\tilde{q}^\eps(0,t)=H(t)e^{iS(t)/\eps}+\mathcal{O}((\log(\eps^{-1}))^{-1/2}),
\label{eq:Dirichlet-Data-Recover}
\end{equation}
and
\begin{equation}
\eps\tilde{q}^\eps_x(0,t)=\myi U(t)H(t)e^{\myi S(t)/\eps} +\mathcal{O}((\log(\eps^{-1}))^{-1/2}),
\label{eq:Neumann-Data-Recover}
\end{equation}
where the error terms are uniform for $t$ in compact subintervals of $(0,+\infty)\setminus\{t_\mathfrak{a},t_\mathfrak{b}\}$.
\label{theorem:boundary-condition-recover}
\end{theorem}
Equation \eqref{eq:Dirichlet-Data-Recover} shows that the same solution $\tilde{q}^\eps(x,t)$ very nearly satisfies the given nonhomogeneous Dirichlet boundary condition \eqref{eq:DirichletDataForm} at $x=0$ for $t>0$.  Moreover, from \eqref{eq:Neumann-Data-Recover}  we see directly that the true
Neumann data at the boundary is indeed asymptotically consistent with the formal approximation given by Definition~\ref{definition:D-to-N}.  The proof of this result is again based on the steepest descent method, this time augmented with the use of a complex phase function $g$.  After describing the general methodology and constructing the function $g$ in \S\ref{sec:g-function}, we present the proof of Theorem~\ref{theorem:boundary-condition-recover} in \S\ref{sec:boundary-condition-recover-proof}.

The main point, however, is that the solution $\tilde{q}^\eps(x,t)$ is represented also for $(x,t)$ not on the boundary of the quarter plane $x>0$, $t>0$ via exactly the same Riemann-Hilbert Problem (see Riemann-Hilbert Problem~\ref{rhp-M-tilde}).  This means that one may use steepest descent methods to calculate $\tilde{q}^\eps(x,t)$ for small $\eps>0$ for positive $t$ and away from the boundary.  For $(x,t)$ close to the boundary of the quarter plane the analysis is virtually the same as it is exactly on the boundary, with similar results.  For example, a corollary of the proof of Theorem~\ref{theorem:initial-condition-small} is the following.  The points $k_\mathfrak{a}$ and $k_\mathfrak{b}$ are specified in terms of the functions characterized by Assumption~\ref{assumption:data} below, and $\Phi$ is explicitly given by \eqref{eq:Phi-define}.
\begin{corollary}[existence of a vacuum domain]
Let $t\ge 0$, and let $X(t)$ denote the smallest nonnegative value of $x_0$ for which the inequality $x+4tk-\Phi'(k)\ge 0$ holds for all $k\in (k_\mathfrak{a},k_\mathfrak{b})$ whenever $x\ge x_0$.  Then the solution $q=\tilde{q}^\eps(x,t)$ of the defocusing nonlinear Schr\"odinger equation \eqref{eq:NLS} obtained from Riemann-Hilbert Problem~\ref{rhp-M-tilde} satisfies $\tilde{q}^\eps(x,t)=\mathcal{O}((\log(\eps^{-1}))^{-1/2})$ as $\eps\downarrow 0$ whenever $x>X(t)$.
\label{corollary-vacuum}
\end{corollary}
In the case that $f(\cdot):=-\Phi'(\cdot)$ is convex, we may characterize $X(t)$ explicitly in terms of the Legendre dual $f^*$ as follows:
\begin{equation}
X(t):=f^*(-4t)=[-\Phi']^*(-4t),\quad t>0,\quad\text{where}\quad f^*(p):=\sup_{k_\mathfrak{a}<k<k_\mathfrak{b}}(pk-f(k)).
\label{eq:Legendre-dual}
\end{equation}
The proof of Corollary~\ref{corollary-vacuum} is given in \S\ref{sec:vacuum-domain}.
We call the domain $x>X(t)$ the \emph{vacuum domain} corresponding to the Dirichlet boundary data $H(t)e^{iS(t)/\eps}$.  In the vacuum domain the solution is influenced predominantly by the homogeneous initial data rather than the nonhomogeneous boundary data in the semiclassical limit.
A concrete calculation of the vacuum domain for a particular choice of Dirichlet boundary data is shown in Figure~\ref{fig:VacuumDomain}.
\begin{figure}[h]
\includegraphics{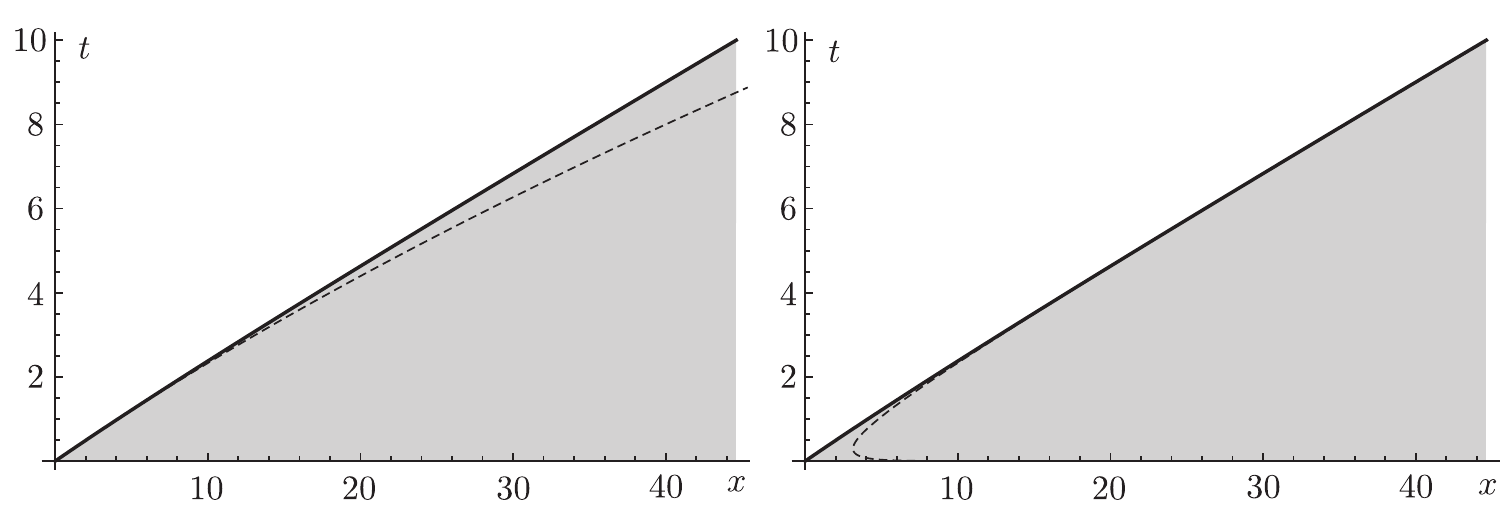}
\caption{The curve $x=X(t)$ (solid black curve) and the $\eps$-independent vacuum domain (shaded) calculated from $\Phi$ corresponding to the explicit boundary data illustrated in Figure~\ref{fig:Assumption1} below.   Also shown (dashed curves) are explicit asymptotes to $x=X(t)$ for small $t$ (left panel, the asymptote $x=X_0(t)$) and large $t$ (right panel, the asymptote $x=X_\infty(t)$).  These asymptotes are described in \S\ref{sec:vacuum-domain} (see \eqref{eq:X-asymptote-t-small} and \eqref{eq:X-asymptote-t-large}).}
\label{fig:VacuumDomain}
\end{figure}

Another result is the following, which is essentially a corollary of the proof of Theorem~\ref{theorem:boundary-condition-recover}.  Here $\mathfrak{a}(\cdot)$ and $\mathfrak{b}(\cdot)$ are defined in terms of the Dirichlet data by \eqref{eq:a-b-define}.
\begin{corollary}[existence of a plane-wave domain]
Each point $(0,t_0)$ with $t_0\in (0,+\infty)\setminus\{t_\mathfrak{a},t_\mathfrak{b}\}$ has a neighborhood $D_{t_0}$ in the $(x,t)$-plane in which there exist unique differentiable functions $\alpha=\alpha(x,t)$ and $\beta=\beta(x,t)$ satisfying $\alpha(0,t)=\mathfrak{a}(t)$ and $\beta(0,t)=\mathfrak{b}(t)$ and the partial differential equations
\begin{equation}
\begin{split}
\frac{\partial\alpha}{\partial t}-(3\alpha+\beta)\frac{\partial\alpha}{\partial x}&=0\\
\frac{\partial\beta}{\partial t}-(\alpha+3\beta)\frac{\partial\beta}{\partial x}&=0,
\end{split}
\label{eq:Whitham-1}
\end{equation}
and such that the solution $q=\tilde{q}^\eps(x,t)$ of the defocusing nonlinear Schr\"odinger equation \eqref{eq:NLS} obtained from Riemann-Hilbert Problem~\ref{rhp-M-tilde} satisfies
\begin{equation}
\tilde{q}^\eps(x,t)=\eta(x,t)e^{\myi\sigma(x,t)/\eps}+\mathcal{O}((\log(\eps^{-1}))^{-1/2})
\label{eq:plane-wave}
\end{equation}
uniformly for $(x,t)\in D_{t_0}$ as $\eps\downarrow 0$, where
\begin{equation}
\eta(x,t):=\frac{1}{2}(\beta(x,t)-\alpha(x,t))\quad\text{and}\quad
u(x,t):=\frac{\partial\sigma}{\partial x}(x,t)=-(\alpha(x,t)+\beta(x,t))
\label{eq:eta-u}
\end{equation}
and $\sigma(0,t)=S(t)$.
\label{corollary:plane-wave}
\end{corollary}
The proof of Corollary~\ref{corollary:plane-wave} is given in \S\ref{sec:plane-wave}.  Note that eliminating $\alpha$ and $\beta$ from \eqref{eq:Whitham-1} in favor of $\eta$ and $u$
using \eqref{eq:eta-u} yields 
\begin{equation}
\frac{\partial}{\partial t} \begin{bmatrix}\eta\\u\end{bmatrix}+\mathbf{C}(\eta,u)
\frac{\partial}{\partial x}\begin{bmatrix}\eta\\u\end{bmatrix}=\begin{bmatrix}0\\0\end{bmatrix},\quad \mathbf{C}(\eta,u):=
\begin{bmatrix}2u & \eta\\4\eta & 2u\end{bmatrix},
\label{eq:dispersion-less-NLS}
\end{equation}
which should be compared with \eqref{eq:NLSWKB2}, the defocusing nonlinear Schr\"odinger equation written without approximation in terms of amplitude $\eta$ and phase derivative $u$.
Therefore, we observe that for small positive $x$, $\tilde{q}^\eps(x,t)$ resembles a modulated plane wave of the form \eqref{eq:plane-wave-form} for amplitude $\eta$ and phase $\sigma$ independent of $\eps$, and the modulation is described by the dispersionless nonlinear Schr\"odinger system \eqref{eq:dispersion-less-NLS}, or equivalently the Whitham (Riemann-invariant form) system 
\eqref{eq:Whitham-1}.  This shows consistency with, and adds yet more weight to, our approximate formula for the Dirichlet-to-Neumann map given in Definition~\ref{definition:D-to-N}.  Indeed, the latter was formally derived under the \emph{initially unjustified assumption} that the solution resembles a modulated plane wave near the boundary $x=0$.

We call the union of the neighborhoods of the $(x,t)$ plane for $x>0$ in which $\tilde{q}^\eps(x,t)$ is described by Corollary~\ref{corollary:plane-wave} the \emph{plane-wave domain} for $\tilde{q}^\eps(x,t)$.  We therefore see that the quarter-plane $x>0$ and $t>0$ is split up into several regions in which the approximate solution $\tilde{q}^\eps(x,t)$ of the Dirichlet boundary-value problem\footnote{We wish to stress that while $q=\tilde{q}^\eps(x,t)$ only approximately satisfies the given initial and boundary conditions, it is an exact solution of the defocusing nonlinear Schr\"odinger equation \eqref{eq:NLS} for every $\eps>0$.} behaves quite differently.  So far we have observed the vacuum domain, in which $\tilde{q}^\eps(x,t)$ simply decays to zero with $\eps>0$, and the plane-wave domain, in which $\tilde{q}^\eps(x,t)$ resembles a modulated plane wave with nonzero amplitude.  It is to be expected that these two domains do not exhaust the quarter plane.  While we do not pursue the topic further in this paper, the methodology presented in \S\ref{sec:g-function} below also allows one to calculate the semiclassical behavior of $\tilde{q}^\eps(x,t)$ for $(x,t)$ in domains not contiguous to the boundary of the quarter plane, in which (in principle) more complicated local behavior of $\tilde{q}^\eps(x,t)$ can occur, with microstructure modeled by higher transcendental functions (e.g., dispersive shock waves described by modulated elliptic functions).
See Remark~\ref{remark-multiphase} for more information.

With the proofs of our results complete, we conclude the body of our paper with some further remarks, some indicating directions for future work, in \S\ref{sec:concluding-remarks}.

\subsection{Related work}
Our paper represents a further contribution to the literature on the use of the unified transform to study nonlinearizable boundary value problems in various asymptotic limits.  A key observation that was made fairly early in the development of the theory was that regardless of whether the spectral functions $A(k)$ and $B(k)$ actually correspond to a compatible Dirichlet-Neumann pair $(Q^\mathrm{D},Q^\mathrm{N})$,  Riemann-Hilbert Problem~\ref{rhp-original} yields to asymptotic analysis in the limit of large $t$ (with $x=vt$ for some nonnegative velocity $v$) by the steepest descent method.  
General asymptotic properties of the solution that can be observed by such analysis therefore necessarily also describe the physical solutions of the Dirichlet problem that simply correspond to the special case that the spectral functions satisfy the global relation.  As a representative of this type of analysis (for the focusing case of the nonlinear Schr\"odinger equation), we cite a paper of Boutet de Monvel, Its, and Kotlyarov \cite{BoutetIK07}, where time-periodic boundary conditions are analyzed.  Such analysis does not require any preliminary asymptotic analysis of the spectral functions, as they are independent of the asymptotic parameter $t$.  Another approach to the asymptotic solution of nonlinearizable boundary value problems is to consider the situation in which the initial and boundary data are small, in which case a perturbation scheme based on the amplitude as a small parameter can be developed in detail, and significantly this allows the global relation to be solved order-by-order.  This means that the asymptotic results obtained are guaranteed to correspond to a compatible Dirichlet-Neumann pair $(Q^\mathrm{D},Q^\mathrm{N})$ even though only $Q^\mathrm{D}$ is given.  The recent papers of Fokas and Lenells \cite{FokasL12a,FokasL12b}
pursue this approach and obtain new convergence results showing that for small simple harmonic Dirichlet boundary data, the solution is eventually periodic with the same period, at least to third order in the small amplitude.  

The semiclassical limit is, in a sense, the exact opposite to the weakly-nonlinear small-amplitude limit.  Indeed, the formal semiclassical limit is given by the strongly nonlinear hyperbolic system \eqref{eq:dispersion-less-NLS}.
There is at least one other paper in the literature on the subject of semiclassical analysis of the Dirichlet initial-boundary-value problem on the half-line for the defocusing nonlinear Schr\"odinger equation, namely a paper of Kamvissis \cite{Kamvissis03}, which directly stimulated our interest in this problem.  Like we do, Kamvissis  considers general nonhomogeneous Dirichlet boundary data together with homogeneous initial conditions, and
he applies the steepest descent methodology for Riemann-Hilbert problems to deduce general properties of the solution in the semiclassical limit.  Our Corollary~\ref{corollary:plane-wave} is consistent with Theorem~5 of \cite{Kamvissis03} (the main result of that paper) albeit in the simplest case of genus $N=0$.
On the other hand, it is less clear whether the vacuum domain $x>X(t)$ described by our Corollary~\ref{corollary-vacuum} is a special case of Kamvissis' Theorem~5. 

While we study the same problem, and apply similar methods, the approach in \cite{Kamvissis03} is fundamentally different from ours, being based solely on the abstract existence result for the unknown Neumann data $Q^\mathrm{N}(t)$ corresponding to the given Dirichlet data $Q^\mathrm{D}(t)$.  While Kamvissis'  assumption that $Q^\mathrm{D}(t)$ is independent of $\eps$ is quite reasonable and physically interesting\footnote{In the setting of \eqref{eq:DirichletDataForm}, Dirichlet boundary data that is independent of $\eps$ corresponds to taking $S(t)\equiv 0$.  Therefore, in a sense our results cannot be compared well with those of \cite{Kamvissis03}, because we require $S'(t)$ to be strictly negative (see \eqref{eq:Sprimebound}).}, his 
subsequent analysis of the direct spectral problem for the $t$-part of the Lax pair (Theorems~2 and 3 of \cite{Kamvissis03}, of which our Propositions~\ref{prop:noturningpoints} and \ref{prop:two-turning-points} are analogues) apparently rests upon the additional hidden assumption that the implicitly-defined function $q_x(0,t)$ is \emph{also} independent of $\eps$; otherwise the WKB methodology cited in \cite[Section III]{Kamvissis03} does not apply.  Since the defocusing nonlinear Schr\"odinger equation involves the parameter $\eps$ in a singular way, whether this assumption is justified is certainly not 
obvious.  Indeed one might worry that a slowly-varying Dirichlet boundary condition might give rise to a Neumann boundary value with rapid variations in amplitude or phase of period proportional to $\eps$.  For example, our Theorem~\ref{theorem:boundary-condition-recover} shows that some bounded Dirichlet data $q(0,t)$ can lead to rapidly oscillatory Neumann data $q_x(0,t)$ that moreover is large of size $\eps^{-1}$.

Our approach is to avoid abstract assumptions, and instead make a very explicit assumption, based on the modulated plane-wave ansatz, concerning the unknown Neumann data as described in Definition~\ref{definition:D-to-N}.  This allows us to justify our steepest descent analysis by ultimately tying the solution generated back to the hypothesized initial and boundary data (Theorems~\ref{theorem:initial-condition-small} and \ref{theorem:boundary-condition-recover}) in an explicit fashion.  This same approach leads to 
a very concrete description of the semiclassical dynamics of the solution in the full domain $x>0$ and $t>0$, as in the characterization of the vacuum domain presented in Corollary~\ref{corollary-vacuum}.

Another paper that we wish to mention is work of Degasperis, Manakov, and Santini \cite{DegasperisMS02} that presents an alternate approach to the general initial-boundary value problem for the defocusing nonlinear Schr\"odinger equation.  The method described in \cite{DegasperisMS02}  avoids using the $t$-part of the Lax pair to formulate the inverse problem and instead uses the inverse theory of the spatial part of the Lax pair only, at the cost of a more implicit nonlinear description of the time evolution of the jump matrices on the real line.  The fact that the inverse problem is ultimately formulated as a Riemann-Hilbert problem relative to the real axis may be a crucial benefit in our view (see Remark~\ref{remark-imaginary-amplify}).  In the future, we plan to explore the possibility of using semiclassical asymptotic techniques to analyze this alternate method of studying initial-boundary value problems.

\section{A Class of Dirichlet Boundary-Value Problems}
\subsection{Characterization of the boundary data}
For simplicity\footnote{See Remark~\ref{remark:nonzero-IC}.},
we consider the case of vanishing initial data:
\begin{equation}
q_0(x)=0,\quad x>0.
\end{equation}
This immediately implies that the spectral transforms defined from the differential equation \eqref{eq:X-problem} satisfy $a(k)\equiv 1$ and $b(k)\equiv 0$ for all $k\in\mathbb{C}$.
We take the Dirichlet boundary data in the form \eqref{eq:DirichletDataForm}, and for convenience
we impose several conditions on the functions $H(t)$ and $S(t)$ for $t\ge 0$.  These are specified in terms of an auxiliary function $U$ as follows:
\begin{assumption}
The functions $H:\mathbb{R}_+\to\mathbb{R}$ and $U:\mathbb{R}_+\to\mathbb{R}$ satisfy the following conditions:
\begin{itemize}
\item
$H:\mathbb{R}_+\to\mathbb{R}$ is real analytic for $t>0$, strictly positive for all $t>0$, and $t^pH^{(q)}(t)\to 0$ as $t\to +\infty$ for all $p\ge 0$ and $q=0,1,2,\dots$.  Also,
there is a positive number $h_0$ such that $H(t)=h_0t^{1/2}(1+o(1))$ and $H'(t)=\tfrac{1}{2}h_0t^{-1/2}(1+o(1))$ hold as $t\to 0$ with $\Re\{t\}\ge 0$.
\item 
$U:\mathbb{R}_+\to\mathbb{R}$ is real analytic for $t>0$, satisfying $U(t)\ge 2H(t)+\delta$ for some $\delta>0$,
and $t^pU^{(q)}(t)\to 0$ as $t\to +\infty$ for all $p\ge 0$ and $q=1,2,3,\dots$.  Also, there is a positive number $U_0$ such that $U(t)=U_0+o(t^{1/2})$ and $U'(t)=\mathcal{O}(t^{-1/2})$ hold as $t\to 0$ with $\Re\{t\}\ge 0$.
\item The functions 
\begin{equation}
\mathfrak{a}(t):=-\tfrac{1}{2}U(t)-H(t)\quad \text{and}\quad \mathfrak{b}(t):=-\tfrac{1}{2}U(t)+H(t)
\label{eq:a-b-define}
\end{equation}
each have precisely one critical point in $(0,\infty)$, corresponding to a nondegenerate maximum for $\mathfrak{b}$ at a point $t=t_\mathfrak{b}$ and a nondegenerate minimum for $\mathfrak{a}$ at a point $t=t_\mathfrak{a}$.  Nondegeneracy means that $\mathfrak{a}''(t_\mathfrak{a})>0$ and $\mathfrak{b}''(t_\mathfrak{b})<0$.
\end{itemize}
\label{assumption:data}
\end{assumption}

\begin{remark}
The square-root behavior of the amplitude $H(t)$ that is specified in Assumption~\ref{assumption:data} evidently violates the conditions for the proof of Carroll and Bu \cite{CarrollB91} to guarantee the existence of a solution of the initial-boundary value problem.
Nonetheless this behavior leads to additional smoothness of the integral transform $\Phi$ defined in \eqref{eq:Phi-define} below that is useful in the proof of Theorem~\ref{theorem:initial-condition-small}.  See Remark~\ref{remark:Phi-smooth}.  

We may avoid this difficulty as follows.  Let $\mathcal{B}(t)$ be a $C^\infty(\mathbb{R}_+)$ ``bump function'' satisfying $\mathcal{B}(t)=0$ for $0<t<1$ and $\mathcal{B}(t)=1$ for $t>2$.  Replacing
$H(t)$ by $H^\eps(t):=\mathcal{B}(\eps^{-1}t)H(t)$, by \cite{CarrollB91} there is a unique solution $q=q^\eps(x,t)$ of \eqref{eq:NLS} for each $\eps>0$ satisfying $q^\eps(x,0)=0$ for $x>0$ and $q^\eps(0,t)=H^\eps(t)e^{\myi S(t)/\eps}$ for $t>0$.  We may view our results as a comparison between $\tilde{q}^\eps(x,t)$ and the function $q^\eps(x,t)$, the latter of which exactly satisfies the given boundary condition \eqref{eq:DirichletDataForm} for every $t>0$ as long as $\eps>0$ is sufficiently small (given $t$).
\myendrmk
\label{remark:amplitude-fix}
\end{remark}

We now use \eqref{eq:Udefine} to define $S(t)$ in terms of functions $H$ and $U$ satisfying the conditions of Assumption~\ref{assumption:data} as
\begin{equation}
S(t):=S(0)-\int_0^t\left[U(s)^2+2H(s)^2\right]\,ds.
\label{eq:SprimeUH}
\end{equation}
Note that as $U$ is real, the inequality \eqref{eq:Sprimebound} is automatically satisfied.
We introduce the following notation:  
\begin{equation}
k_0:=-\tfrac{1}{2}U_0,\quad
U_\infty:=\lim_{t\to +\infty}U(t),\quad\text{and}\quad k_\infty:=-\tfrac{1}{2}U_\infty
\end{equation}
($U_\infty$ is well-defined as $U'(\cdot)\in L^1(\mathbb{R}_+)$), and we set $k_\mathfrak{a}:=\mathfrak{a}(t_\mathfrak{a})$ and $k_\mathfrak{b}:=\mathfrak{b}(t_\mathfrak{b})$.  Note that the assumption $H(t)>0$ guarantees that $\mathfrak{a}(t)<\mathfrak{b}(t)$, and the assumption that $U(t)\ge 2H(t)+\delta$ guarantees that $k_\mathfrak{a}<k_\mathfrak{b}<0$.  
The points $k_0$ and $k_\infty$ lie in the interval $(k_\mathfrak{a},k_\mathfrak{b})$, and we see that 
$\mathfrak{a}(0)=\mathfrak{b}(0)=k_0$ while $\lim_{t\to+\infty}\mathfrak{a}(t)=\lim_{t\to +\infty}\mathfrak{b}(t)=k_\infty$.
These definitions are illustrated for boundary data satisfying Assumption~\ref{assumption:data} in Figure~\ref{fig:Assumption1}.
\begin{figure}[h]
\includegraphics{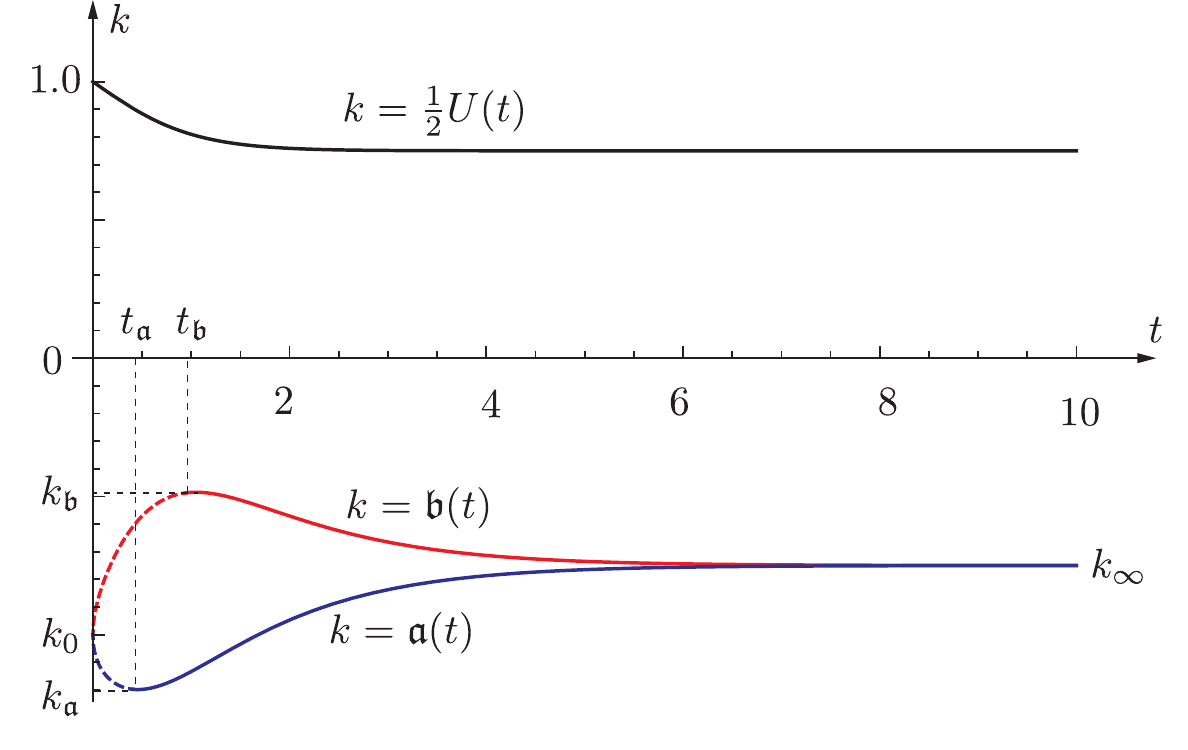}
\caption{The turning point curve consists of three branches, curves along which $\lambda^2=0$  (see \eqref{eq:lambdasquared}).  Here the turning point curve is shown for boundary data $H(t):=\tfrac{1}{2}t^{1/2}\,\mathrm{sech}(t)$ and $U(t):=2-\tfrac{1}{2}\tanh(t)$ consistent with Assumption~\ref{assumption:data}.  The functions $\mathfrak{a}(\cdot)$ and $\mathfrak{b}(\cdot)$ are given in terms of $H(\cdot)$ and $U(\cdot)$ by \eqref{eq:a-b-define}.  The dashed (solid) curves in the interval $k_\mathfrak{a}<k<k_\mathfrak{b}$ correspond to the function $t=t_-(k)$ (the function $t=t_+(k)$).}
\label{fig:Assumption1}
\end{figure}

With the semiclassical approximation of the Dirichlet-to-Neumann map given in Definition~\ref{definition:D-to-N}, the direct scattering problem encoding the boundary data is
\begin{equation}
\begin{split}
\eps\frac{d\mathbf{T}_0}{dt}(t;k)&=
\begin{bmatrix}
-2\myi k^2-\myi|Q^\mathrm{D}(t)|^2 & 2kQ^\mathrm{D}(t)+\myi Q^\mathrm{N}_0(t)\\
2kQ^\mathrm{D}(t)^* -\myi Q^\mathrm{N}_0(t)^* & 2\myi k^2+\myi |Q^\mathrm{D}(t)|^2\end{bmatrix}\mathbf{T}_0(t;k)\\
{}&=
\begin{bmatrix}
-2\myi k^2-\myi H(t)^2 & (2k-U(t))H(t)e^{\myi S(t)/\eps}\\
(2k-U(t))H(t)e^{-\myi S(t)/\eps} & 2\myi k^2+\myi H(t)^2\end{bmatrix}\mathbf{T}_0(t;k),\\
&\qquad\qquad
\lim_{t\to +\infty}\mathbf{T}_0(t;k)e^{2\myi k^2t\sigma_3/\eps}=\mathbb{I}.
\end{split}
\label{eq:T-equation-rewrite}
\end{equation}
The oscillatory factors $e^{\pm \myi S(t)/\eps}$ can be removed from the coefficient matrix by means of a simple substitution:
\begin{equation}
\mathbf{T}_0(t;k)=e^{\myi S(t)\sigma_3/(2\eps)}\mathbf{F}(t;k).
\label{eq:T-and-F}
\end{equation}
Indeed, making use of  \eqref{eq:SprimeUH}, this substitution 
leads to the equivalent system of equations
\begin{equation}
\eps\frac{d\mathbf{F}}{dt}(t;k)=\mathbf{B}(t;k)\mathbf{F}(t;k),
\label{eq:Fsystem}
\end{equation}
with $\epsilon$-independent coefficient matrix given by
\begin{equation}
\mathbf{B}(t;k):=\frac{1}{2}\begin{bmatrix}
-4\myi k^2+\myi U(t)^2 & 2H(t)(2k-U(t))\\
2H(t)(2k-U(t)) & 4\myi k^2-\myi U(t)^2\end{bmatrix}=\frac{1}{2}(2k-U(t))\begin{bmatrix}-\myi (2k+U(t)) & 2H(t)\\2H(t) & \myi (2k+U(t))\end{bmatrix},
\label{eq:B-matrix}
\end{equation}
that we need to solve 
subject to the boundary condition
\begin{equation}
\lim_{t\to +\infty}\mathbf{F}(t;k)e^{\myi(4k^2t+S(t))\sigma_3/(2\eps)}=\mathbb{I}.
\label{eq:Fnorm}
\end{equation}
The corresponding spectral transforms are given for $\Im\{k^2\}\le 0$ by
\begin{equation}
A_0(k^*)^*:=T_{0,11}(0;k)=e^{\myi S(0)/(2\eps)}f_1(0;k)
\label{eq:AstarF}
\end{equation}
and
\begin{equation}
B_0(k^*)^*:=T_{0,21}(0;k)=e^{-\myi S(0)/(2\eps)}f_2(0;k),
\label{eq:BstarF}
\end{equation}
where $\mathbf{f}(t;k)=(f_1(t;k),f_2(t;k))^\mathsf{T}:=(F_{11}(t;k),F_{21}(t;k))^\mathsf{T}$ denotes the first column of $\mathbf{F}(t;k)$.  

\subsection{Semiclassical behavior of the spectral functions $A_0(k)$ and $B_0(k)$}
\label{sec:spectral-analysis}
Since $\eps>0$ appears both in the data $(Q^\mathrm{D},Q^\mathrm{N}_0)$ and in the differential equation \eqref{eq:T-equation-rewrite}, the spectral functions $A_0(\cdot)$ and $B_0(\cdot)$ will also
depend on this small parameter.  We now study this dependence rigorously in the limit $\eps\downarrow 0$.

Given any sufficiently small number $\delta>0$ (not necessarily related to the constant in Assumption~\ref{assumption:data}) we define $\mathcal{Q}_\delta^\mathrm{II}$ to be the closed unbounded subset of the $k$-plane characterized by the inequalities $\pi/2\le\arg(k)\le\pi$ and one of the three inequalities:  $\Re\{k\}\le k_\mathfrak{a}-\delta$ or $\Re\{k\}\ge k_\mathfrak{b}+\delta$ or $\Im\{k\}\ge\delta$.  See Figure~\ref{fig:QIIdelta}.
\begin{figure}[h]
\begin{center}
\includegraphics{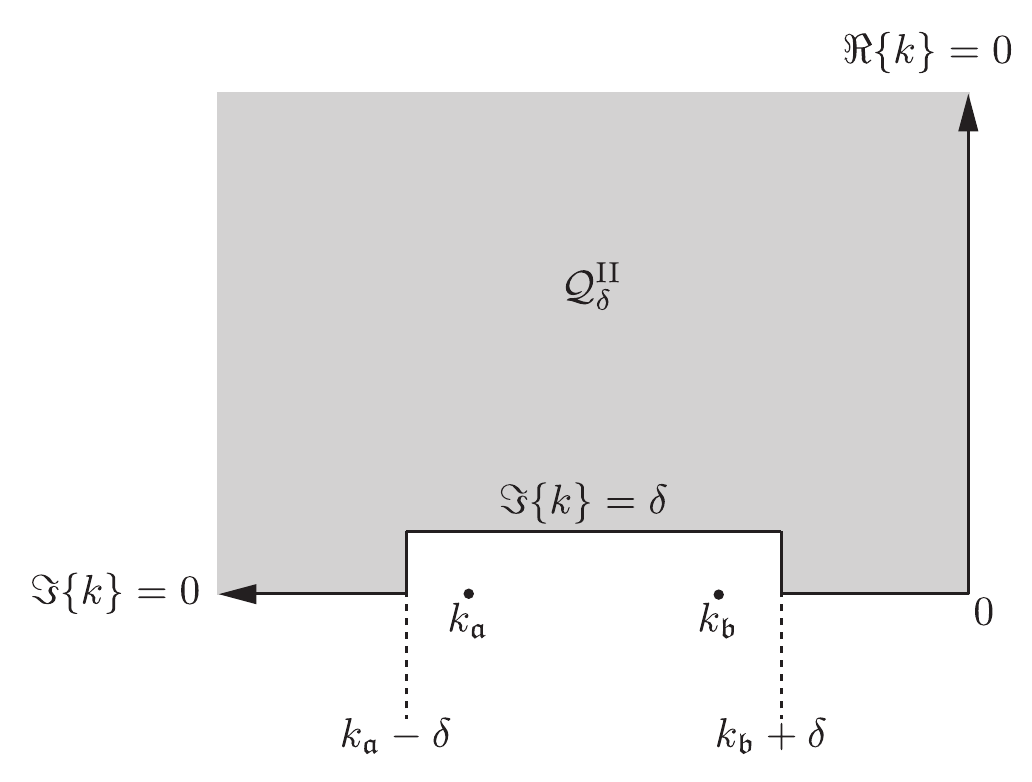}
\end{center}
\caption{The closed unbounded subset $\mathcal{Q}^\mathrm{II}_\delta$ of the second quadrant of the complex $k$-plane.}
\label{fig:QIIdelta}
\end{figure}

The eigenvalues $\lambda$ of $\mathbf{B}(t;k)$ satisfy
\begin{equation}
\lambda^2 = \frac{1}{4}(2k-U(t))^2\left[4H(t)^2-(2k+U(t))^2\right]=(2k-U(t))^2(k-\mathfrak{a}(t))(\mathfrak{b}(t)-k).
\label{eq:lambdasquared}
\end{equation}
Given $k$ with $\Im\{k^2\}\le 0$, a positive real number $t>0$ is called a \emph{turning point} for \eqref{eq:Fsystem} if the two eigenvalues of $\mathbf{B}(t;k)$ degenerate (at $\lambda=0$).
We have the following basic fact.
\begin{lemma}[Existence of turning points]
\label{lem:turningpoints}
Suppose that Assumption~\ref{assumption:data} holds and that $\pi/2\le\arg(k)\le\pi$.  Then there exist turning points $t>0$ 
precisely when
$k$ lies in the negative real interval $k_\mathfrak{a}\le k\le k_\mathfrak{b}<0$.  Moreover,
for each $k\in (k_\mathfrak{a},k_\mathfrak{b})$ there exist precisely two turning points $t_-(k)<t_+(k)$.
The two turning points coalesce as $k\downarrow k_\mathfrak{a}$ and as $k\uparrow k_\mathfrak{b}$:  $t_-(k_\mathfrak{a})=t_+(k_\mathfrak{a})=t_\mathfrak{a}$ and
$t_-(k_\mathfrak{b})=t_+(k_\mathfrak{b})=t_\mathfrak{b}$.  Also, $t_-(k)\to 0$ as $k\to k_0$ while $t_+(k)\to +\infty$ as $k\to k_\infty$.
Finally, given any $\delta>0$, the condition that $k\in \mathcal{Q}^\mathrm{II}_\delta$ bounds $\lambda$ away from zero uniformly for $t>0$.
\end{lemma}
We omit the proof as it is a direct and easy consequence of the conditions on the functions $U(t)$ and $H(t)$ and formula \eqref{eq:lambdasquared}.  Given a value of $k$, the presence or absence of turning points determines the nature of the spectral functions in the semiclassical limit.

\subsubsection{Analysis in the absence of turning points}
According to Lemma~\ref{lem:turningpoints}, there are no turning points if $k\in\mathcal{Q}^\mathrm{II}_\delta$.  This implies a certain triviality of the spectral functions in this region of the $k$-plane.  We have the following result.
\begin{proposition}
Assume that $q_0(x)=0$.
Let a number $\delta>0$ be given, and suppose that the functions $H:\mathbb{R}_+\to \mathbb{R}$ and $U:\mathbb{R}_+\to\mathbb{R}$ satisfy Assumption~\ref{assumption:data}.  Then:
\begin{itemize}
\item
For bounded $\eps>0$,  any zeros in $\mathcal{Q}_\delta^\mathrm{II}$ of the analytic function $d_0(k):=A_0(k^*)^*$
lie in an $\eps$-independent bounded subset.
\item
The analytic function $d_0(k)$ has no zeros in $\mathcal{Q}_\delta^\mathrm{II}$ if $\eps>0$ is sufficiently small.
\item The function $\Gamma_0(k):=B_0(k^*)^*/A_0(k^*)^*$ satisfies a bound of the form
\begin{equation}
\Gamma_0(k) = \mathcal{O}\left(\frac{\eps^{1/2}}{1+|k|^3}\right),\quad k\in \mathcal{Q}_\delta^\mathrm{II},\quad\eps>0,
\label{eq:Gammabound}
\end{equation}
where the constant implicit in the estimate depends only on the functions $H$ and $U$.
\end{itemize}
\label{prop:noturningpoints}
\end{proposition}
In other words, when $\eps>0$ is small, then for $k\in \mathcal{Q}_\delta^\mathrm{II}$, $\Gamma_0(k)$ has no poles and is uniformly small.
The proof of this proposition is given in Appendix~\ref{sec:Appendix-noturningpoints}.

\subsubsection{Analysis in the presence of turning points}
We now study the asymptotic behavior of the function $\Gamma_0(k):=B_0(k^*)^*/A_0(k^*)^*$ for real $k$ in the interval $k_\mathfrak{a}<k<k_\mathfrak{b}$.  For each such $k$, as can be seen in Figure~\ref{fig:Assumption1}, the eigenvalues $\pm\lambda$ of the coefficient matrix $\mathbf{B}(t;k)$ satisfy $\lambda^2>0$ for $t_-(k)<t<t_+(k)$, while $\lambda^2<0$ for $0<t<t_-(k)$ and for $t>t_+(k)$.  Considering $\eps>0$ as being very small, one is reminded of the language of the WKB method, in which the interval $(t_-(k),t_+(k))$ is analogous to a ``classically forbidden region'' separating two ``classically allowed regions''.  Thus we have an analogue of a quantum tunneling problem.  Rather than use the WKB method, which is well-known to fail near the turning points $t_\pm(k)$, in the proof of the following results we use the method of Langer transformations to uniformly handle the neighborhoods of the two turning points while simultaneously maintaining full accuracy when $t$ is not close to either turning point.  The presence of turning points leads to nontrivial behavior of the spectral functions in the limit $\eps\downarrow 0$, as the following result shows.
\begin{proposition}
Let $k\in (k_\mathfrak{a},k_\mathfrak{b})$ with $k\neq k_0$ and $k\neq k_\infty$, and suppose that the functions $H:\mathbb{R}_+\to\mathbb{R}$ and $U:\mathbb{R}_+\to\mathbb{R}$ satisfy Assumption~\ref{assumption:data}.  Then in the limit $\eps\downarrow 0$,
\begin{equation}
A_0(k)=-e^{\tau(k)/\eps}e^{-\myi(\ell(k)\sgn(k^2-k_\infty^2)-\tfrac{1}{2}S_\infty)/\eps}\left[e^{-\myi\Phi(k)/\eps} + \mathcal{O}(\eps)\right]
\end{equation}
and
\begin{equation}
B_0(k)=-e^{\tau(k)/\eps}e^{-\myi(\ell(k)\sgn(k^2-k_\infty^2)-\tfrac{1}{2}S_\infty)/\eps}\left[e^{i\Phi(k)/\eps} + \mathcal{O}(\eps)\right]
\end{equation}
where
\begin{equation}
\ell(k):=\int_{t_+(k)}^{+\infty}\left[(U(t)-2k)\sqrt{(k-\mathfrak{a}(t))(k-\mathfrak{b}(t))}-2|k^2-k_\infty^2|\right]\,dt - 2|k^2-k_\infty^2|t_+(k),
\label{eq:ell-define}
\end{equation}
\begin{equation}
S_\infty:=\lim_{t\to +\infty}\left(S(t)+U_\infty^2 t\right)=\lim_{t\to +\infty}\left(S(t)+4k_\infty^2t\right),
\label{eq:S-infty}
\end{equation}
\begin{equation}
\tau(k):=\int_{t_-(k)}^{t_+(k)}(U(t)-2k)\sqrt{(k-\mathfrak{a}(t))(\mathfrak{b}(t)-k)}\,dt,
\label{eq:tau-define-1}
\end{equation}
and
\begin{equation}
\Phi(k):=\frac{1}{2}S(0)+\sgn(k^2-k_0^2)\int_0^{t_-(k)}(U(t)-2k)\sqrt{(k-\mathfrak{a}(t))(k-\mathfrak{b}(t))}\,dt.
\label{eq:Phi-define}
\end{equation}
The error terms are uniform in $k$ in compact subintervals of $(k_\mathfrak{a},k_\mathfrak{b})\setminus\{k_0,k_\infty\}$. 
\label{prop:two-turning-points}
\end{proposition}

The proof of Proposition~\ref{prop:two-turning-points} is given in Appendix~\ref{sec:Appendix-two-turning-points}.  An immediate corollary is the following.
\begin{corollary}
Suppose that $q_0(x)=0$.
Under the same conditions and with the same characterization of the error terms as in Proposition~\ref{prop:two-turning-points}, we have
\begin{equation}
\Gamma_0(k)=\frac{B_0(k)^*}{A_0(k)^*}=\sqrt{1-e^{-2\tau(k)/\eps}}e^{-2\myi\Phi(k)/\eps} + \mathcal{O}(\eps)\quad\text{and}\quad
1-|\Gamma_0(k)|^2=e^{-2\tau(k)/\eps}(1+\mathcal{O}(\eps)).
\end{equation}
\label{corollary:Gamma}
\end{corollary}
\begin{proof}
Since $q_0(x)=0$ we have $\gamma(k)=0$ and hence $\Gamma_0(k)=B_0(k)^*/A_0(k)^*$ for all real $k<0$.  
The formula for $1-|\Gamma_0(k)|^2$ follows from the identity $1-|\Gamma_0(k)|^2=1/|A_0(k)|^2$ equivalent to the condition that $\det(\mathbf{T}_0(0;k))=1$.  The factor $\sqrt{1-e^{-2\tau(k)/\eps}}$ is exponentially close to $1$ (except near $k_\mathfrak{a}$ and $k_\mathfrak{b}$, points excluded from consideration) and is included in the formula for $\Gamma_0$ to ensure that the jump matrix we shall construct from this approximation has unit determinant.
\end{proof}

\section{Formulation of the Inverse Problem}
\label{sec:Formulation}
Propositions~\ref{prop:noturningpoints} and \ref{prop:two-turning-points} and Corollary~\ref{corollary:Gamma} give a rigorous characterization of the spectral functions associated with vanishing initial data $q_0(x)$  and with a class of boundary data $(Q^\mathrm{D},Q^\mathrm{N}_0)$
given (in terms of the functions $H(\cdot)$ and $U(\cdot)$ described by Assumption~\ref{assumption:data}) by \eqref{eq:DirichletDataForm} and \eqref{5} subject to the equation \eqref{eq:SprimeUH} giving $S(t)$ in terms of $H(\cdot)$ and $U(\cdot)$.  To summarize:
\begin{itemize}
\item From $q_0=0$ we have $\gamma(k)=0$ for all $k\in\mathbb{R}$.  Therefore, Riemann-Hilbert Problem~\ref{rhp-original} has no jump on the positive real axis, and the remaining jump matrices only involve $\Gamma_0(k)$, which is simply a ratio of the spectral functions $A_0(k^*)^*$ and $B_0(k^*)^*$ arising from the approximate boundary data.
\item If the function $\Gamma_0(k)$ has any poles in the second quadrant of the complex plane, they must lie very close (in the limit $\eps\downarrow 0$) to the negative real interval $[k_\mathfrak{a},k_\mathfrak{b}]$.
\item On the imaginary $k$-axis, as well as on the negative real $k$-axis away from the interval $[k_\mathfrak{a},k_\mathfrak{b}]$, $\Gamma_0(k)$ is small in the limit $\eps\downarrow 0$.
\item In the interior of the negative real interval $(k_\mathfrak{a},k_\mathfrak{b})$ and away from the special points $k_0$ and $k_\infty$, $\Gamma_0(k)$ has an accurate explicit approximation given by Corollary~\ref{corollary:Gamma}.
\end{itemize}
However, this information alone is insufficient to properly formulate and analyze Riemann-Hilbert Problem~\ref{rhp-original} associated with the exact spectral transforms $A_0(k)$ and $B_0(k)$ corresponding to the approximate Neumann boundary data $Q^\mathrm{N}_0(t)$.  Indeed, to formulate the Riemann-Hilbert problem without poles one would need to know a priori that there cannot be any poles of $\Gamma_0(k)$ whatsoever in the second quadrant, and it is not enough to know that any poles have to move toward $[k_\mathfrak{a},k_\mathfrak{b}]$ as $\eps\downarrow 0$.  Another issue is that our results do not provide approximations for $\Gamma_0(k)$ near the real points $k_\mathfrak{a}$, $k_0$, $k_\infty$, or $k_\mathfrak{b}$.  In fact, the analytical methodology based on Langer transformations used in the proof of Proposition~\ref{prop:two-turning-points} either requires substantial modification or breaks down entirely in neighborhoods of these points.  

These arguments\footnote{A more serious reason for making this modification, especially the step of setting the jump matrix on the imaginary axis to the identity, is discussed in detail in Remark~\ref{remark-imaginary-amplify}.} suggest making a further modification of the first step of the proposed iteration algorithm:  we will reformulate the inverse problem by:
\begin{itemize}
\item Assuming that the Riemann-Hilbert problem can be formulated without poles,
\item Neglecting $\Gamma_0(k)$ and $\Gamma_0(k)^*$ completely on the imaginary axis,
\item Neglecting $\Gamma_0(k)$ on the real axis for $k<k_\mathfrak{a}$ and $k>k_\mathfrak{b}$, and
\item Replacing $\Gamma_0(k)$ in the whole interval $(k_\mathfrak{a},k_\mathfrak{b})$ by the formulae recorded in Corollary~\ref{corollary:Gamma} with the $\mathcal{O}(\eps)$ error terms set to zero.
\end{itemize}
The resulting Riemann-Hilbert problem has the negative real interval $[k_\mathfrak{a},k_\mathfrak{b}]$ as its only jump contour.  For convenience we will re-orient this contour from left to right, which requires the inversion of the jump matrix written in \eqref{eq:exact-jump-negative}.  

We therefore formulate the following Riemann-Hilbert problem.  Let $\tilde{\Gamma}(k)$ be defined by
\begin{equation}
\tilde{\Gamma}(k):=\chi_{(k_\mathfrak{a},k_\mathfrak{b})}(k)Y^\eps(k)e^{-2i\Phi(k)/\eps},\quad Y^\eps(k):=\sqrt{1-e^{-2\tau(k)/\eps}},\quad k\in\mathbb{R},
\label{eq:tilde-Gamma}
\end{equation}
where $\chi_{(k_\mathfrak{a},k_\mathfrak{b})}$ denotes the characteristic function of the interval $(k_\mathfrak{a},k_\mathfrak{b})$, and where $\tau:(k_\mathfrak{a},k_\mathfrak{b})\to \mathbb{R}_+$ is defined by \eqref{eq:tau-define-1} while $\Phi:(k_\mathfrak{a},k_\mathfrak{b})\to\mathbb{R}$ is defined by \eqref{eq:Phi-define}.  It can be shown that $\tilde{\Gamma}:\mathbb{R}\to\mathbb{C}$ is H\"older continuous with every exponent $0<\alpha\le 1$.
\begin{rhp}
Seek a $2\times 2$ matrix function $\tilde{\mathbf{M}}:\mathbb{C}\setminus\mathbb{R}\to\mathbb{C}^{2\times 2}$ with the following properties:
\begin{itemize}
\item[]\textit{\textbf{Analyticity}}: $\tilde{\mathbf{M}}$ is analytic in $\mathbb{C}_+$ and $\mathbb{C}_-$
and H\"older continuous for some exponent $0<\alpha<1$ in $\overline{\mathbb{C}}_+$ and $\overline{\mathbb{C}}_-$, taking boundary values $\tilde{\mathbf{M}}_\pm:\mathbb{R}\to\mathbb{C}^{2\times 2}$ on $\mathbb{R}$ from $\mathbb{C}_\pm$. 
\item[]\textit{\textbf{Jump Condition}}:  The boundary values are related  by
\begin{equation}
\tilde{\mathbf{M}}_+(k)=\tilde{\mathbf{M}}_-(k)\begin{bmatrix}1-|\tilde{\Gamma}(k)|^2 & -\tilde{\Gamma}(k)^*e^{-2\myi\theta(k;x,t)/\eps}\\\tilde{\Gamma}(k)e^{2\myi\theta(k;x,t)/\eps} & 1\end{bmatrix},\quad k\in\mathbb{R}.
\label{eq:M-tilde-jump}
\end{equation}
\item[]\textit{\textbf{Normalization}}:  The matrix function $\tilde{\mathbf{M}}$ satisfies
\begin{equation}
\lim_{k\to\infty}\tilde{\mathbf{M}}(k)=\mathbb{I},
\label{eq:M-tilde-normalize}
\end{equation}
where the limit is uniform with respect to direction in the complex plane.
\end{itemize}
\label{rhp-M-tilde}
\end{rhp}
The following is a standard result.
\begin{proposition}
Riemann-Hilbert Problem~\ref{rhp-M-tilde} has a unique solution for every $(x,t)\in\mathbb{R}^2$ and for every $\eps>0$.  The function 
\begin{equation}
q=\tilde{q}^\eps(x,t):=2\myi\lim_{k\to\infty}k\tilde{M}_{12}(k)
\label{eq:q-tilde-define}
\end{equation}
is infinitely differentiable with respect to $(x,t)$, and for every $\eps>0$, $q=\tilde{q}^\eps(x,t)$ satisfies the defocusing nonlinear Schr\"odinger equation \eqref{eq:NLS}.
\label{prop:solution-exists}
\end{proposition}
\begin{proof}
To see the uniqueness of the solution of Riemann-Hilbert Problem~\ref{rhp-M-tilde} (assuming existence), one first notes that $\det(\tilde{\mathbf{M}}(k))=1$ necessarily holds as an identity for any solution, and therefore $\tilde{\mathbf{M}}(k)^{-1}$ is also analytic for $k\in\mathbb{C}\setminus\mathbb{R}$.  Given two solutions, say $\tilde{\mathbf{M}}$ and $\tilde{\mathbf{N}}$, one considers the matrix ratio $\mathbf{R}(k):=\tilde{\mathbf{M}}(k)\tilde{\mathbf{N}}(k)^{-1}$,
which is analytic for $k\in\mathbb{C}\setminus\mathbb{R}$ and satisfies $\mathbf{R}(k)\to\mathbb{I}$ as $k\to\infty$.  A simple calculation using the jump condition \eqref{eq:M-tilde-jump} satisfied by both $\tilde{\mathbf{M}}$ and $\tilde{\mathbf{N}}$ shows that the continuous boundary values taken on $\mathbb{R}$ agree:  $\mathbf{R}_+(k)=\mathbf{R}_-(k)$ for all $k\in\mathbb{R}$.  It follows that $\mathbf{R}(k)$ is an entire analytic (matrix-valued) function of $k$ that tends to $\mathbb{I}$ as $k\to\infty$, so by Liouville's theorem $\mathbf{R}(k)=\mathbb{I}$ for all $k$, i.e., $\tilde{\mathbf{N}}(k)=\tilde{\mathbf{M}}(k)$ holds for all $k\in\mathbb{C}\setminus\mathbb{R}$.

To establish existence of a solution, one observes that Riemann-Hilbert Problem~\ref{rhp-M-tilde} is equivalent to a system of linear singular integral equations for which the relevant operator is Fredholm with zero index on an appropriate space of H\"older-continuous functions (see \cite{Muskhelishvili} and \cite{Zhou89}).  It therefore suffices to show that the kernel of this Fredholm operator is trivial.  Triviality of the kernel is equivalent to the assertion that the only solution of the homogeneous form of Riemann-Hilbert Problem~\ref{rhp-M-tilde}, in which the normalization condition \eqref{eq:M-tilde-normalize} is replaced with a limit of $\mathbf{0}$ as $k\to\infty$, is the zero matrix.  Zhou's \emph{vanishing lemma} \cite[Theorem 9.3]{Zhou89} shows that this latter assertion holds true in the present case because the jump contour is the real axis and the jump matrix has a positive semidefinite real part as a consequence of the inequality $|\tilde{\Gamma}(k)|^2\le 1$ holding for $k\in\mathbb{R}$.

The infinite differentiability of the matrix $\tilde{\mathbf{M}}$ with respect to $(x,t)$, and hence that of $\tilde{q}^\eps$, follows from the compact support $[k_\mathfrak{a},k_\mathfrak{b}]$ of $\tilde{\Gamma}$.  Finally, let us show that $q=\tilde{q}^\eps$ satisfies \eqref{eq:NLS}.  We begin by defining a matrix $\mathbf{L}(k)$ from the solution $\tilde{\mathbf{M}}(k)$ by setting
\begin{equation}
\mathbf{L}(k):=\tilde{\mathbf{M}}(k)e^{-\myi\theta(k;x,t)\sigma_3/\eps}.
\end{equation}
One verifies that $\mathbf{L}(k)$ is analytic for $k\in\mathbb{C}\setminus\mathbb{R}$, and that 
\begin{equation}
\mathbf{L}_+(k)=\mathbf{L}_-(k)\begin{bmatrix} 1-|\tilde{\Gamma}(k)|^2 & -\tilde{\Gamma}(k)^*\\
\tilde{\Gamma}(k) & 1\end{bmatrix},\quad k\in\mathbb{R}.
\end{equation}
It follows from the fact that this jump condition is independent of $x$ and $t$, that the matrices $\mathbf{U}(k):=\eps\mathbf{L}_x(k)\mathbf{L}(k)^{-1}$ and $\mathbf{V}(k):=\eps\mathbf{L}_t(k)\mathbf{L}(k)^{-1}$ have no jump across the real axis and since $\det(\mathbf{L}(k))=1$, $\mathbf{U}$ and $\mathbf{V}$ are entire functions of $k$.  Moreover, 
from the asymptotic behavior of $\mathbf{L}(k)$ near $k=\infty$ one can check that $\mathbf{U}$ is a linear function of $k$ while $\mathbf{V}$ is a quadratic polynomial in $k$.  In fact, using \eqref{eq:q-tilde-define} one sees that $\mathbf{U}$ is given by \eqref{eq:xproblem} with $q=\tilde{q}^\eps$.
Moreover, using the fact that $\mathbf{L}(k)$ satisfies the differential equation $\eps \mathbf{L}_x=\mathbf{U}\mathbf{L}$ (by definition of $\mathbf{U}$), one sees that $\mathbf{V}$ is given by \eqref{eq:tproblem} with $q=\tilde{q}^\eps$.
The fact that $\mathbf{L}$ is a simultaneous fundamental solution matrix of the Lax pair equations $\eps\mathbf{L}_x=\mathbf{U}\mathbf{L}$ and $\eps\mathbf{L}_t=\mathbf{V}\mathbf{L}$ means that these equations are consistent, that is, the zero-curvature condition 
\begin{equation}
\eps \mathbf{U}_t-\eps \mathbf{V}_x + [\mathbf{U},\mathbf{V}]=\mathbf{0}
\end{equation}
holds, and substitution from \eqref{eq:xproblem} and \eqref{eq:tproblem} yields the equation 
\eqref{eq:NLS} for $q=\tilde{q}^\eps$ (and the complex conjugate of that equation).
\end{proof}

We note that this proof implies that $\tilde{\mathbf{M}}(k)$ has a convergent Laurent series in descending powers of $k$ for $|k|$ sufficiently large:
\begin{equation}
\tilde{\mathbf{M}}(k)=\mathbb{I}+k^{-1}\tilde{\mathbf{M}}_1 + k^{-2}\tilde{\mathbf{M}}_2 + \mathcal{O}(k^{-3}),\quad k\to\infty,
\label{eq:tilde-M-expansion}
\end{equation}
and that $\tilde{q}^\eps(x,t)$ and $\eps\tilde{q}^\eps_x(x,t)$ can be expressed in terms of the coefficients $\tilde{\mathbf{M}}_1$ and $\tilde{\mathbf{M}}_2$ as follows:
\begin{equation}
\tilde{q}^\eps(x,t)=2\myi\tilde{M}_{1,12}\quad\text{and}\quad
\eps\tilde{q}^\eps_x(x,t)=4\tilde{M}_{2,12}+2\myi\tilde{q}^\eps(x,t)\tilde{M}_{1,22}=4\tilde{M}_{2,12}-4
\tilde{M}_{1,12}\tilde{M}_{1,22}.
\label{eq:tilde-q-qx-moments}
\end{equation}

The question that remains is what, if anything, does the family of functions $\tilde{q}^\eps(x,t)$ have to do with the exact solution $q^\eps(x,t)$ of the Dirichlet initial-boundary value problem with $q^\eps(x,0)=0$ for $x>0$ and $q^\eps(0,t)=H^\eps(t)e^{iS(t)/\eps}$ for $t>0$ (recall the modified amplitude function $H^\eps(t)$ defined in Remark~\ref{remark:amplitude-fix}).  This is the topic we take up next.

\begin{remark}
The values of $\Gamma(k)$ for $k$ real and positive are irrelevant to the inverse problem, as the jump matrix for $k>0$ generally only involves the function $\gamma(k)$ obtained from the initial data for $x>0$ (see \eqref{eq:exact-jump-positive}), and in the present case that $q_0(x)=0$, $\gamma(k)=0$.  
The condition $U(t)>2H(t)$ implied by Assumption~\ref{assumption:data} ensures that $k_\mathfrak{b}<0$, and hence that the full asymptotic support of $\Gamma$ on $\mathbb{R}$ (and hence by definition the exact support of $\tilde{\Gamma}$ on $\mathbb{R}$) contributes to the jump matrix for the inverse problem.  If on the contrary we had $k_\mathfrak{b}>0$, then some information about the boundary data would be lost from the inverse problem in the semiclassical limit.  It therefore seems that it is possible to reconstruct the boundary data only if $U(t)>2H(t)$.  As pointed out earlier, this is a stronger condition than the necessary condition $U(t)>H(t)$ for the boundary to be a spacelike curve for the hyperbolic dispersionless system \eqref{eq:dispersion-less-NLS}.
\myendrmk
\label{remark:kb-negative}
\end{remark}

\begin{remark}
One may observe that Riemann-Hilbert Problem~\ref{rhp-M-tilde} is of exactly the same form as that which occurs in the treatment of the initial-value problem for the defocusing nonlinear Schr\"odinger equation formulated on an appropriate space of decaying functions of $x\in\mathbb{R}$ instead of the half-line.  The function $\tilde{\Gamma}(k)$, here obtained from Dirichlet boundary data via the temporal part of the Lax pair, plays the role usually played by the reflection coefficient calculated from initial data via the spatial part of the Lax pair.  This means that the ``reflection coefficient'' $\tilde{\Gamma}(k)$ corresponds to some initial data given on the whole line $x\in\mathbb{R}$, a fact that has been made quite rigorous in \cite{Zhou98}.  In this case, according to Theorem~\ref{theorem:initial-condition-small}, the initial data is very small for $x>0$; however to reproduce the nontrivial boundary data described by Theorem~\ref{theorem:boundary-condition-recover} the initial data must not be small for $x<0$.  Thus, the formula \eqref{eq:tilde-Gamma} for $\tilde{\Gamma}(k)$ makes a connection in the transform domain between (i) a problem on the half-line with zero initial data and nontrivial boundary data and (ii) a problem on the whole line with initial data supported on the negative half-line.  The latter initial data is then defined implicitly in terms of the boundary data for the former problem via the solution of Riemann-Hilbert Problem~\ref{rhp-M-tilde}.
\myendrmk
\label{remark:half-line-support}
\end{remark}

\section{Semiclassical Analysis of $\tilde{q}^\eps(x,t)$}
\subsection{Asymptotic behavior of $\tilde{q}^\eps(x,0)$ for $x>0$ and related analysis}
\subsubsection{Implications of Assumption~\ref{assumption:data} for the functions $\Phi$ and $\tau$}
\label{sec:preliminary-results}
In light of Assumption~\ref{assumption:data}, the definition \eqref{eq:tau-define-1} of $\tau$ easily implies the following.
\begin{lemma}
Under Assumption~\ref{assumption:data}, the function $\tau:(k_\mathfrak{a},k_\mathfrak{b})\setminus\{k_0,k_\infty\}\to\mathbb{R}$ is real analytic and it extends by continuity to a function of class $C^0(k_\mathfrak{a},k_\mathfrak{b})$ satisfying $\tau(k)>0$.
\label{lemma:tau-continuous-positive}
\end{lemma}
We will also require detailed information about the behavior of the function $\Phi$ near the point $k=k_0$.  In this direction we have the following.
\begin{lemma}
If the functions $U$ and $H$ satisfy Assumption~\ref{assumption:data}, then
$\Phi:(k_\mathfrak{a},k_\mathfrak{b})\setminus\{k_0,k_\infty\}\to\mathbb{R}$ defined by \eqref{eq:Phi-define} extends by continuity to a function analytic in $(k_\mathfrak{a},k_\mathfrak{b})\setminus\{k_0\}$, and of class $C^3$ in a neighborhood of $k=k_0$.  Also, $\Phi'(k)\le 0$ for $k_\mathfrak{a}<k<k_\mathfrak{b}$ with equality only for $k=k_0$.
\label{lemma-Phi-interior}
\end{lemma}

\begin{proof}
Using analyticity of $U$ and $H$, which implies that of $\mathfrak{a}$ and $\mathfrak{b}$, we may
write $\Phi(k)$ in terms of a contour integral.  Indeed, 
\begin{equation}
\Phi(k)=\frac{1}{2}S(0)+\frac{1}{2}\myi\,\sgn(k^2-k_0^2)\oint_C(U(t)-2k)(k-\mathfrak{a}(t))^{1/2}(\mathfrak{b}(t)-k)^{1/2}\,dt,
\label{eq:Phi-contour}
\end{equation}
where the fractional powers denote the principal branches.  The integrand has a branch cut connecting $t=0$ with $t=t_-(k)$ (due to the factor $(k-\mathfrak{a}(t))^{1/2}$ when $k_\mathfrak{a}<k<k_0$ and due to the factor $(\mathfrak{b}(t)-k)^{1/2}$ when $k_0<k<k_\mathfrak{b}$).
The contour $C$ is a positively-oriented loop; it begins at $t=0$ on the lower edge of the branch cut, encloses the cut once passing through the real axis at a point $t=t_-(k)+\delta<t_+(k)$, and terminates at $t=0$ on the upper edge of the cut.  
In the neighborhood of a fixed value of $k$ the contour $C$ may be taken to be independent of $k$, and it follows easily that $\Phi(k)$ is analytic in $k$ separately in the intervals $(k_\mathfrak{a},k_0)$ and $(k_0,k_\mathfrak{b})$.  

For $k\neq k_0$, all derivatives of $\Phi$ may be calculated by differentiation under the integral sign.  Thus,
\begin{equation}
\Phi'(k)=\myi\,\sgn(k^2-k_0^2)\oint_C\frac{2k^2+U(t)k-H(t)^2}{(k-\mathfrak{a}(t))^{1/2}(\mathfrak{b}(t)-k)^{1/2}}\,dt,\quad k\neq k_0,
\end{equation}
\begin{equation}
\Phi''(k)=\frac{1}{4}\myi\,\sgn(k^2-k_0^2)\oint_C\frac{2H(t)^2(6k+U(t))-(2k+U(t))^3}{(k-\mathfrak{a}(t))^{3/2}(\mathfrak{b}(t)-k)^{3/2}}\,dt,\quad k\neq k_0,
\end{equation}
\begin{equation}
\Phi'''(k)=\frac{3}{2}\myi\,\sgn(k^2-k_0^2)\oint_C\frac{H(t)^2(2H(t)^2-U(t)(2k+U(t)))}{(k-\mathfrak{a}(t))^{5/2}(\mathfrak{b}(t)-k)^{5/2}}\,dt,\quad k\neq k_0,
\end{equation}
and
\begin{equation}
\Phi^{(4)}(k)=\frac{3}{2}\myi\,\sgn(k^2-k_0^2)\oint_C\frac{H(t)^2(H(t)^2(10k+3U(t))-2U(t)(2k+U(t))^2)}{(k-\mathfrak{a}(t))^{7/2}(\mathfrak{b}(t)-k)^{7/2}}\,dt,\quad k\neq k_0.
\end{equation}
We consider $k$ to be a real number close to, but not equal to, $k_0$.
Assumption~\ref{assumption:data} ensures that $\mathfrak{a}(t)=k_0-h_0t^{1/2}+o(t^{1/2})$
and $\mathfrak{b}(t)=k_0+h_0t^{1/2}+o(t^{1/2})$ for small $t$.  This implies that when $k-k_0$ is small, $t_-(k)$ is proportional to $(k-k_0)^2$.  Based on this observation, we scale the contour $C$ as $C=h_0^{-2}(k-k_0)^2D$, where $D$ is a suitable contour that we will hold fixed as $k\to k_0$, and we make the substitution $t=h_0^{-2}(k-k_0)^2s$ in the above integrals.  In each case, the integrand considered as a function of $s$ has uniform asymptotic behavior on the contour $D$ in the limit $k\to k_0$, 
being determined from the local behavior of the functions $U$ and $H$ near $t=0$ as specified in Assumption~\ref{assumption:data}.  Thus, uniformly for $s\in D$ one has $U(t)=U(h_0^{-2}(k-k_0)^2s)=-2k_0 + o(k-k_0)$ and $H(t)=H(h_0^{-2}(k-k_0)^2s)=|k-k_0|s^{1/2}(1+o(1))$ in the limit $k\to k_0$, and it follows that
\begin{equation}
\Phi(k)=\frac{1}{2}S(0)+2\myi k_0h_0^{-2}(k-k_0)^3\left[\oint_D(s-1)^{1/2}\,ds+o(1)\right],\quad k\to k_0,
\end{equation}
\begin{equation}
\Phi'(k)=-2\myi k_0 h_0^{-2}(k-k_0)^2\left[\oint_D\frac{ds}{(s-1)^{1/2}}+o(1)\right],\quad k\to k_0,
\end{equation}
\begin{equation}
\Phi''(k)=-2\myi k_0 h_0^{-2} (k-k_0)\left[\oint_D\frac{s\,ds}{(s-1)^{3/2}} + o(1)\right],\quad k\to k_0,
\end{equation}
\begin{equation}
\Phi'''(k)=-6\myi k_0h_0^{-2}\left[\oint_D\frac{s\,ds}{(s-1)^{5/3}}+o(1)\right],\quad k\to k_0,
\end{equation}
and
\begin{equation}
\Phi^{(4)}(k)=-6\myi k_0 h_0^{-2} (k-k_0)^{-1}\left[\oint_D\frac{s^2+4s}{(s-1)^{7/2}}\,ds + o(1)\right],\quad k\to k_0.
\label{eq:Phi-Four}
\end{equation}
The contour $D$ begins and ends at $s=0$ on opposite sides of the branch cut (all fractional powers of $s-1$ are understood as principal branches) and encircles the branch point $s=1$ once in the counterclockwise sense.  It is now obvious that $\Phi(k)$ tends to $\tfrac{1}{2}S(0)$ while $\Phi'(k)$ and $\Phi''(k)$ both vanish as $k\to k_0$ and hence all three extend by continuity to $k=k_0$.  It is also obvious that $\Phi'''(k)$ has a finite limit as $k\to k_0$; by computing an integral we find the limiting value 
\begin{equation}
\Phi'''(k_0)=16k_0 h_0^{-2}.
\label{eq:Phi-Three-k0}
\end{equation}
This completes the proof that $\Phi(k)$ is of class $C^3$ near $k=k_0$.  We note that the fourth derivative $\Phi^{(4)}(k)$ appears to be singular at $k=k_0$, but in reality the issue is subtle because the explicit leading term in the square brackets in \eqref{eq:Phi-Four} is an integral that vanishes identically,
and therefore the asymptotic behavior of $\Phi^{(4)}(k)$ in the limit $k\to k_0$ cannot be determined without making further hypotheses on $U$ and $H$ sufficient to provide leading-order asymptotics for the $o(1)$ error term in \eqref{eq:Phi-Four}.  

It remains to determine the sign of $\Phi'(k)$ for $k\neq k_0$.  We go back to the real integral representation \eqref{eq:Phi-define} for $\Phi(k)$, which admits differentiation by Leibniz' rule because either $k-\mathfrak{a}(t_-(k))=0$ or $k-\mathfrak{b}(t_-(k))=0$ with the result:
\begin{equation}
\Phi'(k)=-\sgn(k^2-k_0^2)\int_0^{t_-(k)}\frac{4k^2+2U(t)k-2H(t)^2}{\sqrt{(k-\mathfrak{a}(t))(k-\mathfrak{b}(t))}}\,dt,
\label{eq:Phi-prime-real-1}
\end{equation}
which can also be written in the form
\begin{equation}
\Phi'(k)=-\sgn(k^2-k_0^2)\int_0^{t_-(k)}\frac{2(\mathfrak{a}(t)-k)(\mathfrak{b}(t)-k)+\tfrac{1}{2}(U(t)-2k)((\mathfrak{a}(t)-k)+(\mathfrak{b}(t)-k))}{\sqrt{(\mathfrak{a}(t)-k)(\mathfrak{b}(t)-k)}}\,dt.
\label{eq:Phi-prime-real-2}
\end{equation}
If $k_0<k<k_\mathfrak{b}$, then we use the form \eqref{eq:Phi-prime-real-1} and factor the quadratic in the numerator as $4(k-k_+(t))(k-k_-(t))$ with
\begin{equation}
k_\pm(t):=\frac{1}{4}\left(\mathfrak{a}(t)+\mathfrak{b}(t)\pm\sqrt{(\mathfrak{a}(t)+\mathfrak{b}(t))^2+(\mathfrak{b}(t)-\mathfrak{a}(t))^2}\right).
\label{eq:k-plus-k-minus}
\end{equation}
Since $k_+(t)>0$ and $k<k_\mathfrak{b}<0$, obviously $k-k_+(t)<0$.  Also, since $k_-(t)<\tfrac{1}{2}(\mathfrak{a}(t)+\mathfrak{b}(t))\le \mathfrak{b}(t)$ we have $k-k_-(t)>k-\mathfrak{b}(t)$, which is nonnegative for $0<t<t_-(k)$ because $k>k_0$.  Hence $\Phi'(k)< 0$ for $k_0<k<k_\mathfrak{b}$.
On the other hand, for $k_\mathfrak{a}<k<k_0$ we instead use the form \eqref{eq:Phi-prime-real-2}, because in this range of $k$ we have $\mathfrak{a}(t)-k>0$ and $\mathfrak{b}(t)-k>0$ for $0<t<t_-(k)$ so combining this with the inequality $U(t)-2k>0$ shows that also for $k_\mathfrak{a}<k<k_0$ we have $\Phi'(k)< 0$.
\end{proof}
\begin{remark}
\label{remark:Phi-smooth}
Although it may seem counterintuitive, assuming that $H$ is smoother at $t=0$, say vanishing linearly rather than like $t^{1/2}$ as $t\downarrow 0$, leads to less smoothness of $\Phi(k)$ at $k=k_0$.  Linear vanishing of $H$ implies continuity of $\Phi$ and $\Phi'$ at $k_0$,
but $\Phi''$ will have a jump discontinuity.  The point is that it should be the inverse function $t_-(k)$ that is smooth at $k=k_0$, not the functions $\mathfrak{a}(t)$ and $\mathfrak{b}(t)$ at $t=0$.
\myendrmk
\end{remark}

The part of Assumption~\ref{assumption:data} concerning the nondegeneracy of the extrema
of $\mathfrak{a}$ and $\mathfrak{b}$ allows us to obtain the following result.
\begin{lemma}
If the functions $U$ and $H$ satisfy Assumption~\ref{assumption:data},
then $\tau:(k_\mathfrak{a},k_\mathfrak{b})\to\mathbb{R}_+$ is analytic at $k=k_\mathfrak{a}$ and $k=k_\mathfrak{b}$, with $\tau(k_\mathfrak{a})=0$ and $\tau'(k_\mathfrak{a})>0$ while $\tau(k_\mathfrak{b})=0$ and $\tau'(k_\mathfrak{b})<0$.   
\label{lemma-tau}
\end{lemma}

\begin{proof}
Using analyticity of $U$ and $H$ for $t$ near $t_\mathfrak{a}$ and $t_\mathfrak{b}$, we can express $\tau(k)$ as a contour integral:
\begin{equation}
\tau(k)=\frac{1}{2}\oint_C(U(t)-2k)R(t;k)\,dt
\label{eq:tau-contour}
\end{equation}
where $C$ is a closed contour enclosing the interval $[t_-(k),t_+(k)]$ once in the positive sense and where $R(t;k)$ is the function analytic in $D\setminus[t_-(k),t_+(k)]$, where $D$ is a domain containing $C$, that satisfies $R(t;k)^2=(k-\mathfrak{a}(t))(\mathfrak{b}(t)-k)$ and that the boundary value $R_+(t;k)$ taken on the upper edge of the branch cut $[t_-(k),t_+(k)]$ is negative.  

To analyze $\tau$ for $k$ near $k_\mathfrak{a}$ and $k_\mathfrak{b}$, we may in each case choose the contour $C$ to be fixed, and then since the integrand is analytic in $k$ at each point $t\in C$
it follows that $\tau$ extends from a function defined for real $k$ in a right (left) neighborhood of $k=k_\mathfrak{a}$ ($k=k_\mathfrak{b}$) to an analytic function of $k$ at $k_\mathfrak{a}$ ($k_\mathfrak{b}$).  Since $C$ is fixed and since $R_k(t;k)=(\mathfrak{a}(t)+\mathfrak{b}(t)-2k)/(2R(t;k))=-(U(t)+2k)/(2R(t;k))$ by differentiation of $R(t;k)^2$, 
\begin{equation}
\tau'(k)=-\oint_C R(t;k)\,dt -\frac{1}{4}\oint_C\frac{(U(t)-2k)(U(t)+2k)\,dt}{R(t;k)}.
\label{eq:tau-prime-contour}
\end{equation}
Now near $t_\mathfrak{a}$, the function $R(t;k_\mathfrak{a})^2$ has the Taylor expansion
$R(t;k_\mathfrak{a})^2=-\tfrac{1}{2}\mathfrak{a}''(t_\mathfrak{a})(\mathfrak{b}(t_\mathfrak{a})-k_\mathfrak{a})(t-t_\mathfrak{a})^2 + \mathcal{O}((t-t_\mathfrak{a})^3)$, and since $R(t;k)$ is positive imaginary to the right of $t_+(k)$ and negative imaginary to the left of $t_-(k)$, it follows that when $k=k_\mathfrak{a}$ so that the branch cut collapses to a point $t=t_\mathfrak{a}$, we have
\begin{equation}
R(t;k_\mathfrak{a})=i\sqrt{\tfrac{1}{2}\mathfrak{a}''(t_\mathfrak{a})(\mathfrak{b}(t_\mathfrak{a})-k_\mathfrak{a})} (t-t_\mathfrak{a}) + \mathcal{O}((t-\mathfrak{a})^2),\quad t\to t_\mathfrak{a}.  
\label{eq:R-a}
\end{equation}
Similar arguments show that 
\begin{equation}
R(t;k_\mathfrak{b})=i\sqrt{-\tfrac{1}{2}\mathfrak{b}''(t_\mathfrak{b})(k_\mathfrak{b}-\mathfrak{a}(t_\mathfrak{b}))}(t-t_\mathfrak{b})+\mathcal{O}((t-t_\mathfrak{b})^2),\quad t\to t_\mathfrak{b}.
\label{eq:R-b}
\end{equation}
In both cases the indicated square roots are positive numbers.  In particular, since $R(t;k_\mathfrak{a})$ and $R(t;k_\mathfrak{b})$ are analytic functions of $t$ within $C$ it follows from \eqref{eq:tau-contour} that $\tau(k_\mathfrak{a})=\tau(k_\mathfrak{b})=0$.  We may now use \eqref{eq:R-a}--\eqref{eq:R-b} in \eqref{eq:tau-prime-contour} to compute $\tau'(k_\mathfrak{a})$ and $\tau'(k_\mathfrak{b})$ by residues:
\begin{equation}
\tau'(k_\mathfrak{a})=-\frac{\pi}{2}\frac{(U(t_\mathfrak{a})-2k_\mathfrak{a})(U(t_\mathfrak{a})+2k_\mathfrak{a})}{\sqrt{\tfrac{1}{2}\mathfrak{a}''(t_\mathfrak{a})(\mathfrak{b}(t_\mathfrak{a})-k_\mathfrak{a})}}=\frac{\pi}{2}\frac{(U(t_\mathfrak{a})-2k_\mathfrak{a})\sqrt{\mathfrak{b}(t_\mathfrak{a})-k_\mathfrak{a}}}{\sqrt{\tfrac{1}{2}\mathfrak{a}''(t_\mathfrak{a})}}>0
\end{equation}
and
\begin{equation}
\tau'(k_\mathfrak{b})=-\frac{\pi}{2}\frac{(U(t_\mathfrak{b})-2k_\mathfrak{b})(U(t_\mathfrak{b})+2k_\mathfrak{b})}{\sqrt{-\tfrac{1}{2}\mathfrak{b}''(t_\mathfrak{b})(k_\mathfrak{b}-\mathfrak{a}(t_\mathfrak{b}))}}=-\frac{\pi}{2}\frac{(U(t_\mathfrak{b})-2k_\mathfrak{b})\sqrt{k_\mathfrak{b}-\mathfrak{a}(t_\mathfrak{b})}}{\sqrt{-\tfrac{1}{2}\mathfrak{b}''(t_\mathfrak{b})}}<0.
\end{equation}
This completes the proof.
\end{proof}

The nondegeneracy of the extrema of $\mathfrak{a}$ and $\mathfrak{b}$ also leads to the following result.
\begin{lemma}
If the functions $U$ and $H$ satisfy Assumption~\ref{assumption:data},
then $\Phi:(k_\mathfrak{a},k_\mathfrak{b})\to\mathbb{R}$ has an analytic continuation into the complex plane from a right neighborhood of $k_\mathfrak{a}$ and from a left neighborhood of $k_\mathfrak{b}$ and 
\begin{equation}
\Phi(k)=\Phi(k_\mathfrak{a})+C_\mathfrak{a}(k-k_\mathfrak{a})\log(k-k_\mathfrak{a})+\mathcal{O}(k-k_\mathfrak{a}),\quad k\to k_\mathfrak{a},\quad
C_\mathfrak{a}:=\frac{1}{4}\frac{(U(t_\mathfrak{a})-2k_\mathfrak{a})\sqrt{\mathfrak{b}(t_\mathfrak{a})-k_\mathfrak{a}}}{\sqrt{\tfrac{1}{2}\mathfrak{a}''(t_\mathfrak{a})}},
\label{eq:Phi-k-a}
\end{equation}
while
\begin{equation}
\Phi(k)=\Phi(k_\mathfrak{b}) +C_\mathfrak{b}(k-k_\mathfrak{b})\log(k_\mathfrak{b}-k) +\mathcal{O}(k-k_\mathfrak{b}),\quad k\to k_\mathfrak{b},\quad C_\mathfrak{b}:=
\frac{1}{4}\frac{(U(t_\mathfrak{b})-2k_\mathfrak{b})\sqrt{k_\mathfrak{b}-\mathfrak{a}(t_\mathfrak{b})}}{\sqrt{-\tfrac{1}{2}\mathfrak{b}''(t_\mathfrak{b})}}.
\label{eq:Phi-k-b}
\end{equation}
\label{lemma-Phi-endpoints}
\end{lemma}
\begin{proof}
This can be shown with the help of the contour integral formula \eqref{eq:Phi-contour}.  In particular, note that $\Phi$ is continuous in the limits $k\downarrow k_\mathfrak{a}$ and $k\uparrow k_\mathfrak{b}$, but formulae \eqref{eq:Phi-k-a} and \eqref{eq:Phi-k-b} hold in full neighborhoods of the indicated limit point with only local branch cuts of the logarithms omitted.  Note that $2\pi C_\mathfrak{a}=\tau'(k_\mathfrak{a})$ and $2\pi C_\mathfrak{b}=-\tau'(k_\mathfrak{b})$ so both $C_\mathfrak{a}$ and $C_\mathfrak{b}$ are positive.
\end{proof}

Combining Lemma~\ref{lemma-Phi-endpoints} with Lemma~\ref{lemma-tau} yields the following result.
\begin{lemma}
The function $\tilde{\Gamma}(k)=Y^\eps(k)e^{-2\myi\Phi(k)/\eps}$ has an analytic continuation into the complex plane from right and left neighborhoods of $k_\mathfrak{a}$ and $k_\mathfrak{b}$, respectively, and for each $\theta\in (0,\pi/2)$ and each $\delta>0$ sufficiently small, $\tilde{\Gamma}(k_\mathfrak{a}+re^{i\theta})=\mathcal{O}((\log(\eps^{-1}))^{-1/2})$ and $\tilde{\Gamma}(k_\mathfrak{b}-re^{-i\theta})=\mathcal{O}((\log(\eps^{-1}))^{-1/2})$ both hold in the limit $\eps\to 0$ with $\eps>0$, uniformly for $0\le r\le \delta$.
(Of course by Schwarz reflection $\tilde{\Gamma}(k^*)^*$ satisfies similar estimates along segments in the lower half-plane with endpoints $k_\mathfrak{a}$ and $k_\mathfrak{b}$.)
\label{lemma-tilde-Gamma-endpoints}
\end{lemma}
\begin{proof}
This follows from the fact that $Y^\eps(k)$ has an analytic continuation satisfying $Y^\eps(k)=\mathcal{O}(\eps^{-1/2}\tau(k)^{1/2})$ near $k_\mathfrak{a}$ and $k_\mathfrak{b}$ as long as $\Re\{\tau(k)\}>0$.  By Lemma~\ref{lemma-tau} this estimate holds locally near $k_\mathfrak{a}$ or $k_\mathfrak{b}$ as long as $\Re\{k-k_\mathfrak{a}\}>0$ or $\Re\{k_\mathfrak{b}-k\}>0$ holds, respectively, and in each case we may replace $\tau(k)$ in the estimate by either $|k-k_\mathfrak{a}|$ or $|k-k_\mathfrak{b}|$.

Using Lemma~\ref{lemma-Phi-endpoints} shows that the problem boils down to estimating functions of a real variable, $x\in (0,1)$, having the form
\begin{equation}
f^\eps(x):=\frac{x^{1/2}}{\eps^{1/2}}e^{x\log(x)/\eps},\quad x\in (0,1),\quad \eps>0.
\end{equation}
This function vanishes as $x\downarrow 0$ for each $\eps>0$, and it has two critical points for $x>0$, only the smaller of the two being relevant for bounded $x$.  This critical point is the global maximizer on $(0,1)$ and it satisfies $x=x_c(\eps)\sim\eps/(2\log(\eps^{-1}))$ as $\eps\downarrow 0$.  It then follows that $f^\eps(x_c(\eps))=\mathcal{O}((\log(\eps^{-1}))^{-1/2})$ as $\eps\downarrow 0$ by direct calculation.
\end{proof}
The intuition behind this result is that while the exponential decay of $e^{-2\myi\Phi(k)/\eps}$ in the upper half-plane is not uniform near $k_\mathfrak{a}$ or $k_\mathfrak{b}$ (and in fact there is no decay at all exactly at these two points), the factor $Y^\eps(k)$ vanishes at these points, with the result being that the product is locally uniformly small with $\eps>0$, albeit exhibiting a very slow rate of decay to zero.

\subsubsection{Proof of Theorem~\ref{theorem:initial-condition-small}}
\label{sec:initial-condition-small-proof}
We now give the proof of Theorem~\ref{theorem:initial-condition-small}.
The strategy is to open a single lens about the entire interval $(k_\mathfrak{a},k_\mathfrak{b})$ based upon the natural factorization of the jump matrix:
\begin{equation}
\begin{bmatrix}
1-|\tilde{\Gamma}(k)|^2 & -\tilde{\Gamma}(k)^*e^{-2\myi\theta(k;x,t)/\eps}\\
\tilde{\Gamma}(k)e^{2\myi\theta(k;x,t)/\eps} & 1\end{bmatrix}=
\begin{bmatrix}1 & -\tilde{\Gamma}(k)^*e^{-2\myi\theta(k;x,t)/\eps}\\0 & 1\end{bmatrix}
\begin{bmatrix}1 & 0\\\tilde{\Gamma}(k)e^{2\myi\theta(k;x,t)/\eps} & 1\end{bmatrix},\quad k_\mathfrak{a}<k<k_\mathfrak{b}.
\end{equation}
However, technical modifications of the steepest descent method will be required because $\tau$ has no analytic continuation from the real axis near the points $k_0$ and $k_\infty$, and $\Phi$ fails to be analytic at $k=k_0$.  

We will in particular need a way to extend the three-times differentiable but non-analytic function $\Phi(k)$ into the complex plane from a real neighborhood of $k=k_0$.  Let $\Re\{k\}=\kr$ and $\Im\{k\}=\ki$ 
denote the real and imaginary parts of the complex variable $k$.  We follow the approach of \cite{McLaughlinM08} and first define a nonanalytic extension of $\Phi(\kr)$ by the formula
\begin{equation}
\hat{\Phi}_0(\kr,\ki):=\Phi(\kr)+\myi \ki\Phi'(\kr) +\frac{1}{2}(\myi \ki)^2\Phi''(\kr).
\end{equation}
Note that $\hat{\Phi}_0(\kr,\ki)$ is nearly analytic close to the real axis $\ki=0$ in the sense that
\begin{equation}
\dbar\hat{\Phi}_0(\kr,\ki):=\frac{1}{2}\left(\frac{\partial}{\partial \kr}+\myi\frac{\partial}{\partial \ki}\right)\hat{\Phi}_0(\kr,\ki)=\frac{1}{4}(\myi \ki)^2\Phi'''(\kr) = \mathcal{O}(\ki^2)
\end{equation}
according to Lemma~\ref{lemma-Phi-interior}.  Also according to Lemma~\ref{lemma-Phi-interior}, we may identify $\Phi(k)$ with two distinct analytic functions, $\Phi_\mathfrak{a}(k)$ denoting the analytic continuation of $\Phi(k)$ from $(k_\mathfrak{a},k_0)$ and $\Phi_\mathfrak{b}(k)$ denoting the analytic continuation of $\Phi(k)$ from $(k_0,k_\mathfrak{b})$.  Note that for each fixed $\kr\in (k_\mathfrak{a},k_0)$ we have $\hat{\Phi}_0(\kr,\ki)-\Phi_\mathfrak{a}(k)=\mathcal{O}(\ki^3)$ and
for each fixed $\kr\in (k_0,k_\mathfrak{b})$ we have $\hat{\Phi}_0(\kr,\ki)-\Phi_\mathfrak{b}(k)=\mathcal{O}(\ki^3)$, with the error terms being uniform by Taylor's theorem for $\kr$ in compact subsets of $(k_\mathfrak{a},k_\mathfrak{b})$ bounded away from $k_0$.  Now let $\delta>0$ be so small that $[k_0-2\delta,k_0+2\delta]\subset (k_\mathfrak{a},k_\mathfrak{b})$, and define the smooth bump function $\mathcal{B}:\mathbb{R}\to [0,1]$ such that
$\mathcal{B}$ is of class $C^\infty$ and 
\begin{equation}
\mathcal{B}(u)=\begin{cases} 1,&\quad |u-k_0|<\delta\\
0,&\quad |u-k_0|>2\delta.
\end{cases}
\end{equation}
Then we define an extension of $\Phi$ into the upper half-plane near $(k_\mathfrak{a},k_\mathfrak{b})$ as follows:
\begin{equation}
\hat{\Phi}(\kr,\ki):=\begin{cases}
\mathcal{B}(\kr)\hat{\Phi}_0(\kr,\ki) +(1-\mathcal{B}(\kr))\Phi_\mathfrak{a}(k),&\quad k_\mathfrak{a}<\kr\le k_0\\
\mathcal{B}(\kr)\hat{\Phi}_0(\kr,\ki)+(1-\mathcal{B}(\kr))\Phi_\mathfrak{b}(k),&\quad k_0\le \kr<k_\mathfrak{b}.
\end{cases}
\label{eq:hat-Phi-define}
\end{equation} 
By direct calculation,
\begin{equation}
\dbar\hat{\Phi}(\kr,\ki)=\begin{cases}
\mathcal{B}(\kr)\dbar\hat{\Phi}_0(\kr,\ki)  + \dbar\mathcal{B}(\kr)\cdot(\hat{\Phi}_0(\kr,\ki)-\Phi_{\mathfrak{a}}(k)),&\quad k_\mathfrak{a}<\kr\le k_0\\
\mathcal{B}(\kr)\dbar\hat{\Phi}_0(\kr,\ki)  + \dbar\mathcal{B}(\kr)\cdot(\hat{\Phi}_0(\kr,\ki)-\Phi_{\mathfrak{b}}(k)),&\quad k_0\le \kr<k_\mathfrak{b}.
\end{cases}
\end{equation}
It follows that $\dbar\hat{\Phi}(\kr,\ki)$ is $\mathcal{O}(\ki^2)$ uniformly for $\kr$ in compact subsets of $(k_\mathfrak{a},k_\mathfrak{b})$.  Note that $\hat{\Phi}(\kr,-\ki)=\hat{\Phi}(\kr,\ki)^*$ holds, as the nonanalytic analogue of the Schwarz reflection symmetry $\Phi_{\mathfrak{a},\mathfrak{b}}(k^*)=\Phi_{\mathfrak{a},\mathfrak{b}}(k)^*$ of the real analytic functions $\Phi_{\mathfrak{a},\mathfrak{b}}$.

Based on the extension $\hat{\Phi}$ we now make an explicit transformation of $\tilde{\mathbf{M}}(k)$ to open lenses about $(k_\mathfrak{a},k_\mathfrak{b})$.  Consider the domains illustrated in Figure~\ref{fig:Vacuum-Lenses}.
\begin{figure}[h]
\includegraphics{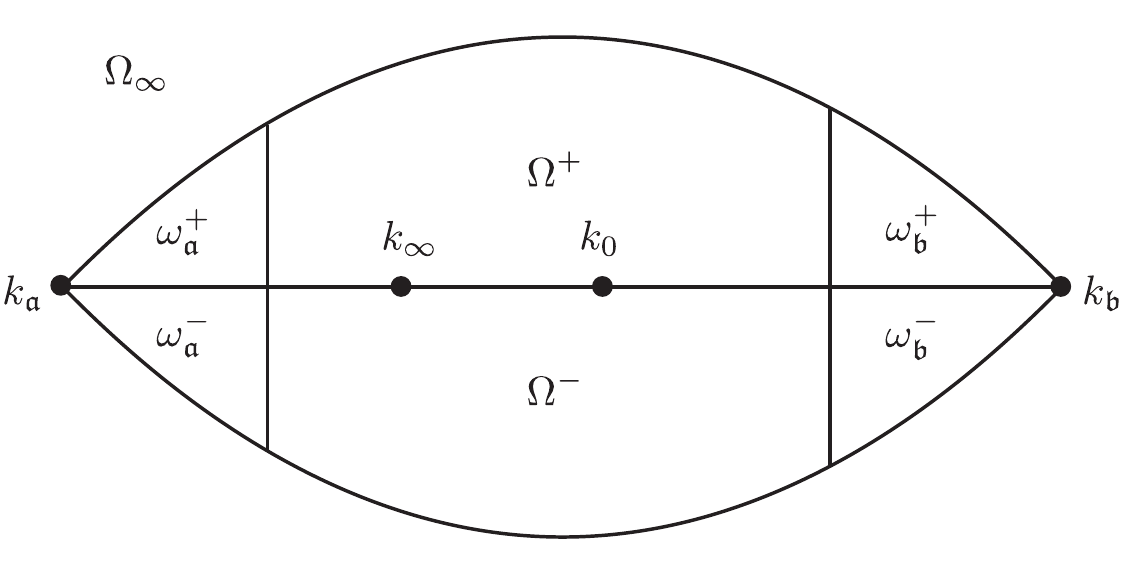}
\caption{The lens domains about the interval $[k_\mathfrak{a},k_\mathfrak{b}]$.}
\label{fig:Vacuum-Lenses}
\end{figure}
We make the following explicit transformations, defining a new matrix unknown $\mathbf{O}(\kr,\ki)$:
\begin{equation}
\mathbf{O}(\kr,\ki):=\tilde{\mathbf{M}}(k)\begin{bmatrix}1 & 0\\
-e^{2\myi(\theta(k;x,t)-\hat{\Phi}(\kr,\ki))/\eps} & 1\end{bmatrix},\quad k\in\Omega^+,
\label{eq:M-O-tilde-Omega-plus}
\end{equation}
\begin{equation}
\mathbf{O}(\kr,\ki):=\tilde{\mathbf{M}}(k)\begin{bmatrix} 1 & -e^{2\myi(\hat{\Phi}(\kr,\ki)-\theta(k;x,t))/\eps}\\0 & 1\end{bmatrix},\quad k\in\Omega^-,
\label{eq:M-O-tilde-Omega-minus}
\end{equation}
\begin{equation}
\mathbf{O}(\kr,\ki):=\tilde{\mathbf{M}}(k)\begin{bmatrix}1 & 0\\
-Y_\mathfrak{a}^\eps(k)e^{2\myi(\theta(k;x,t)-\Phi_\mathfrak{a}(k))/\eps} & 1\end{bmatrix},\quad k\in\omega_\mathfrak{a}^+,
\end{equation}
\begin{equation}
\mathbf{O}(\kr,\ki):=\tilde{\mathbf{M}}(k)\begin{bmatrix}1 & 0\\
-Y_\mathfrak{b}^\eps(k)e^{2\myi(\theta(k;x,t)-\Phi_\mathfrak{b}(k))/\eps} & 1\end{bmatrix},\quad k\in\omega_\mathfrak{b}^+,
\end{equation}
\begin{equation}
\mathbf{O}(\kr,\ki):=\tilde{\mathbf{M}}(k)\begin{bmatrix}1 & -Y_\mathfrak{a}^\eps(k)
e^{2\myi (\Phi_\mathfrak{a}(k)-\theta(k;x,t))/\eps}\\ 0 & 1\end{bmatrix},\quad
k\in\omega_\mathfrak{a}^-,
\end{equation}
\begin{equation}
\mathbf{O}(\kr,\ki):=\tilde{\mathbf{M}}(k)\begin{bmatrix}1 & -Y_\mathfrak{b}^\eps(k)
e^{2\myi (\Phi_\mathfrak{b}(k)-\theta(k;x,t))/\eps}\\ 0 & 1\end{bmatrix},\quad
k\in\omega_\mathfrak{b}^-,
\label{eq:M-O-tilde-last}
\end{equation}
and in the  unbounded domain $\Omega_\infty$ we set $\mathbf{O}(\kr,\ki):=\tilde{\mathbf{M}}(k)$.  Here the notation $Y_{\mathfrak{a},\mathfrak{b}}^\eps(k)$ refers to the two distinct analytic functions $\tau(k)=\tau_\mathfrak{a}(k)$ defined near $k=k_\mathfrak{a}$ and $\tau(k)=\tau_\mathfrak{b}(k)$ defined near $k=k_\mathfrak{b}$ (see Lemma~\ref{lemma:tau-continuous-positive} and Lemma~\ref{lemma-tau}).  Indeed, the domains $\omega_\mathfrak{a}^\pm$ and $\omega_\mathfrak{b}^\pm$ are chosen small enough to exclude both $k_0$ and $k_\infty$, the two points of nonanalyticity of $\tau:(k_\mathfrak{a},k_\mathfrak{b})\to\mathbb{R}_+$.  By making $\delta>0$ smaller if necessary, we also ensure that these domains have no intersection with the vertical strip $|\kr-k_0|\le 2\delta$, in which $\mathbf{O}(\kr,\ki)$ fails to be analytic.
The matrix $\mathbf{O}(\kr,\ki)$ has jump discontinuities across a contour $\Sigma$ illustrated in Figure~\ref{fig:Vacuum-Contour}.  Note that the real segment common to the boundary of the domains $\omega_\mathfrak{a}^\pm$ and the real segment common to the boundary of the domains $\omega_\mathfrak{b}^\pm$ are not part of the jump contour $\Sigma$ as it is easy to check that $\mathbf{O}(\kr,\ki)$ is continuous across these segments.
\begin{figure}[h]
\includegraphics{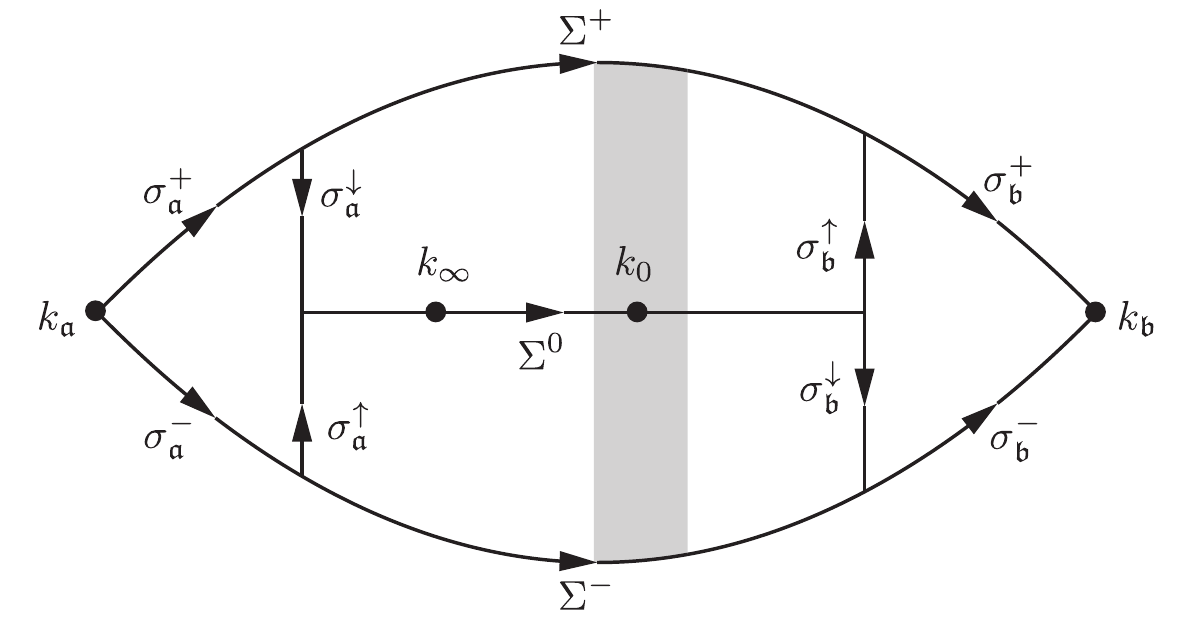}
\caption{The jump contour for the matrix function $\mathbf{O}(\kr,\ki)$ is the union of eleven oriented arcs, labeled as shown.  The domain in which $\mathbf{O}(\kr,\ki)$ fails to be analytic is shaded.}
\label{fig:Vacuum-Contour}
\end{figure}

We claim that the matrix $\mathbf{O}(\kr,\ki)$ satisfies the conditions of a hybrid Riemann-Hilbert-$\dbar$ problem of small-norm type.  This problem is the following.
\begin{rhpdbar}
Find a $2\times 2$ matrix $\mathbf{O}(\kr,\ki)$ with the following properties:
\begin{itemize}
\item[]\textbf{Continuity:}  $\mathbf{O}$ is continuous in each connected component of $\mathbb{R}^2\setminus\Sigma$ and takes continuous boundary values $\mathbf{O}_+$ and $\mathbf{O}_-$ on each oriented arc of $\Sigma$ from the left and right, respectively.
\item[]\textbf{Jump Condition:}  On each oriented arc of $\Sigma$ the boundary values are related by 
$\mathbf{O}_+(\kr,\ki)=\mathbf{O}_-(\kr,\ki)\mathbf{J}_0(\kr,\ki)$ (see below for the explicit definition of $\mathbf{J}_0(\kr,\ki)$).
\item[]\textbf{Deviation from Analyticity:}  In each connected component of $\mathbb{R}^2\setminus\Sigma$ the matrix $\mathbf{O}$ satisfies $\dbar\mathbf{O}(\kr,\ki)=\mathbf{O}(\kr,\ki)\mathbf{W}(\kr,\ki)$ (see below for the explicit definition of $\mathbf{W}(\kr,\ki)$).
\item[]\textbf{Normalization:}  $\mathbf{O}(\kr,\ki)\to\mathbb{I}$ as $(\kr,\ki)\to\infty$ in $\mathbb{R}^2$.
\end{itemize}
\label{rhpdbar-small-norm}
\end{rhpdbar}
The jump matrix $\mathbf{J}_0(\kr,\ki)$ is defined explicitly on each arc of $\Sigma$
simply by using the definitions \eqref{eq:M-O-tilde-Omega-plus}--\eqref{eq:M-O-tilde-last} and the 
jump condition satisfied by $\tilde{\mathbf{M}}$ across the segment $(k_\mathfrak{a},k_\mathfrak{b})$ according to Riemann-Hilbert Problem~\ref{rhp-M-tilde}.  The result is the following:  
\begin{equation}
\mathbf{J}_0(\kr,\ki):=\begin{bmatrix}1 & 0\\Y_{\mathfrak{a},\mathfrak{b}}^\eps(k)
e^{2\myi(\theta(k;x,t)-\Phi_{\mathfrak{a},\mathfrak{b}}(k))/\eps} & 1\end{bmatrix},\quad k\in\sigma_{\mathfrak{a},\mathfrak{b}}^+,
\label{eq:jump-Sigma-a-b-plus}
\end{equation}
\begin{equation}
\mathbf{J}_0(\kr,\ki):=\begin{bmatrix}1 & -Y_{\mathfrak{a},\mathfrak{b}}^\eps(k)e^{2\myi(\Phi_{\mathfrak{a},\mathfrak{b}}(k)-\theta(k;x,t))/\eps}\\0 & 1\end{bmatrix},\quad k\in\sigma_{\mathfrak{a},\mathfrak{b}}^-,
\label{eq:jump-Sigma-a-b-minus}
\end{equation}
\begin{equation}
\mathbf{J}_0(\kr,\ki):=\begin{bmatrix}1 & 0\\(Y_{\mathfrak{a},\mathfrak{b}}^\eps(k)-1)e^{2\myi(\theta(k;x,t)-\Phi_{\mathfrak{a},\mathfrak{b}}(k))/\eps} & 1\end{bmatrix},\quad k\in\sigma_\mathfrak{a}^\downarrow,\sigma_\mathfrak{b}^\uparrow,
\label{eq:Sigma-up-down-jump-1}
\end{equation}
\begin{equation}
\mathbf{J}_0(\kr,\ki):=\begin{bmatrix}
1 & (1-Y_{\mathfrak{a},\mathfrak{b}}^\eps(k))e^{2\myi(\Phi_{\mathfrak{a},\mathfrak{b}}(k)-\theta(k;x,t))/\eps}\\0 & 1\end{bmatrix},\quad k\in\sigma_\mathfrak{a}^\uparrow,\sigma_\mathfrak{b}^\downarrow,
\label{eq:Sigma-up-down-jump-2}
\end{equation}
\begin{equation}
\mathbf{J}_0(\kr,\ki):=\begin{bmatrix}1 & 0\\ e^{2\myi(\theta(k;x,t)-\hat{\Phi}(\kr,\ki))/\eps} & 1\end{bmatrix},\quad k\in\Sigma^+,
\label{eq:Sigma-plus-jump}
\end{equation}
\begin{equation}
\mathbf{J}_0(\kr,\ki):=\begin{bmatrix}1 & -e^{2\myi(\hat{\Phi}(\kr,\ki)-\theta(k;x,t))/\eps}\\
0 & 1\end{bmatrix},\quad k\in\Sigma^-,
\label{eq:Sigma-minus-jump}
\end{equation}
and, finally, using the fact that $\hat{\Phi}(\kr,0)=\Phi(\kr)$, 
\begin{equation}
\mathbf{J}_0(\kr,\ki):=\begin{bmatrix}
1 & (1-Y^\eps(\kr))e^{2\myi(\Phi(\kr)-\theta(\kr;x,t))/\eps}\\0 & 1\end{bmatrix}
\begin{bmatrix}1 & 0\\(Y^\eps(\kr)-1)e^{2\myi(\theta(\kr;x,t)-\Phi(\kr))/\eps} & 1\end{bmatrix},\quad k\in\Sigma^0.
\label{eq:Sigma0-jump}
\end{equation}
Consider the jump matrix defined by \eqref{eq:jump-Sigma-a-b-plus}--\eqref{eq:jump-Sigma-a-b-minus}.  Note that since $\Im\{\theta(k;x,t)\}=o(\Im\{\Phi_{\mathfrak{a},\mathfrak{b}}\})$ near $k=k_{\mathfrak{a},\mathfrak{b}}$ according to Lemma~\ref{lemma-Phi-endpoints}, it follows from Lemma~\ref{lemma-tilde-Gamma-endpoints} that for all $(x,t)\in\mathbb{R}^2$, $\mathbf{J}_0-\mathbb{I}=\mathcal{O}((\log(\eps^{-1}))^{-1/2})$ holds uniformly 
on the four contour arcs $\sigma_{\mathfrak{a},\mathfrak{b}}^{+,-}$, provided the lens opens with an acute nonzero angle and the vertical contours $\sigma_{\mathfrak{a},\mathfrak{b}}^{\uparrow,\downarrow}$ are placed close enough to the respective endpoint $k_{\mathfrak{a},\mathfrak{b}}$.
Similarly, since $\tau(k)>0$ is bounded away from zero while $\Phi(k)$ and $\theta(k;x,t)$ are real for $k\in\Sigma^0$, it is easy to see from \eqref{eq:Sigma0-jump} that $\mathbf{J}_0-\mathbb{I}$ is uniformly exponentially small on $\Sigma^0$ in the limit $\eps\downarrow 0$, again independently of $(x,t)\in\mathbb{R}^2$.  Controlling the jump matrix $\mathbf{J}_0$ on the remaining arcs of $\Sigma$ requires conditions on $(x,t)\in\mathbb{R}^2$ as we will see below.

The matrix $\mathbf{W}(\kr,\ki)$ is defined explicitly by applying the operator $\dbar$ to the formulae
\eqref{eq:M-O-tilde-Omega-plus}--\eqref{eq:M-O-tilde-Omega-minus}.  The result is:
\begin{equation}
\mathbf{W}(\kr,\ki):=\begin{bmatrix}0 & 0\\2\myi\eps^{-1}\dbar\hat{\Phi}(\kr,\ki)\cdot e^{2\myi(\theta(k;x,t)-\hat{\Phi}(\kr,\ki))/\eps} & 0\end{bmatrix},\quad k\in\Omega^+,
\label{eq:W-Omega-plus}
\end{equation}
and
\begin{equation}
\mathbf{W}(\kr,\ki):=\begin{bmatrix}0 & -2\myi\eps^{-1}\dbar\hat{\Phi}(\kr,\ki)\cdot e^{2\myi(\hat{\Phi}(\kr,\ki)-\theta(k;x,t))/\eps}\\0 & 0\end{bmatrix},\quad k\in\Omega^-,
\label{eq:W-Omega-minus}
\end{equation}
and in all other connected components of $\mathbb{R}^2\setminus\Sigma$, $\mathbf{W}(\kr,\ki):=\mathbf{0}$.  In particular, $\mathbf{W}$ has compact support.

Riemann-Hilbert-$\dbar$ Problem~\ref{rhpdbar-small-norm} is of small-norm type in the sense
that, as a consequence of the conditions $t=0$ and $x>0$, the jump matrix $\mathbf{J}_0$ defined on the compact contour $\Sigma$ satisfies $\|\mathbf{J}_0-\mathbb{I}\|_{L^\infty(\Sigma)}=\mathcal{O}((\log(\eps^{-1}))^{-1/2})$ as $\eps\to 0$, and at the same time the matrix $\mathbf{W}$ defined on $\mathbb{R}^2\setminus\Sigma$
satisfies $\|\mathbf{W}\|_{L^\infty(\mathbb{R}^2\setminus\Sigma)}=\mathcal{O}(\eps)$ as $\eps\to 0$.
Indeed,
\begin{equation}
\theta'(k;x,0)-\Phi'(k)=x-\Phi'(k)\ge x>0,\quad k_\mathfrak{a}<k<k_\mathfrak{b},
\label{eq:basic-inequality-1}
\end{equation}
according to Lemma~\ref{lemma-Phi-interior}.  This immediately implies, by the Cauchy-Riemann equations applied to the real analytic functions $\theta(k;x,t)-\Phi_{\mathfrak{a},\mathfrak{b}}(k)$ near the real $k$-axis, that the exponential factors appearing in the formulae \eqref{eq:Sigma-up-down-jump-1}--\eqref{eq:Sigma-up-down-jump-2} are bounded in modulus by $1$ provided the lens is sufficiently thin (independent of $\eps$).  Since the factors $Y^\eps_{\mathfrak{a},\mathfrak{b}}(k)-1$ are exponentially small as they are on the real axis if the lens is thin enough, we conclude that for $x>0$ and $t=0$, $\mathbf{J}_0-\mathbb{I}$ is uniformly exponentially small for $k\in\sigma_\mathfrak{a}^\downarrow\cup\sigma_\mathfrak{b}^\uparrow\cup\sigma_\mathfrak{a}^\uparrow\cup\sigma_\mathfrak{b}^\downarrow$.  Finally, from \eqref{eq:hat-Phi-define} and the Cauchy-Riemann equations, we see that
\begin{equation}
\Im\{\hat{\Phi}(\kr,\ki)\}=\Phi'(\kr)\ki + \mathcal{O}(\ki^2),\quad \ki\to 0,
\end{equation}
a fact which, taken together with \eqref{eq:Sigma-plus-jump}--\eqref{eq:Sigma-minus-jump} for $x>0$ and $t=0$ shows that $\mathbf{J}_0-\mathbb{I}$ is uniformly exponentially small in the limit $\eps\downarrow 0$ for $k\in\Sigma^\pm$.  Combining these estimates yields the claimed $L^\infty(\Sigma)$ bound for $\mathbf{J}_0-\mathbb{I}$.  Similarly, for $t=0$ and for each $x>0$ we obtain exponential decay of the exponential factors in \eqref{eq:W-Omega-plus}--\eqref{eq:W-Omega-minus}, so combining this fact with the fact that $\dbar\hat{\Phi}(\kr,\ki)=\mathcal{O}(\ki^2)$ holds for $k\in\Omega^+\cup\Omega^-$, we obtain ($\|\cdot\|$ denotes any matrix norm)
\begin{equation}
\|\mathbf{W}(\kr,\ki)\|\le K\frac{\ki^2}{\eps}e^{-C|\ki |/\eps},\quad k\in\Omega^+\cup\Omega^-,
\end{equation}
where $C=C(x_0)>0$ for $x>x_0$ and $K>0$ is independent of $x$.  The claimed $L^\infty(\mathbb{R}^2\setminus\Sigma)$ estimate of $\mathbf{W}$ follows immediately because $r^2e^{-Cr}$ is uniformly bounded for all $r>0$ (recall $\mathbf{W}(\kr,\ki)=0$ for $k\not\in \Omega^+\cup\Omega^-$).

One makes use of the estimates $\|\mathbf{J}_0-\mathbb{I}\|_{L^\infty(\Sigma)}=\mathcal{O}((\log(\eps^{-1}))^{-1/2})$ and $\|\mathbf{W}\|_{L^\infty(\mathbb{R}^2\setminus\Sigma)}=\mathcal{O}(\eps)$ as follows.  The strategy is to solve the hybrid Riemann-Hilbert-$\dbar$ problem by first solving the ``$\dbar$ part'' and then using the result to obtain a standard Riemann-Hilbert problem of $L^2$ small-norm type.  We therefore consider the following auxiliary $\dbar$ problem:
\begin{dbarprob}
Find a $2\times 2$ matrix $\dot{\mathbf{O}}(\kr,\ki)$ with the following properties:
\begin{itemize}
\item[]\textbf{Continuity:}  $\dot{\mathbf{O}}:\mathbb{R}^2\to\mathbb{C}^{2\times 2}$ is continuous.
\item[]\textbf{Deviation from Analyticity:}  
$\dbar\dot{\mathbf{O}}(\kr,\ki)=\dot{\mathbf{O}}(\kr,\ki)\mathbf{W}(\kr,\ki)$ holds in the distributional sense.
\item[]\textbf{Normalization:}  $\dot{\mathbf{O}}(\kr,\ki)\to\mathbb{I}$ as $(\kr,\ki)\to\infty$ in $\mathbb{R}^2$.
\end{itemize}
\label{dbar-1}
\end{dbarprob}
We solve for $\dot{\mathbf{O}}$ by setting up an integral equation involving the solid Cauchy transform:
\begin{equation}
\dot{\mathbf{O}}(\kr,\ki)=\mathbb{I}+\mathcal{K}\dot{\mathbf{O}}(\kr,\ki),\quad
\mathcal{K}\mathbf{F}(\kr,\ki):=-\frac{1}{\pi}\iint_{\Omega^+\cup\Omega^-}
\frac{\mathbf{F}(\kr',\ki')\mathbf{W}(\kr',\ki')\,dA(\kr',\ki')}{k'-k},
\label{eq:dbar-integral-equation}
\end{equation}
where $dA(\kr,\ki)=d\kr\,d\ki$ denotes the area element.  It can be shown that the integral equation \eqref{eq:dbar-integral-equation} is in fact equivalent to the formulation of $\dbar$ Problem~\ref{dbar-1}.
The operator norm of $\mathcal{K}$ acting on $L^\infty(\mathbb{R}^2)$ is easily estimated based on the fact that Cauchy kernel is locally integrable in two dimensions.  Thus,
\begin{equation}
\|\mathcal{K}\|_{L^\infty(\mathbb{R}^2)\circlearrowleft}\le\frac{1}{\pi}\|\mathbf{W}\|_{L^\infty(\mathbb{R}^2)}
\sup_{(\kr,\ki)\in\mathbb{R}^2}\iint_{\Omega^+\cup\Omega^-}\frac{dA(\kr',\ki')}{|k'-k|}.
\end{equation}
The latter supremum is finite and depends only on the bounded domain $\Omega^+\cup\Omega^-$ as the double integral is continuous and decays as $|k|^{-1}$ as $(\kr,\ki)\to\infty$ in $\mathbb{R}^2$.
From the bound $\|\mathbf{W}\|_{L^\infty(\mathbb{R}^2)}=\mathcal{O}(\eps)$ it follows also that 
$\|\mathcal{K}\|_{L^\infty(\mathbb{R}^2)\circlearrowleft}=\mathcal{O}(\eps)$ as $\eps\downarrow 0$.
It follows that for $\eps>0$ sufficiently small, the operator $1-\mathcal{K}$ is invertible by Neumann series convergent in $L^\infty(\mathbb{R}^2)$.  Since each term of the series is continuous, so is the sum $\dot{\mathbf{O}}(\kr,\ki)$ of the series for $(1-\mathcal{K})^{-1}$ applied to $\mathbb{I}$.  Furthermore, we obtain the important estimate $\|\dot{\mathbf{O}}-\mathbb{I}\|_{L^\infty(\mathbb{R}^2)}=\mathcal{O}(\eps)$, which in particular implies that $\dot{\mathbf{O}}(\kr,\ki)^{-1}$ exists for sufficiently small $\eps$ as a continuous function on $\mathbb{R}^2$ that satisfies $\|\dot{\mathbf{O}}^{-1}-\mathbb{I}\|_{L^\infty(\mathbb{R}^2)}=\mathcal{O}(\eps)$.  Finally, compact support of $\mathbf{W}$ ensures that $\dot{\mathbf{O}}(\kr,\ki)$ has a convergent Laurent series in descending powers of $k$ for $|k|$ sufficiently large, and in particular we obtain $\dot{\mathbf{O}}(\kr,\ki)=\mathbb{I}+k^{-1}\dot{\mathbf{O}}_1 + \mathcal{O}(k^{-2})$ as $k\to\infty$, where $\|\dot{\mathbf{O}}_1\|=\mathcal{O}(\eps)$.

We now use the unique solution $\dot{\mathbf{O}}(\kr,\ki)$ of $\dbar$ Problem~\ref{dbar-1} as obtained above to convert the hybrid Riemann-Hilbert-$\dbar$ Problem~\ref{rhpdbar-small-norm} into a standard Riemann-Hilbert problem that we can show is of $L^2$ small norm type.  Indeed, consider the matrix function $\mathbf{E}(\kr,\ki)$ defined in terms of $\mathbf{O}(\kr,\ki)$ solving Riemann-Hilbert-$\dbar$ Problem~\ref{rhpdbar-small-norm} and $\dot{\mathbf{O}}(\kr,\ki)$ solving $\dbar$ Problem~\ref{dbar-1} by
\begin{equation}
\mathbf{E}(\kr,\ki):=\mathbf{O}(\kr,\ki)\dot{\mathbf{O}}(\kr,\ki)^{-1},\quad (\kr,\ki)\in\mathbb{R}^2\setminus\Sigma.
\end{equation}
By direct calculation,
\begin{equation}
\begin{split}
\dbar\mathbf{E}(\kr,\ki)&=\left[\dbar\mathbf{O}(\kr,\ki)\right]\dot{\mathbf{O}}(\kr,\ki)^{-1}-
\mathbf{O}(\kr,\ki)\dot{\mathbf{O}}(\kr,\ki)^{-1}\left[\dbar\dot{\mathbf{O}}(\kr,\ki)\right]\dot{\mathbf{O}}(\kr,\ki)^{-1}\\
&=\mathbf{O}(\kr,\ki)\mathbf{W}(\kr,\ki)\dot{\mathbf{O}}(\kr,\ki)^{-1}-
\mathbf{O}(\kr,\ki)\mathbf{W}(\kr,\ki)\dot{\mathbf{O}}(\kr,\ki)^{-1}\\
&=\mathbf{0},\quad (\kr,\ki)\in\mathbb{R}^2\setminus\Sigma,
\end{split}
\end{equation}
and hence $\mathbf{E}$ is analytic in each connected component of $\mathbb{R}^2\setminus\Sigma$.  In light of this result, we will henceforth use the notation $\mathbf{E}=\mathbf{E}(k)$ with $k=\kr+\myi \ki$.
It is a direct matter to calculate the jump conditions satisfied by $\mathbf{E}$ across the arcs of the contour $\Sigma$ in terms of the jump matrix $\mathbf{J}_0$ for $\mathbf{O}$ and the function $\dot{\mathbf{O}}$ restricted to $\Sigma$, and to calculate the asymptotic behavior of $\mathbf{E}$ as $k\to\infty$.  We deduce that $\mathbf{E}$ satisfies the following (pure) Riemann-Hilbert problem:
\begin{rhp}
Find a $2\times 2$ matrix $\mathbf{E}(k)$ with the following properties:
\begin{itemize}
\item[]\textbf{Analyticity:}  $\mathbf{E}$ is analytic in each connected component of $\mathbb{C}\setminus\Sigma$, and takes continuous boundary values $\mathbf{E}_+(k)$ ($\mathbf{E}_-(k)$) from the left (right) at each non-self-intersection point $k$ of $\Sigma$.
\item[]\textbf{Jump condition:}  On each oriented arc of $\Sigma$ the boundary values are related by the jump condition
$\mathbf{E}_+(k)=\mathbf{E}_-(k)\mathbf{J}(\kr,\ki)$, where 
\begin{equation}
\mathbf{J}(\kr,\ki):=\dot{\mathbf{O}}(\kr,\ki)\mathbf{J}_0(\kr,\ki)\dot{\mathbf{O}}(\kr,\ki)^{-1}.
\end{equation}
\item[]\textbf{Normalization:}  $\mathbf{E}(k)\to\mathbb{I}$ as $k\to\infty$ in $\mathbb{C}$.
\end{itemize}
\label{rhp-E-1}
\end{rhp}
Since $\dot{\mathbf{O}}(\kr,\ki)$ and $\dot{\mathbf{O}}(\kr,\ki)^{-1}$ are uniformly bounded independent of $\eps$ for $\eps>0$ sufficiently small, it follows immediately from the estimate $\|\mathbf{J}_0-\mathbb{I}\|_{L^\infty(\Sigma)}=\mathcal{O}((\log(\eps^{-1}))^{-1/2})$ that also $\|\mathbf{J}-\mathbb{I}\|_{L^\infty(\Sigma)}=\mathcal{O}((\log(\eps^{-1}))^{-1/2})$ as $\eps\downarrow 0$.  Since $\Sigma$ is compact, this condition implies the unique solvability of Riemann-Hilbert Problem~\ref{rhp-E-1} as a small-norm problem in the $L^2$ sense.  See \cite{Deift} or \cite[Appendix B]{BuckinghamM13} for details.  In particular, $\mathbf{E}(k)=\mathbb{I}+k^{-1}\mathbf{E}_1 + \mathcal{O}(k^{-2})$ as $k\to\infty$ with $\|\mathbf{E}_1\|=\mathcal{O}((\log(\eps^{-1}))^{-1/2})$.

Finally, we consider the matrix $\tilde{\mathbf{M}}(k)$ solving Riemann-Hilbert Problem~\ref{rhp-M-tilde} for large $|k|$.  We obtain the exact formula
\begin{equation}
\tilde{\mathbf{M}}(k)=\mathbf{O}(\kr,\ki)=\mathbf{E}(k)\dot{\mathbf{O}}(\kr,\ki),\quad\text{for $|k|$ sufficiently large},
\end{equation}
from which we compute 
\begin{equation}
\tilde{q}^\eps(x,0)=2\myi\lim_{k\to\infty} k[\mathbf{E}(k)\dot{\mathbf{O}}(\kr,\ki)]_{12} = 2i[\mathbf{E}_1+\dot{\mathbf{O}}_1]_{12} = \mathcal{O}((\log(\eps^{-1}))^{-1/2}),\quad\eps\downarrow 0.
\end{equation}
As the error term is uniform for $x\ge x_0>0$, the proof is complete.

\subsubsection{Proof of Corollary~\ref{corollary-vacuum}}
\label{sec:vacuum-domain}
To prove Corollary~\ref{corollary-vacuum}, we simply observe that 
the only dependence on $(x,t)\in\mathbb{R}^2$ in the proof of Theorem~\ref{theorem:initial-condition-small} given in \S\ref{sec:initial-condition-small-proof} involved the inequality \eqref{eq:basic-inequality-1}, and it is not hard to see that this inequality holds also for some nonzero $t$.  More generally,
\begin{equation}
\theta'(k;x,t)-\Phi'(k)=x+4tk-\Phi'(k),
\end{equation}
and as $\Phi'(\cdot)$ is a function with maximum value zero at $k=k_0$ only and tending to $-\infty$ as $k\downarrow k_\mathfrak{a}$ and $k\uparrow k_\mathfrak{b}$, given $t>0$ there will be some finite $X(t)>0$ such that $\theta'(k;x,t)-\Phi'(k)>0$ holds strictly on $k_\mathfrak{a}<k<k_\mathfrak{b}$ for $x>X(t)$ but fails for some $k\in (k_\mathfrak{a},k_\mathfrak{b})$ if $x\le X(t)$.  The rest of the proof of Theorem~\ref{theorem:initial-condition-small} then goes through unchanged, with the same result, and the proof of Corollary~\ref{corollary-vacuum} is complete.

We conclude this short section by obtaining explicit and simple asymptotic formulae for the boundary curve $x=X(t)$ valid for small and large $t>0$.  We note firstly that $-\Phi'(k)$ is certainly locally convex (i) for $k-k_0$ small, because $\Phi'''(k_0)=16k_0h_0^{-2}<0$ (see \eqref{eq:Phi-Three-k0}) and according to Lemma~\ref{lemma-Phi-interior}, $\Phi'''(\cdot)$ is continuous on $(k_\mathfrak{a},k_\mathfrak{b})$, and also (ii) for $k-k_\mathfrak{a}$ small, where according to Lemma~\ref{lemma-Phi-endpoints}
we have $\Phi'(k)\sim C_\mathfrak{a}\log(k-k_\mathfrak{a})$ where $C_\mathfrak{a}>0$ is given in \eqref{eq:Phi-k-a}.  In the former case, the slope of the tangent line of $-\Phi'(k)$ is small, while in the latter case the slope is large.  Therefore, for small or large positive $t$ we can apply the following steps to obtain $X(t)$:  firstly solve the equation $\Phi''(k)=4t$ for $k=k(t)$, and then obtain $X(t)=\Phi'(k(t))-4tk(t)$.  

When $t>0$ is small, we expect $k-k_0$ to be small, and since
\begin{equation}
\Phi''(k)=\Phi'''(k_0)(k-k_0) + o(k-k_0)= 16k_0h_0^{-2}(k-k_0)+o(k-k_0),\quad k\to k_0,
\label{eq:PhiPP-near-k0}
\end{equation}
from $\Phi''(k)=4t$ we obtain
\begin{equation}
k(t)=k_0+\frac{4}{\Phi'''(k_0)}t+o(t)=k_0+\frac{h_0^2}{4k_0}t+o(t),\quad t\downarrow 0.
\label{eq:k-t-small}
\end{equation}
Since $\Phi'(k_0)=0$, integrating \eqref{eq:PhiPP-near-k0} and substituting from \eqref{eq:k-t-small} yields
\begin{equation}
X(t)=\Phi'(k(t))-4tk(t)=X_0(t)+o(t^2),\quad t\downarrow 0,
\end{equation}
where the asymptote to $X(t)$ for small $t$ is defined by
\begin{equation}
X_0(t):=-4k_0t-\frac{h_0^2}{2k_0}t^2.
\label{eq:X-asymptote-t-small}
\end{equation}

On the other hand, if $t>0$ is large, then we expect $k-k_\mathfrak{a}$ will be small.  The equation to be solved for $k=k(t)$ in this case is then $\Phi''(k)=4t$, where now
\begin{equation}
\Phi''(k)=\frac{C_\mathfrak{a}}{k-k_\mathfrak{a}} + \mathcal{O}(1),\quad k\downarrow k_\mathfrak{a}.
\end{equation}
Therefore
\begin{equation}
k(t)=k_\mathfrak{a}+\frac{1}{4}C_\mathfrak{a}t^{-1}+\mathcal{O}(t^{-2}),\quad t\to +\infty.
\end{equation}
Using $\Phi'(k)=C_\mathfrak{a}\log(k-k_\mathfrak{a})+\mathcal{O}(1)$  as $k\downarrow k_\mathfrak{a}$, we then have
\begin{equation}
X(t)=\Phi'(k(t))-4tk(t)=X_\infty(t)+\mathcal{O}(1),\quad t\to +\infty,
\end{equation}
where the asymptote to $X(t)$ for large $t$ is given by
\begin{equation}
X_\infty(t):=-4k_\mathfrak{a}t-C_\mathfrak{a}\log(t).
\label{eq:X-asymptote-t-large}
\end{equation}

\subsection{Asymptotic behavior of $\tilde{q}^\eps(0,t)$ for $t>0$ and related analysis}
\subsubsection{General methodology.  The complex phase function $g$}
\label{sec:g-function}
A general strategy to the analysis of the solution of Riemann-Hilbert Problem~\ref{rhp-M-tilde} in the semiclassical limit $\eps\downarrow 0$ follows the basic approach outlined in \cite{DVZ}, which is based on the introduction of a scalar \emph{complex phase function} $g=g(k)$ having the following basic properties:
\begin{itemize}
\item $g$ is analytic for $k\in\mathbb{C}\setminus[k_\mathfrak{a},k_\mathfrak{b}]$ and takes continuous boundary values $g_\pm$ on $(k_\mathfrak{a},k_\mathfrak{b})$ from $\mathbb{C}_\pm$,
\item $g(k)\to 0$ as $k\to\infty$, 
\item $g(k^*)=g(k)^*$ (Schwarz symmetry), and very importantly,
\item $g$ is independent of $\eps$ (although it will generally depend on $x$ and $t$).
\end{itemize}
One introduces such a function $g$ into Riemann-Hilbert Problem~\ref{rhp-M-tilde} by making
the substitution
\begin{equation}
\tilde{\mathbf{M}}(k):=\mathbf{N}(k)e^{\myi g(k)\sigma_3/\eps}.
\label{eq:N-from-g}
\end{equation}
The basic properties of $g$ listed above are by no means sufficient to determine $g$ (this is why we do not formulate them as a proper Riemann-Hilbert problem), and the point is that one should use the freedom of choice of $g$ to try to bring $\mathbf{N}(k)$ into a form amenable for asymptotic analysis in the limit $\eps\downarrow 0$.  

The transformation \eqref{eq:N-from-g} implies that $\mathbf{N}$ is analytic where $\tilde{\mathbf{M}}$ is, takes boundary values in the same way, satisfies exactly the same normalization condition as $k\to\infty$ as does $\tilde{\mathbf{M}}$, and satisfies a modified jump condition:
\begin{equation}
\mathbf{N}_+(k)=\mathbf{N}_-(k)\begin{bmatrix} 
e^{2(\Delta(k)-\tau(k))/\eps} & -Y^\eps(k)e^{-2\myi\phi(k)/\eps}\\
Y^\eps(k)e^{2\myi\phi(k)/\eps} & e^{-2\Delta(k)/\eps}\end{bmatrix},\quad
k_\mathfrak{a}<k<k_\mathfrak{b},
\label{eq:N-jump}
\end{equation}
where
\begin{equation}
\phi(k):=\theta(k;x,t)-\Phi(k)-\frac{1}{2}(g_+(k)+g_-(k))\in\mathbb{R}\quad\text{and}\quad
\Delta(k):=-\myi\frac{1}{2}(g_+(k)-g_-(k))\in\mathbb{R}.
\label{eq:phi-Delta-define}
\end{equation}
The strategy of \cite{DVZ} is to try to choose $g$ so that the interval $(k_\mathfrak{a},k_\mathfrak{b})$ splits into a finite number of subintervals of three distinct types:
\begin{itemize}
\item[]\textbf{Voids:}  intervals in which $\Delta(k)\equiv 0$ and $\phi'(k)>0$.
\item[]\textbf{Bands:}  intervals in which $0<\Delta(k)<\tau(k)$ and $\phi'(k)\equiv 0$.
\item[]\textbf{Saturated regions:}  intervals in which $\Delta(k)\equiv\tau(k)$ and $\phi'(k)<0$.
\end{itemize}
In each void interval $\mathrm{V}$, the modified jump matrix admits an ``upper-lower'' factorization because $\Delta(k)\equiv 0$:
\begin{equation}
\begin{bmatrix} 
e^{2(\Delta(k)-\tau(k))/\eps} & -Y^\eps(k)e^{-2\myi\phi(k)/\eps}\\
Y^\eps(k)e^{2\myi\phi(k)/\eps} & e^{-2\Delta(k)/\eps}\end{bmatrix}=
\begin{bmatrix}1 & -Y^\eps(k)e^{-2\myi\phi(k)/\eps} \\ 0 & 1\end{bmatrix}
\begin{bmatrix}1 & 0\\Y^\eps(k)e^{2\myi\phi(k)/\eps} & 1\end{bmatrix},\quad k\in \mathrm{V},
\label{eq:Void-Trick}
\end{equation}
and the monotonicity condition $\phi'(k)>0$ suggests that the first (second) factor has a continuation into the lower (upper) half-plane that is exponentially close to the identity matrix.  In each band interval $\mathrm{B}$, the modified jump matrix is obviously exponentially close to a constant off-diagonal matrix due to the inequalities $0<\Delta(k)<\tau(k)$:
\begin{equation}
\begin{bmatrix} 
e^{2(\Delta(k)-\tau(k))/\eps} & -Y^\eps(k)e^{-2\myi\phi(k)/\eps}\\
Y^\eps(k)e^{2\myi\phi(k)/\eps} & e^{-2\Delta(k)/\eps}\end{bmatrix}=\begin{bmatrix}0 & -e^{-2\myi\phi_\mathrm{B}/\eps}\\
e^{2\myi\phi_\mathrm{B}/\eps} & 0\end{bmatrix} + \text{exponentially small terms},\quad k\in \mathrm{B},
\label{eq:Band-Trick}
\end{equation}
where $\phi_\mathrm{B}$ denotes the constant value of $\phi(k)$ in the band $\mathrm{B}$.  Finally, in each saturated interval $\mathrm{S}$, the modified jump matrix admits a ``lower-upper'' factorization because $\Delta(k)\equiv\tau(k)$:
\begin{equation}
\begin{bmatrix} 
e^{2(\Delta(k)-\tau(k))/\eps} & -Y^\eps(k)e^{-2\myi\phi(k)/\eps}\\
Y^\eps(k)e^{2\myi\phi(k)/\eps} & e^{-2\Delta(k)/\eps}\end{bmatrix}=
\begin{bmatrix}1 & 0\\Y^\eps(k)e^{2\myi\phi(k)/\eps} & 1\end{bmatrix}
\begin{bmatrix}1 & -Y^\eps(k)e^{-2\myi\phi(k)/\eps} \\ 0 & 1\end{bmatrix},\quad k\in \mathrm{S}.
\label{eq:Saturated-Region-Trick}
\end{equation}
The inequality $\phi'(k)<0$ then suggests that the first (second) factor can be continued into the lower (upper) half-plane, becoming an exponentially small perturbation of the identity matrix $\mathbb{I}$.

The function $g$ may be constructed by temporarily setting aside the inequalities involved with the voids, bands, and saturated regions.  We suppose that there are $N+1$ bands in $(k_\mathfrak{a},k_\mathfrak{b})$ that we denote by $\mathrm{B}_j:=[\alpha_j,\beta_j]$ with $k_\mathfrak{a}<\alpha_0<\beta_0<\alpha_1<\beta_1<\cdots<\alpha_N<\beta_N<k_\mathfrak{b}$.  The complementary intervals are either voids or saturated regions.  The boundary values of the function $g(k)$ then satisfy
\begin{itemize}
\item $g_+(k)-g_-(k)=0$ which implies $g_+'(k)-g_-'(k)=0$ for $k$ in voids.
\item $g_+'(k)+g_-'(k)=2\theta'(k;x,t)-2\Phi'(k)$ for $k$ in bands.
\item $g_+(k)-g_-(k)=2\myi\tau(k)$ which implies $g_+'(k)-g_-'(k)=2\myi \tau'(k)$ for $k$ in saturated regions.
\end{itemize}
Therefore, we know the value of $g_+'(k)-g_-'(k)$ everywhere in the interval $(k_\mathfrak{a},k_\mathfrak{b})$ with the exception of the band intervals, where instead we know $g_+'(k)+g_-'(k)$.  Denoting by $r(k)$ the function analytic for $k\in\mathbb{C}\setminus (\mathrm{B}_0\cup \mathrm{B}_1\cup\cdots\cup \mathrm{B}_N)$ that satisfies $r(k)^2=(k-\alpha_0)(k-\beta_0)\cdots (k-\alpha_N)(k-\beta_N)$ and $r(k)=k^{N+1}+\mathcal{O}(k^N)$ as $k\to\infty$, we may consider instead of $g'(k)$ the related function $m(k):=g'(k)/r(k)$.  This function is analytic for $k\in\mathbb{C}\setminus[k_\mathfrak{a},k_\mathfrak{b}]$ and since $r$ changes sign across the band intervals and is otherwise analytic, $m$ satisfies
\begin{equation}
m_+(k)-m_-(k)=\begin{cases}0,&\quad \text{$k$ in voids (and outside of $[k_\mathfrak{a},k_\mathfrak{b}]$)}\\
\displaystyle\frac{2\theta'(k;x,t)-2\Phi'(k)}{r_+(k)},&\quad \text{$k$ in bands}\\
\displaystyle\frac{2\myi \tau'(k)}{r(k)},&\quad\text{$k$ in saturated regions}.
\end{cases}
\label{eq:m-difference-boundary-values}
\end{equation}
Note that $g'(k)$ must decay as $\mathcal{O}(k^{-2})$ as $k\to\infty$ because $g(k)\to 0$ in this limit; this implies that $m(k)=\mathcal{O}(k^{-(N+3)})$ for large $k$, and in particular $m(k)=o(1)$.  Therefore $m$ is necessarily given in terms of the difference of its boundary values explicitly written in \eqref{eq:m-difference-boundary-values} by the Plemelj formula, which implies that
\begin{equation}
g'(k)=\frac{r(k)}{\pi \myi}\int_\mathrm{B}\frac{(\theta'(l;x,t)-\Phi'(l))\,dl}{r_+(l)(l-k)} +\frac{r(k)}{\pi}\int_\mathrm{S}
\frac{\tau'(l)\,dl}{r(l)(l-k)},
\label{eq:gprime}
\end{equation}
where here $\mathrm{B}$ denotes the union of all bands and $\mathrm{S}$ denotes the union of all saturated regions.
Expanding the Cauchy kernel $(l-k)^{-1}$ in geometric series for large $k$, we see that the condition $g'(k)=\mathcal{O}(k^{-2})$ as $k\to\infty$ is equivalent to the following \emph{moment conditions}:
\begin{equation}
m_n:=\int_\mathrm{B}\frac{(\theta'(k;x,t)-\Phi'(k))k^{n-1}\,dk}{\myi r_+(k)} +\int_\mathrm{S}\frac{\tau'(k)k^{n-1}\,dk}{r(k)}=0,\quad n=1,2,\dots,N+2.
\label{eq:moments}
\end{equation}
Subject to these conditions, $g'(k)$ is integrable at infinity, and $g(k)$ may be expressed as a contour integral:
\begin{equation}
g(k)=\int_\infty^k g'(l)\,dl
\label{eq:g-gprime-integrate}
\end{equation}
where $g'(k)$ is explicitly given by \eqref{eq:gprime}.
Equations \eqref{eq:moments} are $N+2$ conditions on the $2N+2$ unknown endpoints of the bands $\mathrm{B}_0,\dots,\mathrm{B}_N$.  In general, additional conditions arise in order to get the integration constants right so that instead of just $g_+'(k)-g_-'(k)=0$ in voids we actually have $g_+(k)-g_-(k)=0$, and so that instead of just $g_+'(k)-g_-'(k)=2\myi\tau'(k)$ in saturated regions we actually have $g_+(k)-g_-(k)=2\myi\tau(k)$.  Since $\tau(k_{\mathfrak{a},\mathfrak{b}})=0$, no additional conditions are required if there is only one band, i.e., $N=0$, in which case the moment conditions \eqref{eq:moments} may determine the band endpoints $\alpha=\alpha_0$ and $\beta=\beta_0$.  For the purposes of the proofs of Theorem~\ref{theorem:boundary-condition-recover} and Corollary~\ref{corollary:plane-wave}, we only consider this case.

Because $\theta'(k;x,t)$ is entire in $k$, the corresponding integral in the expression \eqref{eq:moments} for $m_n$ can always be expressed in closed form by a residue calculation at $k=\infty$:
\begin{equation}
\int_\mathrm{B}\frac{\theta'(k;x,t)k^{n-1}\,dk}{\myi r_+(k)} = -\frac{1}{2\myi}\oint_L\frac{(x+4kt)k^{n-1}\,dk}{r(k)}
\end{equation}
where $L$ is a large, positively-oriented circular contour that encloses all of the bands.  With the help of the expansion
\begin{equation}
\frac{1}{r(k)}=\frac{1}{k}+\frac{\alpha+\beta}{2k^2} +\frac{3\alpha^2+2\alpha\beta+3\beta^2}{8k^3} + \mathcal{O}(k^{-4}),\quad k\to\infty,\quad N=0,
\end{equation}
we therefore obtain in the case $N=0$ that the moment conditions \eqref{eq:moments} take the form
\begin{equation}
\begin{split}
m_1(\alpha,\beta)&=I_1(\alpha,\beta)-\pi\left(x+2(\alpha+\beta)t\right)=0\\
m_2(\alpha,\beta)&=I_2(\alpha,\beta)-\frac{\pi}{2}
\left((\alpha+\beta)x +(3\alpha^2+2\alpha\beta+3\beta^2)t\right)=0,
\end{split}
\label{eq:moments-genus-zero}
\end{equation}
where
\begin{equation}
I_p(\alpha,\beta):=\int_S\frac{\tau'(k)k^{p-1}\,dk}{r(k)}-\int_\alpha^\beta\frac{\Phi'(k)k^{p-1}\,dk}{\myi r_+(k)},\quad p=1,2.
\label{eq:Ip-formula}
\end{equation}
\begin{lemma}
Suppose that $x=0$.  The equations \eqref{eq:moments-genus-zero} are satisfied for $t>0$ by
$\alpha=\mathfrak{a}(t)$ and $\beta=\mathfrak{b}(t)$ as long as $(k_\mathfrak{a},\mathfrak{a}(t))$
is a void (saturated region) if $\mathfrak{a}'(t)<0$ ($\mathfrak{a}'(t)>0$) and $(\mathfrak{b}(t),k_\mathfrak{b})$ is a void (saturated region) if $\mathfrak{b}'(t)>0$ ($\mathfrak{b}'(t)<0$).
\label{lemma:endpoints-x-zero}
\end{lemma}
\begin{proof}
Suppose that for some $t'>0$ we have $\alpha=\mathfrak{a}(t')$ and $\beta=\mathfrak{b}(t')$.  Let us evaluate $I_p(\mathfrak{a}(t'),\mathfrak{b}(t'))$ in four cases:
\begin{itemize}
\item[]\textbf{VBV:}  in this case we assume that $\mathfrak{a}'(t')<0$, $\mathfrak{b}'(t')>0$, and both intervals $(k_\mathfrak{a},\mathfrak{a}(t'))$ and $(\mathfrak{b}(t'),k_\mathfrak{b})$ are voids.
\item[]\textbf{VBS:}  in this case we assume that $\mathfrak{a}'(t')<0$ and $\mathfrak{b}'(t')<0$, and that $(k_\mathfrak{a},\mathfrak{a}(t'))$ is a void but $(\mathfrak{b}(t'),k_\mathfrak{b})$ is a saturated region.
\item[]\textbf{SBV:}  in this case we assume that $\mathfrak{a}'(t')>0$ and $\mathfrak{b}'(t')>0$, and that $(\mathfrak{b}(t'),k_\mathfrak{b})$ is a void but $(k_\mathfrak{a},\mathfrak{a}(t'))$ is a saturated region.
\item[]\textbf{SBS:}  in this case we assume that $\mathfrak{a}'(t')>0$ and $\mathfrak{b}'(t')<0$, and that both intervals $(k_\mathfrak{a},\mathfrak{a}(t'))$ and $(\mathfrak{b}(t'),k_\mathfrak{b})$ are saturated regions.
\end{itemize}
\begin{figure}[h]
\includegraphics{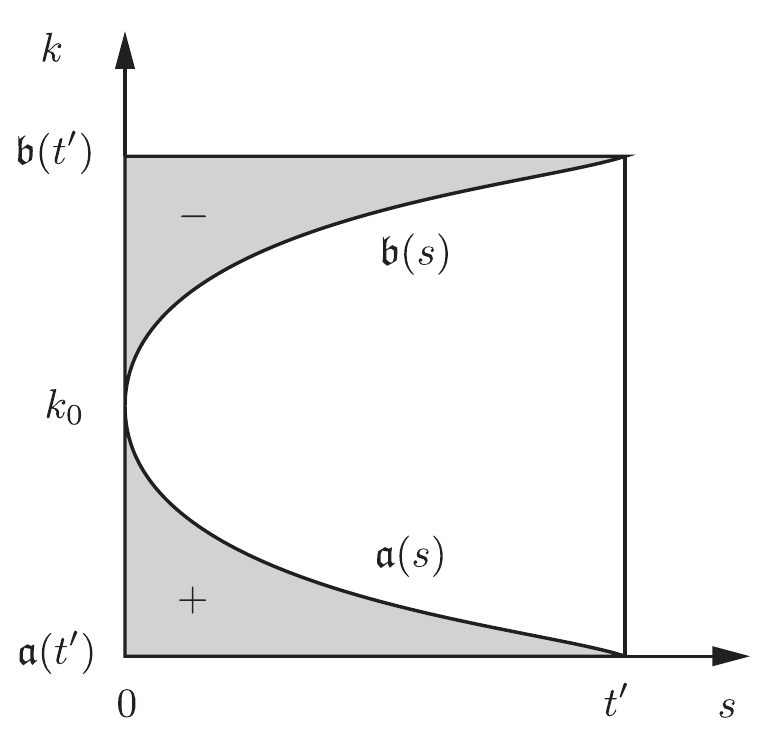}
\caption{The region of integration in the $(s,k)$-plane in the VBV case.}
\label{fig:VBV}
\end{figure}
\begin{figure}[h]
\includegraphics{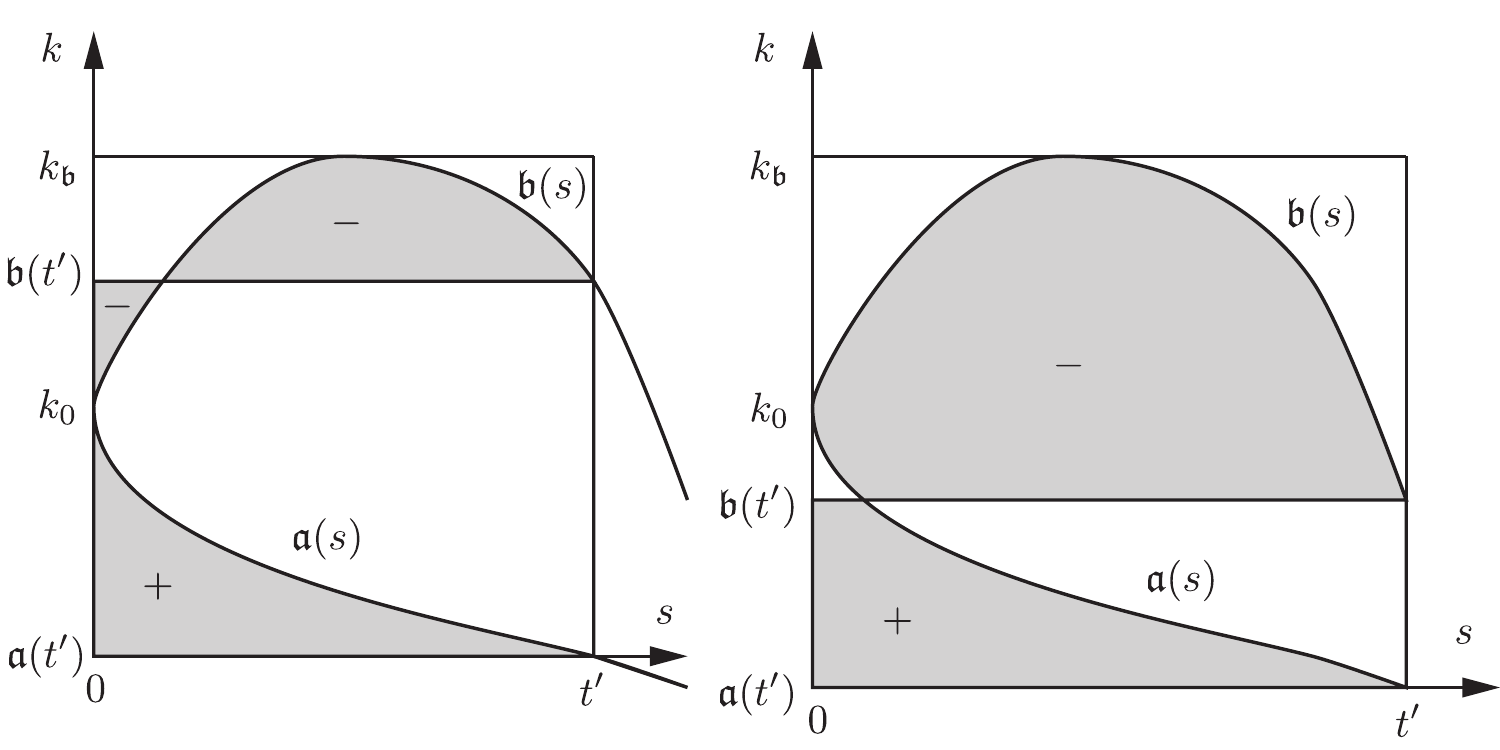}
\caption{The region of integration in the $(s,k)$-plane in the VBS case.  Left:  $\mathfrak{b}(t')>k_0$.  Right:  $\mathfrak{b}(t')<k_0$.}
\label{fig:VBS}
\end{figure}
\begin{figure}[h]
\includegraphics{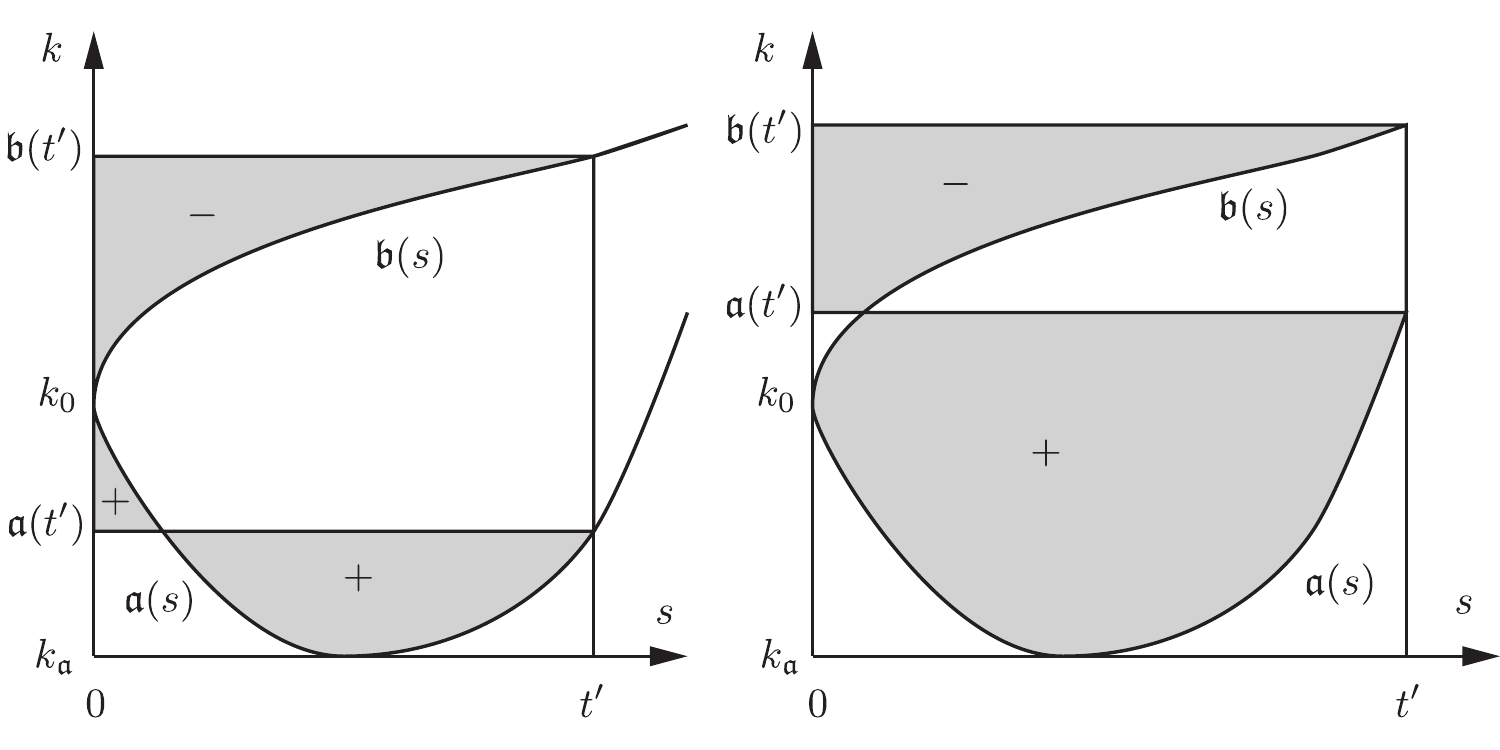}
\caption{The region of integration in the $(s,k)$-plane in the SBV case.  Left:  $\mathfrak{a}(t')<k_0$.  Right:  $\mathfrak{a}(t')>k_0$.}
\label{fig:SBV}
\end{figure}
\begin{figure}[h]
\includegraphics{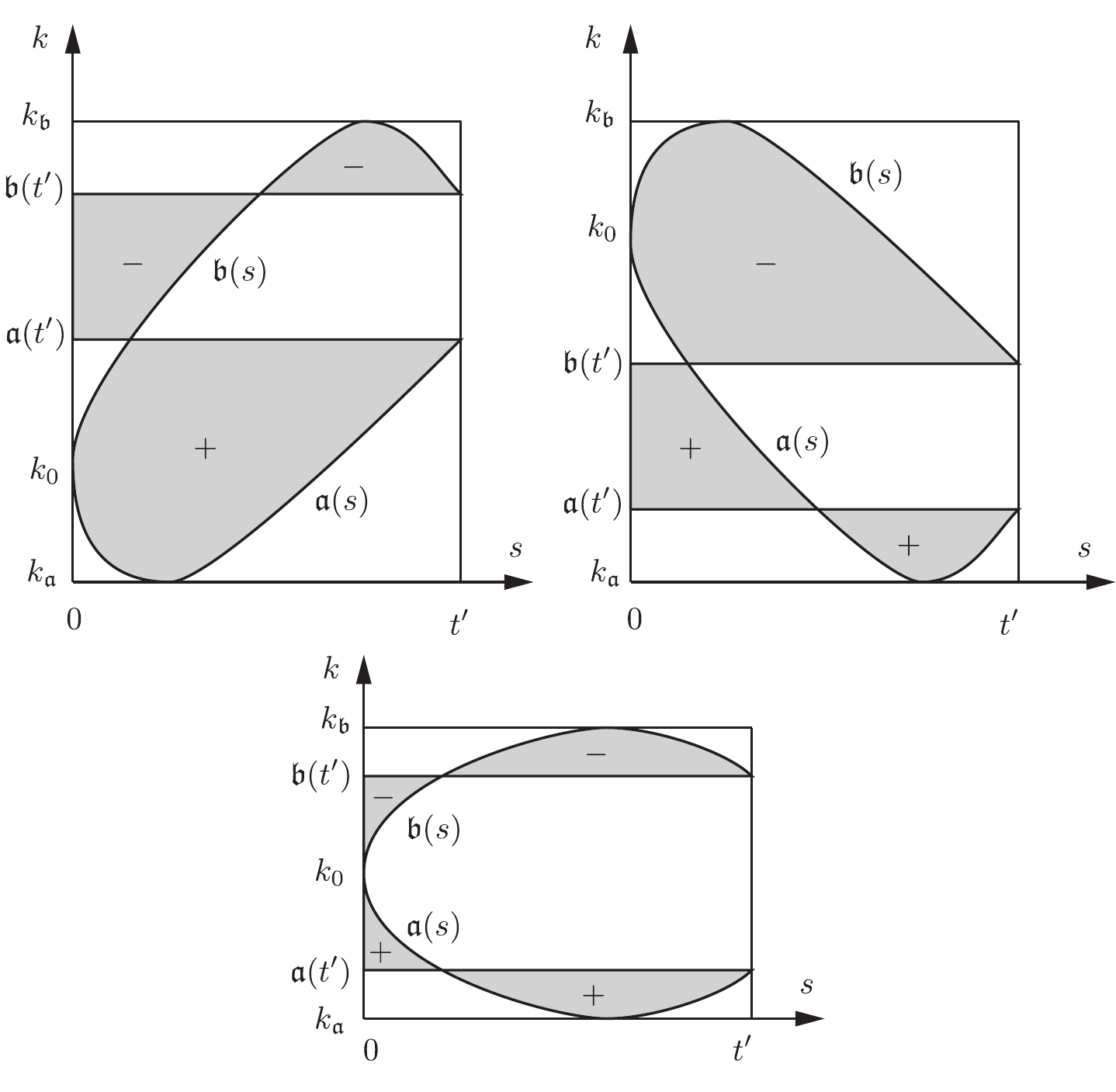}
\caption{The region of integration in the $(s,k)$-plane in the SBS case.  Top left: $\mathfrak{a}(t')>k_0$.   Top right:  $\mathfrak{b}(t')<k_0$.  Bottom:  $\mathfrak{a}(t')<k_0<\mathfrak{b}(t')$.}
\label{fig:SBS}
\end{figure}
We may substitute into \eqref{eq:Ip-formula} from \eqref{eq:tau-define-1} (after differentiating $\tau(k)$ by Leibniz' rule taking into account that the integrand vanishes at both endpoints) and from \eqref{eq:Phi-prime-real-1}.  We may also simplify $\myi r_+(k)$ for $\alpha<k<\beta$ as $-\sqrt{(k-\alpha)(\beta-k)}<0$, while in the integral over $\mathrm{S}$ we have either $r(k)=\sqrt{(k-\alpha)(k-\beta)}>0$ if $\mathrm{S}=(\beta,k_\mathfrak{b})$ or $r(k)=-\sqrt{(\alpha-k)(\beta-k)}<0$ if $\mathrm{S}=(k_\mathfrak{a},\alpha)$.  The first remarkable fact is that the result can be written in a uniform way in all four cases (including all sub-cases related to where $k_0$ falls with respect to $\alpha=\mathfrak{a}(t')$ and $\beta=\mathfrak{b}(t')$ as illustrated in Figures~\ref{fig:VBV}--\ref{fig:SBS}).  Namely, we have
\begin{multline}
I_p(\mathfrak{a}(t'),\mathfrak{b}(t'))=\iint_{D_+}\frac{4k^{p+1}+2U(s)k^p-2H(s)^2k^{p-1}}{
\sqrt{-(k-\mathfrak{a}(t'))(k-\mathfrak{a}(s))(k-\mathfrak{b}(t'))(k-\mathfrak{b}(s))}}\,dA(s,k) \\{}-
\iint_{D_-}\frac{4k^{p+1}+2U(s)k^p-2H(s)^2k^{p-1}}{
\sqrt{-(k-\mathfrak{a}(t'))(k-\mathfrak{a}(s))(k-\mathfrak{b}(t'))(k-\mathfrak{b}(s))}}\,dA(s,k),
\label{eq:Ip-formula-1}
\end{multline}
where $dA(s,k)$ denotes the positive area element and where the domains $D_+$ and $D_-$ are the indicated shaded regions in Figures~\ref{fig:VBV}--\ref{fig:SBS}.  It is now obvious that the original order of integration is easily reversed in all cases, with the outer $s$ integral over the interval $0<s<t'$
and the inner integrals over the intervals with endpoints being the two most negative (for the domain $D_+$) and the two least negative (for the domain $D_-$) among the four values $\mathfrak{a}(t')$, $\mathfrak{a}(s)$, $\mathfrak{b}(t')$, and $\mathfrak{b}(s)$.  Carrying out this reinterpretation of the formula \eqref{eq:Ip-formula-1}, we may further combine the inner $k$ integrals with the introduction of the function $S(k;s,t')$
that is analytic for $k$ in the complex plane with branch cuts lying in the two intervals of integration omitted, whose square is $S(k;s,t')^2=-(k-\mathfrak{a}(t'))(k-\mathfrak{a}(s))(k-\mathfrak{b}(t'))(k-\mathfrak{b}(s))$, and that satisfies $S(k;s,t')=-\myi k^2 + \mathcal{O}(k)$ as $k\to\infty$.  The result is
\begin{equation}
I_p(\mathfrak{a}(t'),\mathfrak{b}(t'))=\int_0^{t'}\oint_L\frac{2k^{p+1}+U(s)k^{p}-H(s)^2k^{p-1}}{S(k;s,t')}\,dk\,ds,
\label{eq:Ip-formula-2}
\end{equation}
where $L$ is a positively-oriented loop that encloses both branch cuts of $S$.  The inner $k$ integral may now be computed by residues for each $s\in (0,t')$.  The second remarkable fact is that for $p=1,2$ the inner $k$-integral is independent of $s$:
\begin{equation}
I_0(\mathfrak{a}(t'),\mathfrak{b}(t'))=\int_0^{t'}2\pi (\mathfrak{a}(t')+\mathfrak{b}(t'))\,ds = 2\pi (\mathfrak{a}(t')+\mathfrak{b}(t'))t',
\end{equation}
and
\begin{equation}
I_1(\mathfrak{a}(t'),\mathfrak{b}(t'))=\int_0^{t'}\frac{\pi}{2}(3\mathfrak{a}(t')^2+2\mathfrak{a}(t')\mathfrak{b}(t')+3\mathfrak{b}(t')^2)\,dt'=\frac{\pi}{2}(3\mathfrak{a}(t')^2+2\mathfrak{a}(t')\mathfrak{b}(t')+3\mathfrak{b}(t')^2)t'.
\end{equation}
It therefore follows by inspection that in all four cases, the equations $m_1(\alpha,\beta)=0$ and $m_2(\alpha,\beta)=0$ written in the form \eqref{eq:moments-genus-zero} are satisfied for $x=0$ and $t>0$ by taking $\alpha=\mathfrak{a}(t)$ and $\beta=\mathfrak{b}(t)$.
\end{proof}

The locations of the voids, bands, and saturated regions are indicated for the boundary data from Figure~\ref{fig:Assumption1} in Figure~\ref{fig:Intervals}.
\begin{figure}[h]
\includegraphics{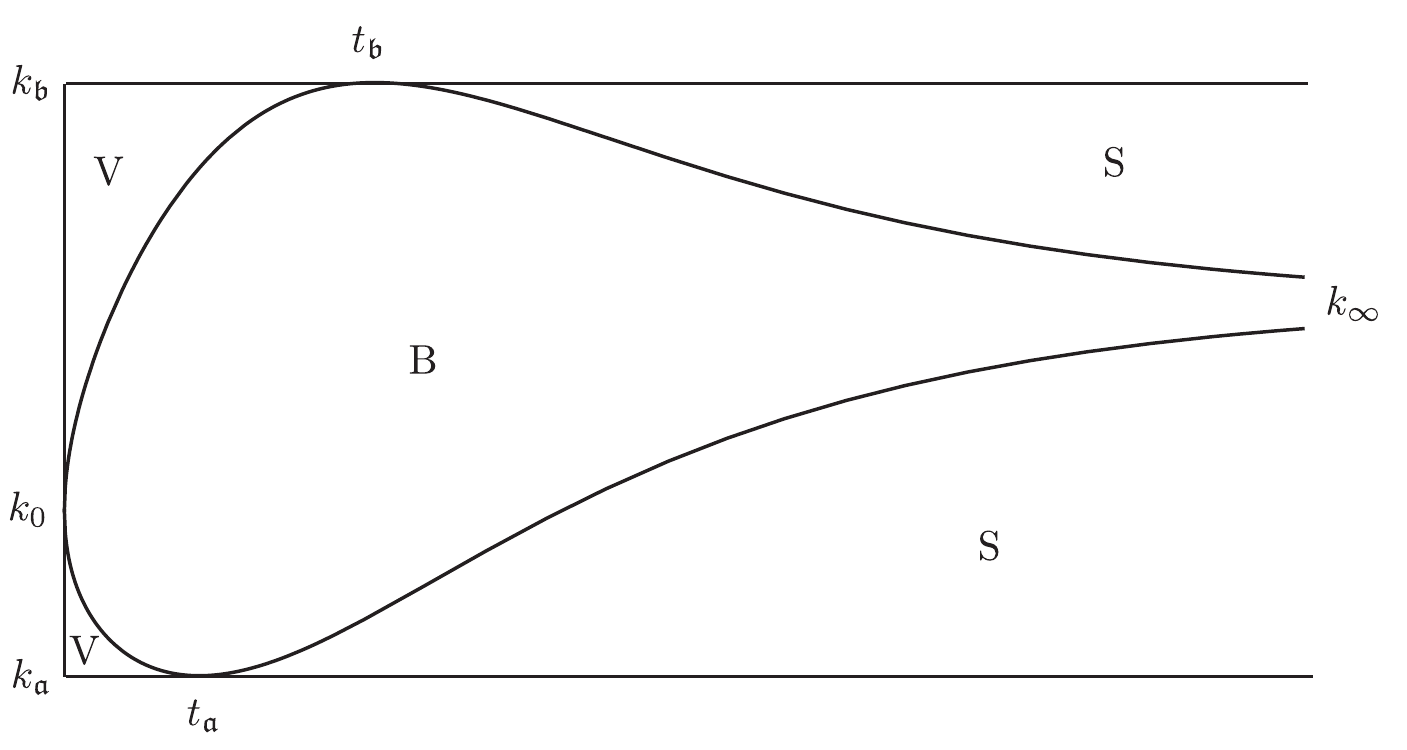}
\caption{The curves $k=\mathfrak{a}(t)$ and $k=\mathfrak{b}(t)$ from Figure~\ref{fig:Assumption1} and the corresponding voids (V), bands (B), and saturated regions (S).}
\label{fig:Intervals}
\end{figure}
We therefore find for each $t>0$ a well-defined candidate for $g(k)$ that we will denote by $g(k;t)$, with corresponding functions $\phi(k;t)$ and $\Delta(k;t)$ defined by \eqref{eq:phi-Delta-define}, and it remains only to confirm the inequalities that were dropped earlier.  Indeed, we will now prove the following.
\begin{lemma}
Let $x=0$ and $t>0$, and let the function $g(k;t)$ be determined from the values $\alpha=\mathfrak{a}(t)$ and $\beta=\mathfrak{b}(t)$ and the configuration of voids and saturated regions as described in Lemma~\ref{lemma:endpoints-x-zero}.  The functions $\phi(k;t)$ and $\Delta(k;t)$ given by \eqref{eq:phi-Delta-define} in terms of $g$ satisfy the following inequalities:
\begin{equation}
0<\Delta(k;t)<\tau(k),\quad \mathfrak{a}(t)<k<\mathfrak{b}(t),
\label{eq:Band-Inequality}
\end{equation}
\begin{equation}
\phi'(k;t)>0,\quad k\in \mathrm{V},
\label{eq:Void-Inequality}
\end{equation}
and 
\begin{equation}
\phi'(k;t)<0,\quad k\in \mathrm{S},
\label{eq:Saturated-Region-Inequality}
\end{equation}
where $\mathrm{V}$ denotes the union of the voids, $\mathrm{S}$ denotes the union of the saturated regions, and the prime denotes differentiation in $k$ for fixed $t$.  
\label{lemma:inequalities-x-zero}
\end{lemma}
\begin{proof}
The proof of these three statements involves 
the same object, namely the partial derivative of $g$ in $t$ for fixed $k$, denoted $g_t(k;t)$.  We may construct $g_t(k;t)$ explicitly as follows.  Consider differentiation with respect to $t$ of the three equations $g_+(k;t)-g_-(k;t)=0$ for $k\in \mathrm{V}$, $g_+(k;t)-g_-(k;t)=2\myi\tau(k)$ for $k\in \mathrm{S}$, and $g_+(k;t)+g_-(k;t)=2\theta(k;0,t)-2\Phi(k)-2\phi_B(t)$ for $k\in \mathrm{B}=(\mathfrak{a}(t),\mathfrak{b}(t))$.  Since neither $\tau$ nor $\Phi$ depend on $t$, we find that $g_t(k;t)$ is analytic for $k\in\mathbb{C}\setminus [\mathfrak{a}(t),\mathfrak{b}(t)]$, and on the cut the equation 
\begin{equation}
g_{t+}(k;t)+g_{t-}(k;t)=4k^2-2\phi_{\mathrm{B}}'(t),\quad \mathfrak{a}(t)<k<\mathfrak{b}(t)
\end{equation}
holds.  Keeping in mind the condition $g_t(k;t)=\mathcal{O}(k^{-1})$ as $k\to\infty$, it therefore follows by similar arguments as led to \eqref{eq:gprime} that
\begin{equation}
g_t(k;t)=\frac{r(k;t)}{2\pi i}\int_{\mathfrak{a}(t)}^{\mathfrak{b}(t)}\frac{4l^2-2\phi_{\mathrm{B}}'(t)}{r_+(l;t)(l-k)}\,dl,
\end{equation}
where the notation $r(k;t)$ reminds us that the branch points are $\alpha=\mathfrak{a}(t)$ and $\beta=\mathfrak{b}(t)$,
and the integral can be evaluated in closed form by residues at $l=k$ and $l=\infty$:
\begin{equation}
g_t(k;t)=2k^2-\phi_{\mathrm{B}}'(t)+(U(t)-2k)r(k;t),\quad k\in\mathbb{C}\setminus [\mathfrak{a}(t),\mathfrak{b}(t)],
\label{eq:gt}
\end{equation}
where we used the identity $\mathfrak{a}(t)+\mathfrak{b}(t)=-U(t)$.  Now we prove \eqref{eq:Band-Inequality}.  From \eqref{eq:phi-Delta-define} and \eqref{eq:gt} it follows that
\begin{equation}
\begin{split}
\Delta_t(k;t)&=-\frac{1}{2}\myi\left(g_{t+}(k;t)-g_{t-}(k;t)\right)\\ &=-\myi (U(t)-2k)r_+(k;t)\\ &=(U(t)-2k)\sqrt{(k-\mathfrak{a}(t))(\mathfrak{b}(t)-k)},\quad \mathfrak{a}(t)<k<\mathfrak{b}(t).
\end{split}
\end{equation}
For a given $k\in (\mathfrak{a}(t),\mathfrak{b}(t))$, we use the fundamental theorem of calculus to write
\begin{equation}
\Delta(k;t)=\Delta(k;t_-(k)) +\int_{t_-(k)}^t(U(s)-2k)\sqrt{(k-\mathfrak{a}(s))(\mathfrak{b}(s)-k)}\,ds,
\label{eq:Delta-integral}
\end{equation}
a formula that makes use of the fact that $\mathfrak{a}(s)<k<\mathfrak{b}(s)$ holds for $t_-(k)<s<t$.
But since $t=t_-(k)$ always corresponds to the boundary between a band and a void, we have
$\Delta(k;t_-(k))=0$.  Taking this into account and comparing \eqref{eq:Delta-integral} with \eqref{eq:tau-define-1} completes the proof of \eqref{eq:Band-Inequality}, since for $\mathfrak{a}(t)<k<\mathfrak{b}(t)$ we have
\begin{equation}
\begin{split}
\Delta(k;t)&=\int_{t_-(k)}^t (U(s)-2k)\sqrt{(k-\mathfrak{a}(s))(\mathfrak{b}(s)-k)}\,ds>0\\
\tau(k)-\Delta(k;t)&=\int_t^{t_+(k)}(U(s)-2k)\sqrt{(k-\mathfrak{a}(s))(\mathfrak{b}(s)-k)}\,ds>0.
\end{split}
\end{equation}
Next we consider together \eqref{eq:Void-Inequality} and \eqref{eq:Saturated-Region-Inequality}.
From \eqref{eq:phi-Delta-define} and \eqref{eq:gt} we have
\begin{equation}
\begin{split}
\phi_t(k;t)&=\theta_t(k;0,t)-\frac{1}{2}(g_{t+}(k;t)+g_{t-}(k;t))\\
&=2k^2-g_t(k;t),\quad k\in (k_\mathfrak{a},\mathfrak{a}(t))\cup (\mathfrak{b}(t),k_\mathfrak{b}),\\
&=\phi_{B}'(t)-(U(t)-2k)r(k;t).
\end{split}
\end{equation}
Differentiation with respect to $k$ using $r(k;t)^2=(k-\mathfrak{a}(t))(k-\mathfrak{b}(t))$ then gives
\begin{equation}
\phi'_t(k;t)=\frac{4k^2+2U(t)k-2H(t)^2}{r(k;t)},\quad k\in (k_\mathfrak{a},\mathfrak{a}(t))\cup(\mathfrak{b}(t),k_\mathfrak{b}).
\end{equation}
Now we apply the fundamental theorem of calculus to obtain
\begin{equation}
\phi'(k;t)=\phi'(k;t_\pm(k)) +\int_{t_\pm(k)}^t\frac{4k^2+2U(s)k-2H(s)^2}{r(k;s)}\,ds.
\label{eq:phi-prime-integral}
\end{equation}
Here the turning point $t_\pm$ is selected so that either $k<\mathfrak{a}(s)$ or $k>\mathfrak{b}(s)$ holds for all $s$ in the interval of integration.  Since $k$ lies on the boundary of the band 
$(\mathfrak{a}(t_\pm(k)),\mathfrak{b}(t_\pm(k)))$, we have $\phi'(k;t_\pm(k))=0$, so it remains to determine the sign of the integral.  First observe that $4k^2+2U(s)k-2H(s)^2$ and $r(k;s)$ always have opposite signs, regardless of whether $k<\mathfrak{a}(s)$ or $k>\mathfrak{b}(s)$.  Indeed,
if $k<\mathfrak{a}(s)$, then $r(k;s)=-\sqrt{(\mathfrak{a}(s)-k)(\mathfrak{b}(s)-k)}<0$, and by the identity 
\begin{equation}
4k^2+2U(s)k-2H(s)^2=2(\mathfrak{a}(s)-k)(\mathfrak{b}(s)-k)+\frac{1}{2}(U(s)-2k)((\mathfrak{a}(s)-k)+(\mathfrak{b}(s)-k))
\end{equation}
one has $4k^2+2U(s)k-2H(s)^2>0$.  On the other hand, if $k>\mathfrak{b}(s)$, then $r(k;s)=\sqrt{(k-\mathfrak{a}(s))(k-\mathfrak{b}(s))}>0$ and by the identity $4k^2+2U(s)k-2H(s)^2=4(k-k_+(s))(k-k_-(s))$ where $k_\pm(\cdot)$ are defined by \eqref{eq:k-plus-k-minus} one sees that $4k^2+2U(s)k-2H(s)^2<0$, because $k_+(s)>0$ which implies $k-k_+(s)<0$, and also $k_-(s)<\mathfrak{b}(s)$ which implies $k-k_-(s)>k-\mathfrak{b}(s)>0$.  Therefore, we conclude that $\phi'(k;t)>0$ if $t<t_-(k)$, which characterizes $k$ lying in a void, and $\phi'(k;t)<0$ if $t>t_+(k)$, which characterizes $k$ lying in a saturated region.  This completes the proof of the inequalities \eqref{eq:Void-Inequality} and \eqref{eq:Saturated-Region-Inequality}.
\end{proof}

The expression \eqref{eq:gt} for $g_t(k;t)$ from this proof actually leads to a complete characterization of the constant $\phi_\mathrm{B}(t)$ as the following result shows.
\begin{lemma}
Let $x=0$ and $t>0$.  Then $\phi_\mathrm{B}(t)=-\tfrac{1}{2}S(t)$.
\label{lemma-phi-B-S}
\end{lemma}
\begin{proof}
Since $g(k;t)=\mathcal{O}(k^{-1})$ as $k\to\infty$ for all $t>0$ (by \eqref{eq:g-gprime-integrate}), it follows that also $g_t(k;t)=\mathcal{O}(k^{-1})$ as $k\to\infty$.  But if we examine the explicit expression for $g_t(k;t)$ given by \eqref{eq:gt}, we observe that there is a constant leading term in the Laurent series of $g_t(k;t)$ for large $|k|$.  This constant term therefore must vanish, and this gives rise to an identity expressing $\phi_\mathrm{B}'(t)$ explicitly in terms of $\mathfrak{a}(t)$ and $\mathfrak{b}(t)$ (which can be simplified further with the help of \eqref{eq:a-b-define} and \eqref{eq:SprimeUH}):
\begin{equation}
\phi_{\mathrm{B}}'(t)=\frac{1}{4}\left(3\mathfrak{a}(t)^2+2\mathfrak{a}(t)\mathfrak{b}(t)+3\mathfrak{b}(t)^2\right)=\frac{1}{2}U(t)^2+H(t)^2=-\frac{1}{2}S'(t).
\label{eq:phi-prime-t}
\end{equation}

Therefore it remains to determine an integration constant.  Suppose that $t>0$ is sufficiently small that both $\mathfrak{a}'(t)<0$ and $\mathfrak{b}'(t)>0$, i.e., we have a VBV configuration for $g(k;t)$.  Then $g(k)=g(k;t)$ is analytic for $k\in\mathbb{C}\setminus [\alpha,\beta]$, and for $\alpha<k<\beta$ we have $\phi(k)=\phi_\mathrm{B}$, so using \eqref{eq:phi-Delta-define} we get
\begin{equation}
g_+(k)+g_-(k)=2\theta(k;x,t)-2\Phi(k)-2\phi_\mathrm{B},\quad \alpha<k<\beta.
\end{equation}
It follows from these considerations that $g(k)$ must be given by the formula
\begin{equation}
g(k)=\frac{1}{r(k)}\left[g_1 +\frac{1}{\pi\myi} \int_\alpha^\beta\frac{(\theta(l;x,t)-\Phi(l)-\phi_\mathrm{B})r_+(l)}{l-k}\,dl\right]
\end{equation}
where $g_1$ is an additional constant ($g(k)=g_1k^{-1}+\mathcal{O}(k^{-2})$ as $k\to\infty$).  But $g(k)$ is known to be bounded at the band endpoints $k=\alpha$ and $k=\beta$, so the expression in square brackets must be made to vanish for these values of $k$, resulting in a system of linear equations for $g_1$ and $\phi_\mathrm{B}$:
\begin{equation}
\begin{bmatrix}1 & -\displaystyle\frac{1}{\pi\myi}\int_\alpha^\beta\frac{r_+(l)\,dl}{l-\alpha}\\
1 & \displaystyle -\frac{1}{\pi \myi}\int_\alpha^\beta\frac{r_+(l)\,dl}{l-\beta}\end{bmatrix}
\begin{bmatrix}g_1\\\phi_\mathrm{B}\end{bmatrix}=\begin{bmatrix}
\displaystyle\frac{1}{\pi\myi}\int_\alpha^\beta\frac{(\Phi(l)-\theta(l;x,t))r_+(l)}{l-\alpha}\,dl\\
\displaystyle\frac{1}{\pi\myi}\int_\alpha^\beta\frac{(\Phi(l)-\theta(l;x,t))r_+(l)}{l-\beta}\,dl
\end{bmatrix}.
\end{equation}
The integrals in the coefficient matrix can be calculated by residues at $l=\infty$ and the system solved for $\phi_\mathrm{B}$:
\begin{equation}
\phi_\mathrm{B}=\frac{1}{\pi\myi}\int_\alpha^\beta\frac{\Phi(l)-\theta(l;x,t)}{r_+(l)}\,dl.
\end{equation}
Likewise, the integral involving $\theta(\ell;x,t)$ can be evaluated by a residue at $l=\infty$, yielding
\begin{equation}
\phi_\mathrm{B}=-\frac{1}{2}tS'(t)-\frac{1}{2}xU(t)+\frac{1}{\pi i}\int_\alpha^\beta\frac{\Phi(l)\,dl}{r_+(l)},
\end{equation}
where we have also used \eqref{eq:a-b-define} and \eqref{eq:SprimeUH}.  Finally, we set $x=0$ and consider the limit $t\downarrow 0$, in which $\alpha=\mathfrak{a}(t)$ and $\beta=\mathfrak{b}(t)$ converge to $k_0$.  Using the fact, as shown in the proof of Lemma~\ref{lemma-Phi-interior},
that $\Phi(k)=\tfrac{1}{2}S(0)+\mathcal{O}((k-k_0)^3)$ as $k\to k_0$, we replace $\Phi(l)$ in the integral by its limiting value and calculate the resulting integral by a residue at $l=\infty$.  Therefore
\begin{equation}
\lim_{t\downarrow 0}\phi_\mathrm{B}(t)=-\frac{1}{2}S(0)
\end{equation}
and the proof is complete.
\end{proof}

\begin{lemma}
Let $x=0$ and $t>0$.  The function $\phi(k;t)$ is analytic for $k_\mathfrak{a}<k<\mathfrak{a}(t)$
and for $\mathfrak{b}(t)<k<k_\mathfrak{b}$.
\label{lemma:phi-analytic-gaps}
\end{lemma}
\begin{proof}
Suppose first that the interval containing $k$ is a void.  Then $g_+(k;t)=g_-(k;t)$, so $g(k;t)$ is analytic at $k$, and we may write $\phi(k;t)=\theta(k;0,t)-\Phi(k)-g(k;t)$, which is clearly analytic except at $k=k_0$, according to Lemma~\ref{lemma-Phi-interior}.  Since $\mathfrak{a}(0)=\mathfrak{b}(0)=k_0$,
and since $\mathfrak{a}(t)$ is decreasing if $(k_\mathfrak{a},\mathfrak{a}(t))$ is a void while $\mathfrak{b}(t)$ is increasing if $(\mathfrak{b}(t),k_\mathfrak{b})$ is a void, it follows that the point $k_0$ cannot lie in a void for any $t>0$.

Next suppose that the interval containing $k$ is a saturated region.  Then $g_+(k;t)-g_-(k;t)=2\myi \tau(k)$, so we may write $\phi(k;t)$ in two alternate forms:
\begin{equation}
\phi(k;t)=\theta(k;0,t)-g_\pm(k;t)-(\Phi(k)\mp\myi\tau(k)).
\end{equation}
Therefore $\phi(k;t)$ will be the boundary value of a function analytic in $\mathbb{C}_\pm$ if this is true of the function $\Phi(k)\mp\myi\tau(k)$.  From \eqref{eq:tau-define-1} and \eqref{eq:Phi-define},
we have
\begin{equation}
\begin{split}
\Phi(k)\mp \myi\tau(k)&=\frac{1}{2}S(0)+\sgn(k^2-k_0^2)\int_0^{t_-(k)}(U(t)-2k)\sqrt{(k-\mathfrak{a}(t))(k-\mathfrak{b}(t))}\,dt\\
&\quad\quad\quad\quad{}\mp\myi\int_{t_-(k)}^{t_+(k)}(U(t)-2k)\sqrt{(k-\mathfrak{a}(t))(\mathfrak{b}(t)-k)}\,dt\\
&=\frac{1}{2}S(0)-\int_0^{t_-(k)}(U(t)-2k)r(k;t)\,dt -\int_{t_-(k)}^{t_+(k)}(U(t)-2k)r_\pm(k;t)\,dt\\
&=\frac{1}{2}S(0)-\int_0^{t_+(k)}(U(t)-2k)r_\pm(k;t)\,dt.
\end{split}
\end{equation}
But the only point of nonanalyticity of $t_+(k)$ in $(k_\mathfrak{a},k_\mathfrak{b})$ is $k=k_\infty$,
so if $k\neq k_\infty$, then $\Phi(k)\mp\myi\tau(k)$ is the boundary value of a function analytic for $k$ in the half-plane $\mathbb{C}_\pm$ near $k$.  It follows that if $k$ is in a saturated region and $k\neq k_\infty$ then
$\phi(k;t)$ is analytic at $k$, i.e., it can be continued into both half-planes.  But since $\mathfrak{a}(t)$ and $\mathfrak{b}(t)$ both tend to $k_\infty$ as $t\to\infty$ and $\mathfrak{a}(t)$ is increasing if $(k_\mathfrak{a},\mathfrak{a}(t))$ is a saturated region while $\mathfrak{b}(t)$ is decreasing if $(\mathfrak{b}(t),k_\mathfrak{b})$ is a saturated region, it is impossible for $k_\infty$ to lie in a saturated region for any $t>0$.
\end{proof}

We will refer to the analytic function $\phi(k;t)$ defined in the interval $(k_\mathfrak{a},\mathfrak{a}(t))$ (respectively, in the interval $(\mathfrak{b}(t),k_\mathfrak{b})$) as $\phi_\mathfrak{a}(k;t)$ (respectively, $\phi_\mathfrak{b}(k;t)$).  Finally, we require an analogue of Lemma~\ref{lemma-tilde-Gamma-endpoints}.
\begin{lemma}
The functions $Y^\eps(k)e^{\pm 2\myi\phi(k;t)/\eps}$ have analytic continuations into the complex plane
from right and left neighborhoods of $k_\mathfrak{a}$ and $k_\mathfrak{b}$, respectively,
and along small segments with one endpoint $k_\mathfrak{a}$ or $k_\mathfrak{b}$ and the other endpoint having real part in the interior of $(k_\mathfrak{a},k_\mathfrak{b})$ and nonzero imaginary part of the appropriate sign so that $|e^{\pm 2\myi\phi(k;t)/\eps}|\le 1$ along the segment, the uniform estimate $Y^\eps(k)e^{\pm 2\myi\phi(k;t)/\eps}=\mathcal{O}((\log(\eps^{-1}))^{-1/2})$ holds.
\label{lemma-Y-phi-endpoints}
\end{lemma}
\begin{proof}
Applying the nondegeneracy of the extrema of $\mathfrak{a}$ and $\mathfrak{b}$ guaranteed by Assumption~\ref{assumption:data}
to the formula \eqref{eq:phi-prime-integral}, it is easy to see that $\phi'(k;t)$ always diverges logarithmically as $k\downarrow k_\mathfrak{a}$ and as $k\uparrow k_\mathfrak{b}$.  Hence, by integration in $k$ one sees that up to a nonzero constant factor plus an integration constant, the leading-order behavior of $\phi(k;t)$ is the same as that of $\Phi(k)$ as established in Lemma~\ref{lemma-Phi-endpoints}.
The rest of the proof is then exactly the same as that of Lemma~\ref{lemma-tilde-Gamma-endpoints}, playing off the exponential decay of $e^{\pm 2\myi\phi(k;t)/\eps}$ into the appropriate half-plane away from $k_\mathfrak{a}$ or $k_\mathfrak{b}$ against the linear vanishing of $\tau_\mathfrak{a}$ or $\tau_\mathfrak{b}$ to establish the claimed uniform estimate.
\end{proof}

\subsubsection{Proof of Theorem~\ref{theorem:boundary-condition-recover}}
\label{sec:boundary-condition-recover-proof}
To prove Theorem~\ref{theorem:boundary-condition-recover}, 
we apply the steepest descent method to Riemann-Hilbert Problem~\ref{rhp-M-tilde} with the help of the complex phase function $g=g(k;t)$ introduced in \S\ref{sec:g-function}, that is, we exploit the transformation \eqref{eq:N-from-g} from $\tilde{\mathbf{M}}(k)$ to $\mathbf{N}(k)$ and use \eqref{eq:Void-Trick}--\eqref{eq:Saturated-Region-Trick} to handle the jump condition \eqref{eq:N-jump} by opening lenses about the voids and saturated regions.  Some minor modifications are required because $Y^\eps(k)$ is not analytic at $k=k_0$ (which may lie in a saturated region but not a void) or $k=k_\infty$ (which may lie in a void but not a saturated region), however it will not be necessary to introduce any nonanalyticity or deal with $\dbar$ problems as in the proof of Theorem~\ref{theorem:initial-condition-small}.

To open the lenses, we define domains of the complex plane as illustrated in Figure~\ref{fig:FourConfigurations} 
\begin{figure}[h]
\includegraphics{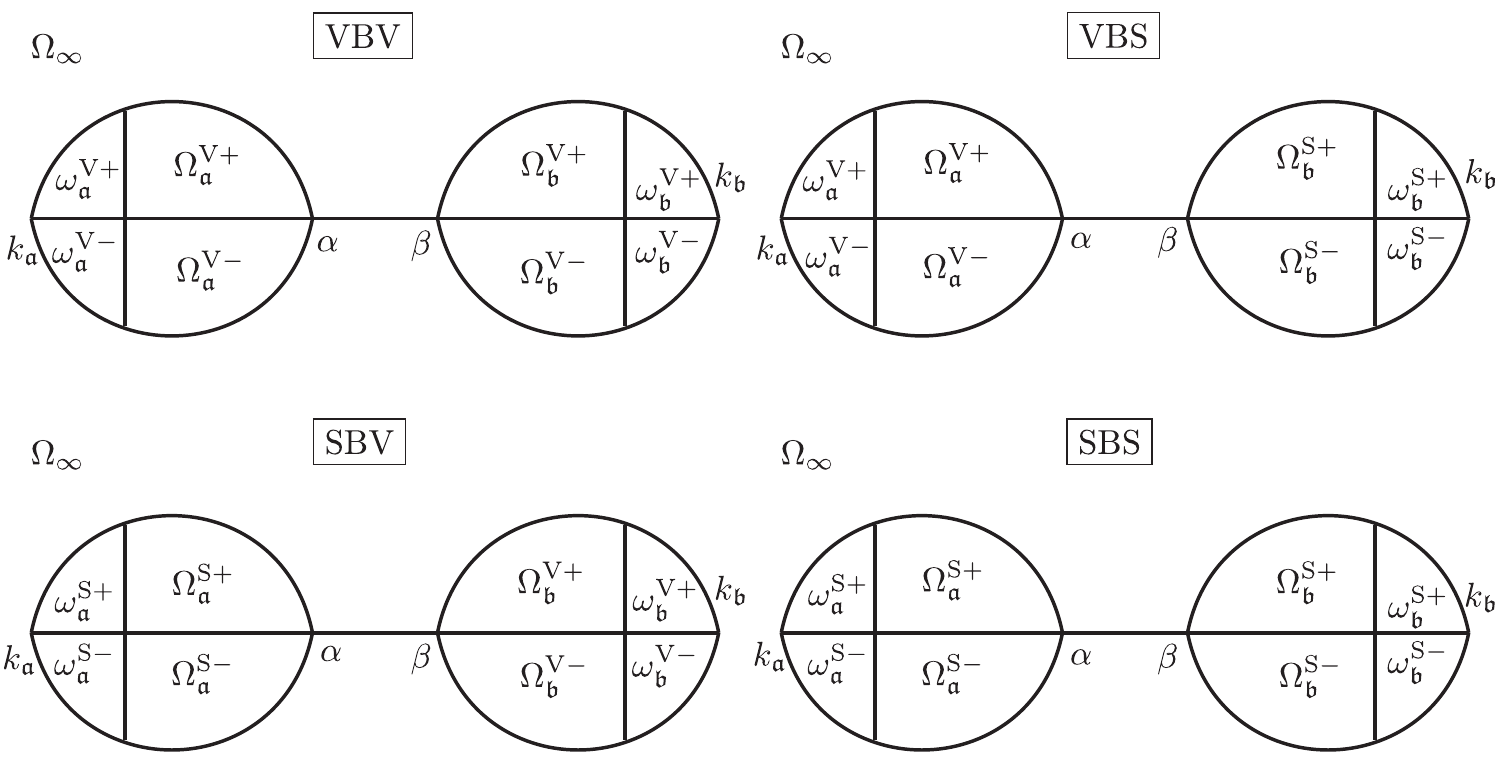}
\caption{The lens domains about the interval $[k_\mathfrak{a},k_\mathfrak{b}]$ in the four configurations of the complex phase function $g$.  Void intervals cannot contain the point $k_0$, and saturated regions cannot contain the point $k_\infty$.  The various domains labeled with the letter $\omega$ are so small as to exclude both of these points, because only near the endpoints of $[k_\mathfrak{a},k_\mathfrak{b}]$ do we need to exploit analyticity of $\tau$ in order to obtain decay along the lens boundaries without installing unusual parametrices but rather by using Lemma~\ref{lemma-Y-phi-endpoints}.}
\label{fig:FourConfigurations}
\end{figure}
and make the following explicit substitution:
\begin{equation}
\mathbf{O}(k):=\mathbf{N}(k)\begin{bmatrix}1 & 0\\-e^{2\myi\phi_{\mathfrak{a},\mathfrak{b}}(k;t)/\eps} & 1\end{bmatrix},\quad
k\in\Omega_{\mathfrak{a},\mathfrak{b}}^{\mathrm{V}+},
\end{equation}
\begin{equation}
\mathbf{O}(k):=\mathbf{N}(k)\begin{bmatrix}1 & -e^{-2\myi\phi_{\mathfrak{a},\mathfrak{b}}(k;t)/\eps}\\ 0 & 1\end{bmatrix},\quad k\in\Omega_{\mathfrak{a},\mathfrak{b}}^{\mathrm{V}-},
\end{equation}
\begin{equation}
\mathbf{O}(k):=\mathbf{N}(k)\begin{bmatrix} 1&e^{-2\myi\phi_{\mathfrak{a},\mathfrak{b}}(k;t)/\eps}\\
0 & 1\end{bmatrix},\quad k\in\Omega_{\mathfrak{a},\mathfrak{b}}^{\mathrm{S}+},
\end{equation}
\begin{equation}
\mathbf{O}(k):=\mathbf{N}(k)\begin{bmatrix} 1 & 0\\ e^{2\myi\phi_{\mathfrak{a},\mathfrak{b}}(k;t)/\eps} & 1\end{bmatrix},\quad k\in\Omega_{\mathfrak{a},\mathfrak{b}}^{\mathrm{S}-},
\end{equation}
\begin{equation}
\mathbf{O}(k):=\mathbf{N}(k)\begin{bmatrix}1 & 0\\ -Y_{\mathfrak{a},\mathfrak{b}}^\eps(k)e^{2\myi\phi_{\mathfrak{a},\mathfrak{b}}(k;t)/\eps} & 1\end{bmatrix},\quad
k\in\omega^{\mathrm{V}+}_{\mathfrak{a},\mathfrak{b}},
\end{equation}
\begin{equation}
\mathbf{O}(k):=\mathbf{N}(k)\begin{bmatrix}1 & -Y_{\mathfrak{a},\mathfrak{b}}^\eps(k)e^{-2\myi\phi_{\mathfrak{a},\mathfrak{b}}(k;t)/\eps}\\0 & 1\end{bmatrix},\quad k\in\omega_{\mathfrak{a},\mathfrak{b}}^{\mathrm{V}-},
\end{equation}
\begin{equation}
\mathbf{O}(k):=\mathbf{N}(k)\begin{bmatrix}1 & Y_{\mathfrak{a},\mathfrak{b}}^\eps(k)e^{-2\myi\phi_{\mathfrak{a},\mathfrak{b}}(k;t)/\eps}\\ 0 & 1\end{bmatrix},\quad k\in\omega_{\mathfrak{a},\mathfrak{b}}^{\mathrm{S}+},
\end{equation}
\begin{equation}
\mathbf{O}(k):=\mathbf{N}(k)\begin{bmatrix} 1 & 0\\Y_{\mathfrak{a},\mathfrak{b}}^\eps(k)
e^{2\myi\phi_{\mathfrak{a},\mathfrak{b}}(k;t)/\eps} & 1\end{bmatrix},\quad k\in\omega^{\mathrm{S}-}_{\mathfrak{a},\mathfrak{b}},
\end{equation}
and in the unbounded domain $\Omega_\infty$ we set $\mathbf{O}(k):=\mathbf{N}(k)$.  The matrix $\mathbf{O}(k)$ satisfies the conditions of the following Riemann-Hilbert problem.
\begin{rhp}
Find a $2\times 2$ matrix $\mathbf{O}(k)$ with the following properties:
\begin{itemize}
\item[]\textbf{Analyticity:}  $\mathbf{O}(k)$ is analytic for $k\in\mathbb{C}\setminus\Sigma$, where $\Sigma$ is the contour illustrated in Figure~\ref{fig:FourContours}
\begin{figure}[h]
\includegraphics{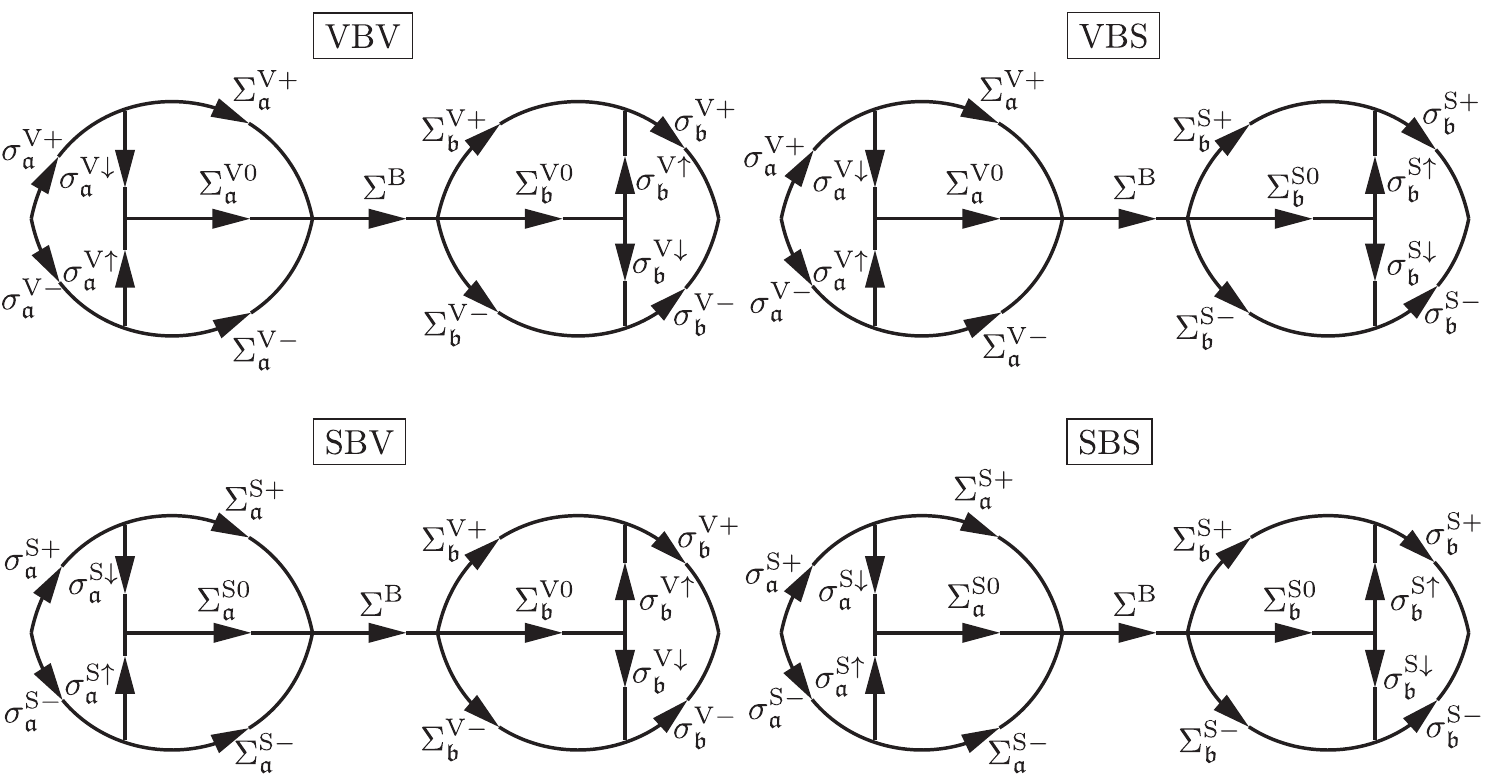}
\caption{The oriented arcs of the contour $\Sigma$ for Riemann-Hilbert Problem~\ref{rhp:t-positive-open-lenses} in the four cases for the complex phase function $g$.}
\label{fig:FourContours}
\end{figure}
and $\mathbf{O}$ takes continuous boundary values on each oriented arc of $\Sigma$, $\mathbf{O}_+(k)$ from the left and $\mathbf{O}_-(k)$ from the right.
\item[]\textbf{Jump Condition:}  The boundary values on each oriented arc of $\Sigma$ are related by $\mathbf{O}_+(k)=\mathbf{O}_-(k)\mathbf{J}(k)$ (see below for the definition of $\mathbf{J}$).
\item[]\textbf{Normalization:}  $\mathbf{O}(k)\to\mathbb{I}$ as $k\to\infty$.
\end{itemize}
\label{rhp:t-positive-open-lenses}
\end{rhp}
The jump matrix $\mathbf{J}$ is defined on $\Sigma$ as follows:
\begin{equation}
\mathbf{J}(k):=\begin{bmatrix}1 & 0\\Y_{\mathfrak{a},\mathfrak{b}}^\eps(k)e^{2\myi\phi_{\mathfrak{a},\mathfrak{b}}(k;t)/\eps} & 1
\end{bmatrix},\quad k\in\sigma_{\mathfrak{a},\mathfrak{b}}^{\mathrm{V}+}\cup\sigma_{\mathfrak{a},\mathfrak{b}}^{\mathrm{S}-},
\end{equation}
\begin{equation}
\mathbf{J}(k):=\begin{bmatrix}1 & -Y_{\mathfrak{a},\mathfrak{b}}^\eps(k)e^{-2\myi\phi_{\mathfrak{a},\mathfrak{b}}(k;t)/\eps}\\0 & 1\end{bmatrix},\quad k\in\sigma_{\mathfrak{a},\mathfrak{b}}^{\mathrm{V}-}\cup\sigma_{\mathfrak{a},\mathfrak{b}}^{\mathrm{S}+},
\end{equation}
\begin{equation}
\mathbf{J}(k):=\begin{bmatrix}1 & 0\\e^{2\myi\phi_{\mathfrak{a},\mathfrak{b}}(k;t)/\eps} & 1
\end{bmatrix},\quad k\in\Sigma_{\mathfrak{a},\mathfrak{b}}^{\mathrm{V}+}\cup\Sigma_{\mathfrak{a},\mathfrak{b}}^{\mathrm{S}-},
\end{equation}
\begin{equation}
\mathbf{J}(k):=\begin{bmatrix}1 & -e^{-2\myi\phi_{\mathfrak{a},\mathfrak{b}}(k;t)/\eps}\\0 & 1\end{bmatrix},\quad k\in\Sigma_{\mathfrak{a},\mathfrak{b}}^{\mathrm{V}-}\cup\Sigma_{\mathfrak{a},\mathfrak{b}}^{\mathrm{S}+},
\end{equation}
\begin{equation}
\mathbf{J}(k):=\begin{bmatrix}1 & 0\\(Y_{\mathfrak{a},\mathfrak{b}}^\eps(k)-1)e^{2\myi\phi_{\mathfrak{a},\mathfrak{b}}(k;t)/\eps} & 1
\end{bmatrix},\quad k\in\sigma_{\mathfrak{a},\mathfrak{b}}^{\mathrm{V}\downarrow}\cup\sigma_{\mathfrak{a},\mathfrak{b}}^{\mathrm{S}\downarrow},
\end{equation}
\begin{equation}
\mathbf{J}(k):=\begin{bmatrix}1 & (1-Y_{\mathfrak{a},\mathfrak{b}}^\eps(k))e^{-2\myi\phi_{\mathfrak{a},\mathfrak{b}}(k;t)/\eps}\\0 & 1\end{bmatrix},\quad k\in\sigma_{\mathfrak{a},\mathfrak{b}}^{\mathrm{V}\uparrow}\cup\sigma_{\mathfrak{a},\mathfrak{b}}^{\mathrm{S}\uparrow},
\end{equation}
\begin{equation}
\mathbf{J}(k):=\begin{bmatrix}1 & (1-Y^\eps(k))e^{-2\myi\phi(k;t)/\eps}\\0 & 1\end{bmatrix}
\begin{bmatrix}1 & 0\\ (Y^\eps(k)-1)e^{2\myi\phi(k;t)/\eps} & 1\end{bmatrix},\quad k\in
\Sigma_\mathfrak{a}^{\mathrm{V}0}\cup\Sigma_\mathfrak{b}^{\mathrm{V}0},
\end{equation}
\begin{equation}
\mathbf{J}(k):=
\begin{bmatrix}1 & 0\\(Y^\eps(k)-1)e^{2\myi\phi(k;t)/\eps} & 1\end{bmatrix}\begin{bmatrix}1 & (1-Y^\eps(k))e^{-2\myi\phi(k;t)/\eps}\\0 & 1\end{bmatrix},\quad k\in
\Sigma_\mathfrak{a}^{\mathrm{S}0}\cup\Sigma_\mathfrak{b}^{\mathrm{S}0},
\end{equation}
and
\begin{equation}
\mathbf{J}(k):=\begin{bmatrix}e^{2(\Delta(k;t)-\tau(k))/\eps} & -Y^\eps(k)e^{-2\myi\phi_\mathrm{B}(t)/\eps}\\
Y^\eps(k)e^{2\myi\phi_\mathrm{B}(t)/\eps} & e^{-2\Delta(k;t)/\eps}\end{bmatrix},\quad k\in\Sigma^\mathrm{B}.
\end{equation}
It follows from \eqref{eq:Void-Inequality} and \eqref{eq:Saturated-Region-Inequality} in Lemma~\ref{lemma:inequalities-x-zero}, from Lemma~\ref{lemma:phi-analytic-gaps},  and from Lemma~\ref{lemma-Y-phi-endpoints} that $\mathbf{J}-\mathbb{I}$ is uniformly small on $\sigma_{\mathfrak{a},\mathfrak{b}}^{\mathrm{V}\pm}\cup\sigma_{\mathfrak{a},\mathfrak{b}}^{\mathrm{S}\pm}\cup\Sigma_{\mathfrak{a},\mathfrak{b}}^{\mathrm{V}\pm}\cup\Sigma_{\mathfrak{a},\mathfrak{b}}^{\mathrm{S}\pm}$ (that is, all non-vertical and non-horizontal arcs of $\Sigma$) omitting only neighborhoods of the band endpoints $\alpha=\mathfrak{a}(t)$ and $\beta=\mathfrak{b}(t)$.  The rate of decay is determined from neighborhoods of $k_\mathfrak{a}$ and $k_\mathfrak{b}$ according to Lemma~\ref{lemma-Y-phi-endpoints}, namely $\mathcal{O}((\log(\eps^{-1}))^{-1/2})$, but away from these points one has exponential decay.    Likewise,
from the fact that $\tau(k)>0$ is uniformly bounded away from zero on compact subsets of $(k_\mathfrak{a},k_\mathfrak{b})$ means that $\mathbf{J}-\mathbb{I}$ is also uniformly exponentially small on $\Sigma_{\mathfrak{a},\mathfrak{b}}^{\mathrm{V}0}\cup\Sigma_{\mathfrak{a},\mathfrak{b}}^{\mathrm{S}0}\cup\sigma_{\mathfrak{a},\mathfrak{b}}^{\mathrm{V}\downarrow}\cup\sigma_{\mathfrak{a},\mathfrak{b}}^{\mathrm{V}\uparrow}\cup\sigma_{\mathfrak{a},\mathfrak{b}}^{\mathrm{S}\downarrow}\cup\sigma_{\mathfrak{a},\mathfrak{b}}^{\mathrm{S}\uparrow}$, that is,
on all vertical arcs of $\Sigma$ and on all horizontal arcs except the band $\mathrm{B}$.

Riemann-Hilbert Problem~\ref{rhp:t-positive-open-lenses} is \emph{not}, however, a small-norm problem in the semiclassical limit $\eps\downarrow 0$, because $\mathbf{J}-\mathbb{I}$ is not decaying with $\eps$ on the band $B$, nor is the decay on the non-real arcs of $\Sigma$ that meet at $k=\alpha$ and $k=\beta$ uniform near these band endpoints.  We will now remedy this situation by constructing an explicit parametrix for $\mathbf{O}(k)$ in a standard fashion.  First, we exhibit a matrix solving the limiting form of the jump condition on the band $\mathrm{B}$:
\begin{equation}
\dot{\mathbf{O}}^{(\mathrm{out})}(k):=e^{-\myi\phi_\mathrm{B}(t)\sigma_3/\eps}
\mathbf{S}\left(\frac{(k-\beta)^{1/4}}{(k-\alpha)^{1/4}}\right)^{\sigma_3}
\mathbf{S}^{-1}
e^{\myi\phi_\mathrm{B}(t)\sigma_3/\eps},\quad \mathbf{S}:=\frac{1}{\sqrt{2}}\begin{bmatrix}1 & -\myi\\ -\myi & 1\end{bmatrix},
\end{equation}
where $(k-\beta)^{1/4}$ and $(k-\alpha)^{1/4}$ denote the principal branches (and hence the ratio may be considered to be well-defined on the interval $k<\alpha$ common to both branch cuts).  It is easy to confirm that this \emph{outer parametrix} has the following properties:
\begin{itemize}
\item $\dot{\mathbf{O}}^{(\mathrm{out})}(k)$ is analytic for $k\in\mathbb{C}\setminus [\alpha,\beta]$,
\item $\dot{\mathbf{O}}^{(\mathrm{out})}(k)$ satisfies the jump condition
\begin{equation}
\dot{\mathbf{O}}^{(\mathrm{out})}_+(k)=\dot{\mathbf{O}}^{(\mathrm{out})}_-(k)\begin{bmatrix}0 & -e^{-2\myi\phi_\mathrm{B}(t)/\epsilon}\\e^{2\myi\phi_\mathrm{B}(t)/\epsilon} & 0\end{bmatrix},\quad \alpha<k<\beta,
\end{equation}
\item $\det(\dot{\mathbf{O}}^{(\mathrm{out})}(k))=1$,
\item $\dot{\mathbf{O}}^{(\mathrm{out})}(k)\to\mathbb{I}$ as $k\to\infty$, and
\item For $k$ bounded away from $\{\alpha,\beta\}$, $\dot{\mathbf{O}}^{(\mathrm{out})}(k)$ is uniformly bounded independent of $\eps$.
\end{itemize}
The outer parametrix will turn out to be a good approximation of $\mathbf{O}(k)$ away from the 
points $k=\alpha,\beta$, but it blows up at these two points and fails to even approximately satisfy the non-negligible jump conditions on the complex contours nearby.  For now, we record what will be a useful formula later on:
\begin{equation}
\dot{\mathbf{O}}^{(\mathrm{out})}(k)=\mathbb{I} +k^{-1}\dot{\mathbf{O}}_1+k^{-2}\dot{\mathbf{O}}_2 +\mathcal{O}(k^{-3}),\quad k\to\infty,
\label{eq:dot-O-out-expansion}
\end{equation}
where
\begin{equation}
\begin{split}
\dot{\mathbf{O}}_1&:=\frac{\beta-\alpha}{4}\begin{bmatrix}0 & -\myi e^{-2\myi\phi_\mathrm{B}(t)/\eps}\\\myi e^{2\myi\phi_\mathrm{B}(t)/\eps} & 0\end{bmatrix}\\
\dot{\mathbf{O}}_2&:=\frac{\beta-\alpha}{32}\begin{bmatrix}\beta-\alpha & -4\myi (\beta+\alpha)e^{-2\myi\phi_\mathrm{B}(t)/\eps}\\4\myi(\beta+\alpha)e^{2\myi\phi_\mathrm{B}(t)/\eps} &
\beta-\alpha\end{bmatrix}.
\end{split}
\label{eq:dot-O-moments}
\end{equation}

Let $D_\alpha$ and $D_\beta$ be open disks centered at $k=\alpha$ and $k=\beta$ respectively, with radius $\delta>0$ sufficiently small, but independent of $\eps$.  We shall construct an \emph{inner parametrix} in each of these disks, an approximation that will locally be far superior to the outer parametrix.

First consider the disk $D_\alpha$.  If 
the interval $(k_\mathfrak{a},\alpha)$, $\alpha=\mathfrak{a}(t)$, is a void $\mathrm{V}$, then  we claim that the function $w_\mathfrak{a}^\mathrm{V}(k;t)$ defined for $k\in\mathrm{B}$ near $\alpha$ by
$w_\mathfrak{a}^\mathrm{V}(k;t):=(2\Delta(k;t))^{2/3}$ (positive $2/3$ power) can be analytically continued to a full complex neighborhood of $k=\alpha$ as a function that satisfies $w_\mathfrak{a}^{\mathrm{V}\prime}(\alpha;t)>0$.  This is a simple consequence of the fact that $\Delta(\mathfrak{a}(t);t)=0$ and that $\Delta'(k;t)$ vanishes like a square root and no higher power at $k=\mathfrak{a}(t)$.  It follows that $w_\mathfrak{a}^\mathrm{V}(\cdot;t)$ defines a conformal mapping from $D_\alpha$ onto a neighborhood of the origin.  The outer parametrix may be represented locally near $k=\alpha$ in terms of $w_\mathfrak{a}^\mathrm{V}(k;t)$ as follows:
\begin{equation}
\dot{\mathbf{O}}^{(\mathrm{out})}(k)=\mathbf{H}_\mathfrak{a}^\mathrm{V}(k)(-w_\mathfrak{a}^\mathrm{V}(k;t))^{-\sigma_3/4}\mathbf{S}^{-1}e^{\myi\phi_\mathrm{B}(t)\sigma_3/\eps},
\end{equation}
where $\mathbf{H}_\mathfrak{a}^\mathrm{V}(k)$ is a well-defined unimodular matrix function holomorphic near $k=\alpha$ that is obviously uniformly bounded in $D_\alpha$ independent of $\eps$.  It will be useful later to write this in the equivalent form
\begin{equation}
\dot{\mathbf{O}}^{(\mathrm{out})}(k)=\mathbf{H}_\mathfrak{a}^\mathrm{V}(k)\eps^{-\sigma_3/6}
(-\zeta)^{-\sigma_3/4}\mathbf{S}^{-1}e^{\myi\phi_\mathrm{B}(t)\sigma_3/\eps},\quad k\in D_\alpha,\quad
\zeta:=\eps^{-2/3}w_\mathfrak{a}^\mathrm{V}(k;t).
\end{equation}
Let an auxiliary matrix function $\mathbf{P}(\zeta)$ be defined as follows ($\xi:=(\tfrac{3}{4})^{2/3}\zeta$):
\begin{equation}
\mathbf{P}(\zeta):=\sqrt{2\pi}\left(\frac{3}{4}\right)^{\tfrac{1}{6}\sigma_3}e^{-\tfrac{1}{4}\pi\myi\sigma_3}
\begin{bmatrix}
-\myi e^{\tfrac{2}{3}\pi\myi}\mathrm{Ai}(\xi e^{\tfrac{2}{3}\pi\myi}) & -\myi e^{-\tfrac{2}{3}\pi\myi}\mathrm{Ai}(\xi e^{-\tfrac{2}{3}\pi\myi})\\
e^{-\tfrac{2}{3}\pi\myi}\mathrm{Ai}'(\xi e^{\tfrac{2}{3}\pi\myi}) & e^{\tfrac{2}{3}\pi\myi}\mathrm{Ai}'(\xi e^{-\tfrac{2}{3}\pi\myi})\end{bmatrix}
e^{-\tfrac{2}{3}\myi(-\xi)^{3/2}\sigma_3},\quad |\arg(-\zeta)|<\frac{\pi}{3},
\end{equation}
\begin{equation}
\mathbf{P}(\zeta):=
\sqrt{2\pi}\left(\frac{3}{4}\right)^{\tfrac{1}{6}\sigma_3}e^{-\tfrac{1}{4}\pi\myi\sigma_3}\begin{bmatrix}
\myi\mathrm{Ai}(\xi) & -\myi e^{-\tfrac{2}{3}\pi\myi}\mathrm{Ai}(\xi e^{-\tfrac{2}{3}\pi\myi})\\
-\mathrm{Ai}'(\xi) & e^{\tfrac{2}{3}\pi\myi}\mathrm{Ai}'(\xi e^{-\tfrac{2}{3}\pi \myi})\end{bmatrix}
e^{\tfrac{2}{3}\xi^{3/2}\sigma_3},\quad 0<\arg(\zeta)<\frac{2\pi}{3},
\end{equation}
and
\begin{equation}
\mathbf{P}(\zeta):=\sqrt{2\pi}\left(\frac{3}{4}\right)^{\tfrac{1}{6}\sigma_3}e^{-\tfrac{1}{4}\pi\myi\sigma_3}
\begin{bmatrix} -\myi e^{\tfrac{2}{3}\pi\myi}\mathrm{Ai}(\xi e^{\tfrac{2}{3}\pi \myi}) & \myi\mathrm{Ai}(\xi)\\
e^{-\tfrac{2}{3}\pi \myi}\mathrm{Ai}'(\xi e^{\tfrac{2}{3}\pi\myi}) & -\mathrm{Ai}'(\xi)\end{bmatrix}
e^{-\tfrac{2}{3}\xi^{3/2}\sigma_3},\quad -\frac{2\pi}{3}<\arg(\zeta)<0.
\end{equation}
Using well-documented formulae involving the Airy function $\mathrm{Ai}$ and its derivative \cite{DLMF}, it follows that $\mathbf{P}(\zeta)$ is analytic in the three sectors of its definition,
that across the rays bounding the sectors one has
\begin{equation}
\lim_{\mu\downarrow 0}\mathbf{P}(\zeta e^{\myi\mu}) = \lim_{\mu\downarrow 0}\mathbf{P}(\zeta e^{-\myi \mu})\begin{bmatrix}0 & -1\\ 1 & e^{-\zeta^{3/2}}\end{bmatrix},\quad \arg(\zeta)=0,
\end{equation}
\begin{equation}
\lim_{\mu\downarrow 0}\mathbf{P}(\zeta e^{\myi\mu}) = \lim_{\mu\downarrow 0}\mathbf{P}(\zeta e^{-\myi \mu})\begin{bmatrix}1 & 0\\-e^{\zeta^{3/2}} & 1\end{bmatrix},\quad \arg(\zeta)=\frac{2\pi}{3},
\end{equation}
and
\begin{equation}
\lim_{\mu\downarrow 0}\mathbf{P}(\zeta e^{\myi\mu}) = \lim_{\mu\downarrow 0}\mathbf{P}(\zeta e^{-\myi \mu})\begin{bmatrix}1 & e^{\zeta^{3/2}}\\0 & 1\end{bmatrix},\quad \arg(\zeta)=-\frac{2\pi}{3},
\end{equation}
and that 
\begin{equation}
\mathbf{P}(\zeta)\mathbf{U}(-\zeta)^{\sigma_3/4}=\mathbb{I}+\begin{bmatrix}
\mathcal{O}(\zeta^{-3}) & \mathcal{O}(\zeta^{-2})\\
\mathcal{O}(\zeta^{-1}) & \mathcal{O}(\zeta^{-3})\end{bmatrix},\quad \zeta\to\infty,
\label{eq:Airy-asymptotic}
\end{equation}
with the asymptotics being uniform with respect to direction in the complex plane, including along the sector boundary rays.
Then, in terms of $\mathbf{P}(\zeta)$ and $\mathbf{H}_\mathfrak{a}^\mathrm{V}(k)$ we define an inner parametrix near $k=\alpha$ by setting
\begin{equation}
\dot{\mathbf{O}}_\mathfrak{a}^\mathrm{V}(k):=\mathbf{H}_\mathfrak{a}^\mathrm{V}(k)
\eps^{-\sigma_3/6}
\mathbf{P}(\eps^{-2/3}w_\mathfrak{a}^\mathrm{V}(k;t))e^{\myi\phi_\mathrm{B}(t)\sigma_3/\eps},\quad
k\in D_\alpha.
\end{equation}
Assuming without loss of generality that for $k\in D_\alpha$ the contours $\Sigma_\mathfrak{a}^{\mathrm{V}\pm}$ coincide with radial segments in the $w_\mathfrak{a}^\mathrm{V}$-plane with angles $\pm 2\pi/3$, one can check that the following facts hold true:
\begin{itemize}
\item $\dot{\mathbf{O}}_\mathfrak{a}^\mathrm{V}(k)$ is analytic for $k\in D_\alpha\setminus (\Sigma_\mathfrak{a}^{\mathrm{V}+}\cup\Sigma_\mathfrak{a}^{\mathrm{V}-}\cup\Sigma^\mathrm{B})$ (there is no jump across $\Sigma_\mathfrak{a}^{\mathrm{V}0}$).
\item $\dot{\mathbf{O}}_\mathfrak{a}^\mathrm{V}(k)$ satisfies the jump conditions
\begin{equation}
\dot{\mathbf{O}}_{\mathfrak{a}+}^\mathrm{V}(k)=\dot{\mathbf{O}}_{\mathfrak{a}-}^\mathrm{V}(k)
\begin{bmatrix}1 & 0\\e^{2\myi\phi_\mathfrak{a}(k;t)/\eps} & 1\end{bmatrix},\quad k\in\Sigma_\mathfrak{a}^{\mathrm{V}+}\cap D_\alpha,
\end{equation}
\begin{equation}
\dot{\mathbf{O}}_{\mathfrak{a}+}^\mathrm{V}(k)=\dot{\mathbf{O}}_{\mathfrak{a}-}^\mathrm{V}(k)
\begin{bmatrix}1 & -e^{-2\myi\phi_\mathfrak{a}(k;t)/\eps}\\0 & 1\end{bmatrix},\quad
k\in\Sigma_\mathfrak{a}^{\mathrm{V}-}\cap D_\alpha,
\end{equation}
and
\begin{equation}
\dot{\mathbf{O}}_{\mathfrak{a}+}^\mathrm{V}(k)=\dot{\mathbf{O}}_{\mathfrak{a}-}^\mathrm{V}(k)
\begin{bmatrix}0 & -e^{-2\myi\phi_\mathrm{B}(t)/\eps}\\
e^{2\myi\phi_\mathrm{B}(t)/\eps} & e^{-2\Delta(k;t)/\eps}\end{bmatrix},\quad k\in \Sigma^\mathrm{B}\cap D_\alpha.
\end{equation}
\item Inner and outer parametrices match well on the disk boundary:
\begin{equation}
\dot{\mathbf{O}}_\mathfrak{a}^\mathrm{V}(k)\dot{\mathbf{O}}^{(\mathrm{out})}(k)^{-1}=
\mathbb{I}+\mathcal{O}(\eps),\quad k\in\partial D_\alpha.
\label{eq:good-match-a-V}
\end{equation}
\item $\det(\dot{\mathbf{O}}_\mathfrak{a}^\mathrm{V}(k))=1$ and $\dot{\mathbf{O}}_\mathfrak{a}^{\mathrm{V}}(k)=\mathcal{O}(\eps^{-1/6})$ holds uniformly for $k\in D_\alpha$.
\end{itemize}
To check the jump conditions one should first express the analytic function $\phi_\mathfrak{a}(k;t)$ on the arcs $\Sigma_\mathfrak{a}^{\mathrm{V}\pm}$ in terms of the conformal coordinate $w_\mathfrak{a}^\mathrm{V}(k;t)$ (this has already been done for $\Delta(k;t)$, really by definition).  This is accomplished by noting   
that  
\begin{equation}
\begin{split}
g_\pm(k;t)+\Phi(k)-\theta(k;x,t)&=\frac{1}{2}(g_+(k;t)+g_-(k;t))\pm\frac{1}{2}(g_+(k;t)-g_-(k;t))+\Phi(k)-\theta(k;x,t) \\ &= 
-\phi(k;t) \pm \myi\Delta(k;t)
\end{split}
\label{eq:branch-point-continue}
\end{equation}
holds at every point of $(k_\mathfrak{a},k_\mathfrak{b})$.  Therefore, since $(k_\mathfrak{a},\mathfrak{a}(t))$ is a void $\mathrm{V}$, then for $t>0$, $\Phi$ is analytic at $k=\mathfrak{a}(t)$
and it follows that the function $-\phi_\mathfrak{a}(k;t)$ for $k_\mathfrak{a}<k<\mathfrak{a}(t)$
is the analytic continuation through $\mathbb{C}_\pm$ of the function $-\phi_\mathrm{B}(t)\pm \myi\Delta(k;t)$ for $\mathfrak{a}(t)<k<\mathfrak{b}(t)$.  Thus, for $k\in\Sigma^\mathrm{B}\cap D_\alpha$ we have $2\Delta(k;t)=w_\mathfrak{a}^\mathrm{V}(k;t)^{3/2}$, and for $k\in\Sigma_\mathfrak{a}^{\mathrm{V}\pm}\cap D_\alpha$ we have $\pm 2\myi(\phi_\mathfrak{a}(k;t)-\phi_\mathrm{B}(t))=w_\mathfrak{a}^\mathrm{V}(k;t)^{3/2}$.  To confirm the matching between the inner and outer parametrices, one uses the asymptotic condition \eqref{eq:Airy-asymptotic} and the fact that $k\in\partial D_\alpha$ means $|\zeta|\sim\eps^{-2/3}$.  The $\mathcal{O}(\eps^{-1/6})$ bound within the disk can be proved similarly.

If instead $(k_\mathfrak{a},\alpha)$ is a saturated region $\mathrm{S}$, then one defines a conformal mapping $w_\mathfrak{a}^\mathrm{S}(k;t)$ taking $D_\alpha$ to a neighborhood of the origin by the analytic continuation from $k\in\Sigma^\mathrm{B}$ of 
$w_\mathfrak{a}^\mathrm{S}(k;t):=(2\tau(k)-2\Delta(k;t))^{2/3}>0$.  One next defines a uniformly bounded unimodular holomorphic matrix function $\mathbf{H}_\mathfrak{a}^\mathrm{S}(k)$ near $k=\alpha$ by writing the outer parametrix in the form
\begin{equation}
\dot{\mathbf{O}}^{(\mathrm{out})}(k)=\mathbf{H}_\mathfrak{a}^\mathrm{S}(k)
(-w_\mathfrak{a}^\mathrm{S}(k;t))^{-\sigma_3/4}\mathbf{S}^{-1}e^{\myi\phi_\mathrm{B}\sigma_3/\eps},\quad k\in D_\alpha.
\end{equation}
Finally, one defines
\begin{equation}
\dot{\mathbf{O}}_\mathfrak{a}^\mathrm{S}(k):=\mathbf{H}_\mathfrak{a}^\mathrm{S}(k)
\eps^{-\sigma_3/6}(\myi\sigma_3)\mathbf{P}(\eps^{-2/3}w_\mathfrak{a}^\mathrm{S}(k;t))(\myi\sigma_2) e^{\myi\phi_\mathrm{B}(t)\sigma_3/\eps},\quad k\in D_\alpha.
\end{equation}
Assuming that for $k\in D_\alpha$ the contours $\Sigma_\mathfrak{a}^{\mathrm{S}\pm}$
coincide with $\arg(w_\mathfrak{a}^\mathrm{S}(k;t))=\pm 2\pi/3$, one has the following.
\begin{itemize}
\item $\dot{\mathbf{O}}_\mathfrak{a}^\mathrm{S}(k)$ is analytic for $k\in D_\alpha\setminus(\Sigma_\mathfrak{a}^{\mathrm{S}+}\cup\Sigma_\mathfrak{a}^{\mathrm{S}-}\cup\Sigma^\mathrm{B})$ (there is no jump across $\Sigma_\mathfrak{a}^{\mathrm{S}0}$).
\item $\dot{\mathbf{O}}_\mathfrak{a}^\mathrm{S}(k)$ satisfies the jump conditions
\begin{equation}
\dot{\mathbf{O}}_{\mathfrak{a}+}^\mathrm{S}(k)=\dot{\mathbf{O}}_{\mathfrak{a}-}^\mathrm{S}(k)
\begin{bmatrix}1 & -e^{-2\myi\phi_\mathfrak{a}(k;t)/\eps}\\0 & 1\end{bmatrix},\quad
k\in\Sigma_\mathbf{a}^{\mathrm{S}+}\cap D_\alpha,
\end{equation}
\begin{equation}
\dot{\mathbf{O}}_{\mathfrak{a}+}^\mathrm{S}(k)=\dot{\mathbf{O}}_{\mathfrak{a}-}^\mathrm{S}(k)
\begin{bmatrix}1 &0\\e^{2\myi\phi_\mathfrak{a}(k;t)/\eps} & 1\end{bmatrix},\quad
k\in\Sigma_\mathbf{a}^{\mathrm{S}-}\cap D_\alpha,
\end{equation}
and
\begin{equation}
\dot{\mathbf{O}}_{\mathfrak{a}+}^\mathrm{S}(k)=\dot{\mathbf{O}}_{\mathfrak{a}-}^\mathrm{S}(k)
\begin{bmatrix}e^{2(\Delta(k;t)-\tau(k))/\eps} & -e^{2\myi\phi_\mathrm{B}(t)/\eps}\\
e^{2\myi\phi_\mathrm{B}(t)/\eps} & 0
\end{bmatrix},\quad
k\in\Sigma^\mathrm{B}\cap D_\alpha.
\end{equation}
\item Inner and outer parametrices match well on the disk boundary:
\begin{equation}
\dot{\mathbf{O}}_\mathfrak{a}^\mathrm{S}(k)\dot{\mathbf{O}}^{(\mathrm{out})}(k)^{-1}=\mathbb{I}+\mathcal{O}(\eps),\quad k\in\partial D_\alpha.
\label{eq:good-match-a-S}
\end{equation}
\item $\det(\dot{\mathbf{O}}_\mathfrak{a}^\mathrm{S}(k))=1$ and $\dot{\mathbf{O}}_\mathfrak{a}^\mathrm{S}(k)=\mathcal{O}(\eps^{-1/6})$ holds uniformly for $k\in D_\alpha$.
\end{itemize}
The proof is virtually the same as before, with the main difference being that the identity \eqref{eq:branch-point-continue} implies that if
$(k_\mathfrak{a},\mathfrak{a}(t))$
is a saturated region $\mathrm{S}$, then the function $-\phi_\mathfrak{a}(k;t)\pm \myi\tau(k)$ for $k\in \mathrm{S}$ is the analytic continuation through $\mathbb{C}_\pm$ of the function $-\phi_\mathrm{B}(t)\pm \myi\Delta(k;t)$.  

Now consider the disk $D_\beta$.
If $(\beta,k_\mathfrak{b})$ is a void $\mathrm{V}$, we define a conformal map $w_\mathfrak{b}^\mathrm{V}:D_\beta\to\mathbb{C}$
by continuation from $\mathrm{B}$ of the formula $w_\mathfrak{b}^\mathrm{V}(k;t):=(2\Delta(k;t))^{2/3}>0$.  Introduce a uniformly bounded, unimodular, and holomorphic matrix $\mathbf{H}_\mathfrak{b}^\mathrm{V}(k)$ for $k\in D_\beta$ by writing
\begin{equation}
\dot{\mathbf{O}}^{(\mathrm{out})}(k)=\mathbf{H}_\mathfrak{b}^\mathrm{V}(k)(-w_\mathfrak{b}^\mathrm{V}(k;t))^{\sigma_3/4}\mathbf{S}^{-1}e^{\myi\phi_\mathrm{B}(t)\sigma_3/\eps},\quad k\in D_\beta.
\end{equation}
Then define an inner parametrix in $D_\beta$ by the formula
\begin{equation}
\dot{\mathbf{O}}_\mathfrak{b}^\mathrm{V}(k):=\mathbf{H}_\mathfrak{b}^\mathrm{V}(k)\eps^{\sigma_3/6}(-\myi\sigma_1)\mathbf{P}(\eps^{-2/3}w_\mathfrak{b}^\mathrm{V}(k))(\myi\sigma_1)
e^{\myi\phi_\mathrm{B}(t)\sigma_3/\eps},\quad k\in D_\beta.
\end{equation}
Assuming that for $k\in D_\beta$ the contours $\Sigma_\mathfrak{b}^{\mathrm{V}\pm}$ coincide
with segments with angles $\arg(w_\mathfrak{b}^\mathrm{V})=\mp 2\pi/3$, this parametrix satisfies the following conditions.
\begin{itemize}
\item $\dot{\mathbf{O}}_\mathfrak{b}^\mathrm{V}(k)$  is analytic for $k\in D_\beta\setminus (\Sigma_\mathfrak{b}^{\mathrm{V}+}\cup\Sigma_\mathfrak{b}^{\mathrm{V}-}\cup
\Sigma^\mathrm{B})$ (there is no jump across $\Sigma_\mathrm{b}^{\mathrm{V}0}$).
\item $\dot{\mathbf{O}}_\mathfrak{b}^\mathrm{V}(k)$ satisfies the jump conditions
\begin{equation}
\dot{\mathbf{O}}_{\mathfrak{b}+}^\mathrm{V}(k)=
\dot{\mathbf{O}}_{\mathfrak{b}-}^\mathrm{V}(k)\begin{bmatrix}
0 & -e^{-2\myi\phi_\mathrm{B}(t)/\eps}\\
e^{2\myi\phi_\mathrm{B}(t)/\eps} & e^{-2\Delta(k;t)/\eps}\end{bmatrix},\quad
k\in\Sigma^\mathrm{B}\cap D_\beta,
\end{equation}
\begin{equation}
\dot{\mathbf{O}}_{\mathfrak{b}+}^\mathrm{V}(k)=
\dot{\mathbf{O}}_{\mathfrak{b}-}^\mathrm{V}(k)\begin{bmatrix}
1 & 0\\
e^{2\myi\phi_\mathfrak{b}(k;t)/\eps} & 1\end{bmatrix},\quad
k\in\Sigma_\mathfrak{b}^{\mathrm{V}+}\cap D_\beta,
\end{equation}
and
\begin{equation}
\dot{\mathbf{O}}_{\mathfrak{b}+}^\mathrm{V}(k)=
\dot{\mathbf{O}}_{\mathfrak{b}-}^\mathrm{V}(k)\begin{bmatrix}
1 & -e^{-2\myi\phi_\mathfrak{b}(k;t)/\eps}\\
0 & 1\end{bmatrix},\quad
k\in\Sigma_\mathfrak{b}^{\mathrm{V}-}\cap D_\beta.
\end{equation}
\item Inner and outer parametrices match well on the disk boundary:
\begin{equation}
\dot{\mathbf{O}}_\mathfrak{b}^\mathrm{V}(k)\dot{\mathbf{O}}^{(\mathrm{out})}(k)^{-1}=\mathbb{I}+\mathcal{O}(\eps),\quad k\in \partial D_\beta.
\label{eq:good-match-b-V}
\end{equation}
\item $\det(\dot{\mathbf{O}}_{\mathfrak{b}}^\mathrm{V}(k))=1$ and $\dot{\mathbf{O}}_\mathfrak{b}^\mathrm{V}(k)=\mathcal{O}(\eps^{-1/6})$ holds uniformly for $k\in D_\beta$.
\end{itemize}

If instead  $(\beta,k_\mathfrak{b})$ is a saturated region $\mathrm{S}$, then start with the conformal map defined in $D_\beta$ by the continuation from $\mathrm{B}$ of the formula $w_\mathfrak{b}^\mathrm{S}(k;t):=(2\tau(k)-2\Delta(k;t))^{2/3}>0$.  Introduce the matrix $\mathbf{H}_\mathfrak{b}^\mathrm{S}(k)$ bounded, holomorphic, and unimodular in $D_\beta$, by
\begin{equation}
\dot{\mathbf{O}}^{(\mathrm{out})}(k)=\mathbf{H}_\mathfrak{b}^\mathrm{S}(k)(-w_\mathfrak{b}^\mathrm{S}(k;t))^{\sigma_3/4}\mathbf{S}^{-1}e^{\myi\phi_\mathrm{B}\sigma_3/\eps},\quad k\in D_\beta.
\end{equation}
Then set
\begin{equation}
\dot{\mathbf{O}}_\mathfrak{b}^\mathrm{S}(k):= \mathbf{H}_\mathfrak{b}^\mathrm{S}(k)\eps^{\sigma_3/6}(-\myi\sigma_2)\mathbf{P}(\eps^{-2/3}w_\mathfrak{b}^\mathrm{S}(k;t))(\myi\sigma_3)e^{2\myi\phi_\mathrm{B}(t)\sigma_3/\eps},\quad k\in D_\beta.
\end{equation}
Assuming that for $k\in D_\beta$ the contours $\Sigma_\mathfrak{b}^{\mathrm{S}\pm}$ coincide
with segments with angles $\arg(w_\mathfrak{b}^\mathrm{S})=\mp 2\pi/3$, this parametrix satisfies these conditions.
\begin{itemize}
\item $\dot{\mathbf{O}}_{\mathfrak{b}}^\mathrm{S}(k)$ is analytic for $k\in D_\beta\setminus (\Sigma_\mathfrak{b}^{\mathrm{S}+}\cup\Sigma_\mathfrak{b}^{\mathrm{S}-}\cup\Sigma^\mathrm{B})$ (there is no jump across $\Sigma_\mathfrak{b}^{\mathrm{S}0}$).
\item $\dot{\mathbf{O}}_{\mathfrak{b}}^\mathrm{S}(k)$ satisfies the jump conditions
\begin{equation}
\dot{\mathbf{O}}_{\mathfrak{b}+}^\mathrm{S}(k)=
\dot{\mathbf{O}}_{\mathfrak{b}-}^\mathrm{S}(k)
\begin{bmatrix} e^{2(\Delta(k;t)-\tau(k))/\eps} & -e^{-2\myi\phi_\mathrm{B}(t)/\eps}\\
e^{2\myi\phi_\mathrm{B}(t)/\eps} & 0\end{bmatrix},\quad k\in\Sigma^\mathrm{B}\cap D_\beta,
\end{equation}
\begin{equation}
\dot{\mathbf{O}}_{\mathfrak{b}+}^\mathrm{S}(k)=
\dot{\mathbf{O}}_{\mathfrak{b}-}^\mathrm{S}(k)
\begin{bmatrix}1 & -e^{-2\myi\phi_\mathfrak{b}(k;t)/\eps}\\0 & 1\end{bmatrix},\quad
k\in\Sigma_\mathfrak{b}^{\mathrm{S}+}\cap D_\beta,
\end{equation}
and
\begin{equation}
\dot{\mathbf{O}}_{\mathfrak{b}+}^\mathrm{S}(k)=
\dot{\mathbf{O}}_{\mathfrak{b}-}^\mathrm{S}(k)
\begin{bmatrix}1 & 0\\e^{2\myi\phi_\mathfrak{b}(k;t)/\eps} & 1\end{bmatrix},\quad
k\in\Sigma_\mathfrak{b}^{\mathrm{S}-}\cap D_\beta.
\end{equation}
\item
Inner and outer parametrices match well on the disk boundary:
\begin{equation}
\dot{\mathbf{O}}_\mathfrak{b}^{\mathrm{S}}(k)\dot{\mathbf{O}}^{(\mathrm{out})}(k)^{-1}=
\mathbb{I}+\mathcal{O}(\eps),\quad k\in\partial D_\beta.
\label{eq:good-match-b-S}
\end{equation}
\item $\det(\dot{\mathbf{O}}_\mathfrak{b}^\mathrm{S}(k))=1$ and $\dot{\mathbf{O}}_\mathfrak{b}^\mathrm{S}(k)=\mathcal{O}(\eps^{-1/6})$ holds uniformly for $k\in D_\beta$.
\end{itemize}

Now we combine the outer and inner parametrices into a \emph{global parametrix} $\dot{\mathbf{O}}(k)$ defined as follows.  If $\mathfrak{a}'(t)<0$ and $\mathfrak{b}'(t)>0$ (case VBV), then
\begin{equation}
\dot{\mathbf{O}}(k):=\begin{cases}
\dot{\mathbf{O}}_\mathfrak{a}^\mathrm{V}(k),&\quad k\in D_\alpha\\
\dot{\mathbf{O}}_\mathfrak{b}^\mathrm{V}(k),&\quad k\in D_\beta\\
\dot{\mathbf{O}}^{(\mathrm{out})}(k),&\quad k\in\mathbb{C}\setminus (\overline{D}_\alpha\cup\overline{D}_\beta).
\end{cases}
\label{eq:global-parametrix-VBV}
\end{equation}
If $\mathfrak{a}'(t)<0$ and $\mathfrak{b}'(t)<0$ (case VBS), then
\begin{equation}
\dot{\mathbf{O}}(k):=\begin{cases}
\dot{\mathbf{O}}_\mathfrak{a}^\mathrm{V}(k),&\quad k\in D_\alpha\\
\dot{\mathbf{O}}_\mathfrak{b}^\mathrm{S}(k),&\quad k\in D_\beta\\
\dot{\mathbf{O}}^{(\mathrm{out})}(k),&\quad k\in\mathbb{C}\setminus (\overline{D}_\alpha\cup\overline{D}_\beta).
\end{cases}
\label{eq:global-parametrix-VBS}
\end{equation}
If $\mathfrak{a}'(t)>0$ and $\mathfrak{b}'(t)>0$ (case SBV), then
\begin{equation}
\dot{\mathbf{O}}(k):=\begin{cases}
\dot{\mathbf{O}}_\mathfrak{a}^\mathrm{S}(k),&\quad k\in D_\alpha\\
\dot{\mathbf{O}}_\mathfrak{b}^\mathrm{V}(k),&\quad k\in D_\beta\\
\dot{\mathbf{O}}^{(\mathrm{out})}(k),&\quad k\in\mathbb{C}\setminus (\overline{D}_\alpha\cup\overline{D}_\beta).
\end{cases}
\label{eq:global-parametrix-SBV}
\end{equation}
Finally, if $\mathfrak{a}'(t)>0$ and $\mathfrak{b}'(t)<0$ (case SBS), then
\begin{equation}
\dot{\mathbf{O}}(k):=\begin{cases}
\dot{\mathbf{O}}_\mathfrak{a}^\mathrm{S}(k),&\quad k\in D_\alpha\\
\dot{\mathbf{O}}_\mathfrak{b}^\mathrm{S}(k),&\quad k\in D_\beta\\
\dot{\mathbf{O}}^{(\mathrm{out})}(k),&\quad k\in\mathbb{C}\setminus (\overline{D}_\alpha\cup\overline{D}_\beta).
\end{cases}
\label{eq:global-parametrix-SBS}
\end{equation}
The global parametrix is intended to be a good model for $\mathbf{O}(k)$ in the whole complex plane.  To evaluate this claim, consider the error $\mathbf{E}(k)$ defined as
\begin{equation}
\mathbf{E}(k):=\mathbf{O}(k)\dot{\mathbf{O}}(k)^{-1}.
\end{equation}
Because $\mathbf{O}(k)$ satisfies the conditions of Riemann-Hilbert Problem~\ref{rhp:t-positive-open-lenses} while $\dot{\mathbf{O}}(k)$ is known explicitly, $\mathbf{E}(k)$ solves a Riemann-Hilbert problem equivalent to that for $\mathbf{O}(k)$.  This problem is the following.
\begin{rhp}
Find a $2\times 2$ matrix $\mathbf{E}(k)$ with the following properties:
\begin{itemize}
\item[]\textbf{Analyticity:}  $\mathbf{E}(k)$ is analytic for $k\in\mathbb{C}\setminus\Sigma^\mathbf{E}$, where $\Sigma^\mathbf{E}$ is the contour illustrated in Figure~\ref{fig:SmallNormContour},
\begin{figure}[h]
\includegraphics{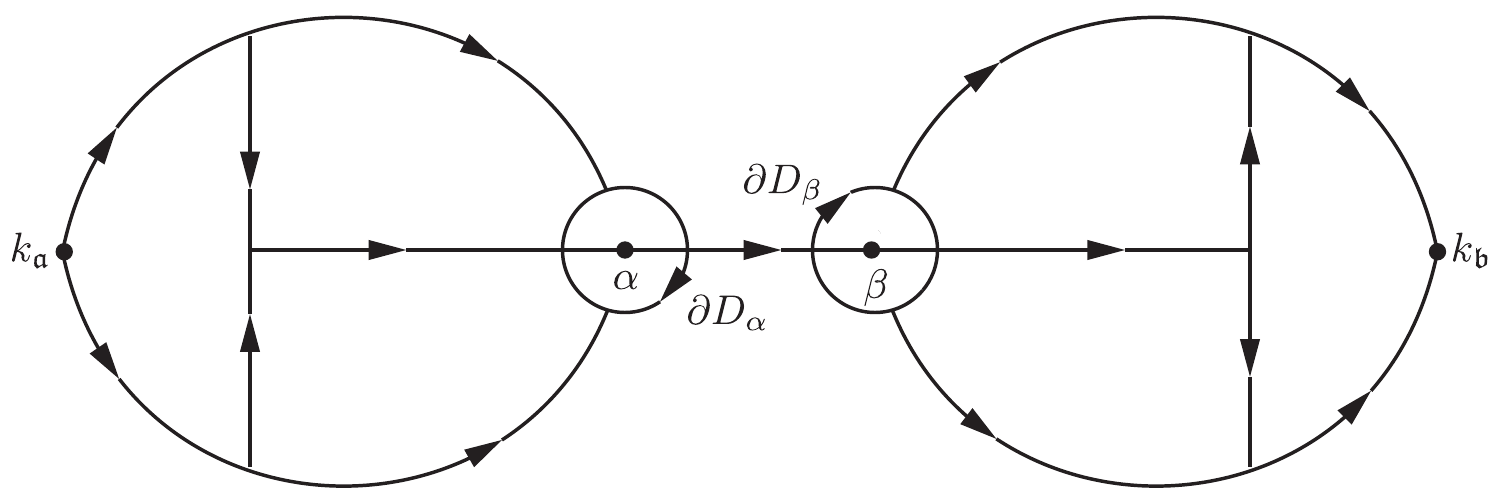}
\caption{The jump contour $\Sigma^{\mathbf{E}}$ for the error matrix $\mathbf{E}(k)$.  All contour arcs are labeled as in Figure~\ref{fig:FourContours} (different labels in the four cases) with the exception of the disk boundaries $\partial D_\alpha$ and $\partial D_\beta$ (both oriented clockwise).  Note that the sub-arcs of $\Sigma_\mathfrak{a}^{\mathrm{V}\pm}$ or $\Sigma_\mathfrak{a}^{\mathrm{S}\pm}$ within $D_\alpha$ and the sub-arcs of $\Sigma_\mathfrak{b}^{\mathrm{V}\pm}$ or $\Sigma_\mathfrak{b}^{\mathrm{S}\pm}$ within $D_\beta$ are absent
from $\Sigma^\mathbf{E}$, because the inner parametrices satisfy exactly the jump conditions of $\mathbf{O}(k)$ on these sub-arcs.}
\label{fig:SmallNormContour}
\end{figure}
and $\mathbf{E}(k)$ takes continuous boundary values $\mathbf{E}_+(k)$ and $\mathbf{E}_-(k)$ on each oriented arc of $\Sigma^\mathbf{E}$ from the left and right, respectively.
\item[]\textbf{Jump Condition:}  The boundary values on each oriented arc of $\Sigma^\mathbf{E}$ are related by $\mathbf{E}_+(k)=\mathbf{E}_-(k)\mathbf{J}^\mathbf{E}(k)$ (see below for a complete explicit characterization of  the jump matrix $\mathbf{J}^\mathbf{E}(k)$).
\item[]\textbf{Normalization:}  $\mathbf{E}(k)\to\mathbb{I}$ as $k\to\infty$.
\end{itemize}
\label{rhp:small-norm}
\end{rhp}
On all arcs of $\Sigma^\mathbf{E}$ with the exception of (i) the disk boundaries $\partial D_\alpha$ and $\partial D_\beta$, (ii) the real arcs of $\Sigma_\mathfrak{a}^{\mathrm{V}0}$ or $\Sigma_\mathfrak{a}^{\mathrm{S}0}$ within $D_\alpha$, (iii) the real arcs of $\Sigma_\mathfrak{b}^{\mathrm{V}0}$ or $\Sigma_\mathfrak{b}^{\mathrm{S}0}$ within $D_\beta$, and (iv) the real arc $\Sigma^\mathrm{B}$ (including parts inside and outside the disks), the jump matrix $\mathbf{J}^\mathbf{E}(k)$ is given by 
\begin{equation}
\mathbf{J}^\mathbf{E}(k):=\dot{\mathbf{O}}^{(\mathrm{out})}(k)\mathbf{J}(k)\dot{\mathbf{O}}^{(\mathrm{out})}(k)^{-1}.
\end{equation}
Here $\mathbf{J}(k)$ is the jump matrix for Riemann-Hilbert Problem~\ref{rhp:t-positive-open-lenses} characterizing $\mathbf{O}(k)$, and it has already been shown that $\mathbf{J}(k)-\mathbb{I}$ is uniformly $\mathcal{O}((\log(\eps^{-1}))^{-1/2})$ on these arcs of $\Sigma^\mathbf{E}$.  Since $\dot{\mathbf{O}}^{(\mathrm{out})}(k)$ is unimodular and bounded independently of $\eps$ away from $k=\alpha$ and $k=\beta$, it follows that also $\mathbf{J}^\mathbf{E}(k)-\mathbb{I}=\mathcal{O}((\log(\eps^{-1}))^{-1/2})$ holds uniformly on these arcs.

The discontinuity of $\mathbf{E}(k)$ across the boundary of the two disks is caused by the mismatch of the inner and outer parametrices, since $\mathbf{O}(k)$ has no jump. Letting $\dot{\mathbf{O}}_\mathfrak{a}(k)$ denote the inner parametrix installed in $D_\alpha$ and 
$\dot{\mathbf{O}}_\mathfrak{b}(k)$ denote the inner parametrix installed in $D_\beta$, a calculation shows that
\begin{equation}
\mathbf{J}^{\mathbf{E}}(k)=\dot{\mathbf{O}}_{\mathfrak{a},\mathfrak{b}}(k)\dot{\mathbf{O}}^{(\mathrm{out})}(k)^{-1},\quad k\in\partial D_{\alpha,\beta}.
\end{equation}
But using \eqref{eq:good-match-a-V} or \eqref{eq:good-match-a-S} for $k\in\partial D_\alpha$, and using \eqref{eq:good-match-b-V} or \eqref{eq:good-match-b-S} for $k\in D_\beta$, we conclude that $\mathbf{J}^{\mathbf{E}}(k)-\mathbb{I}=\mathcal{O}(\eps)$ holds uniformly on both $\partial D_\alpha$ and $\partial D_\beta$.

If $k$ lies on the real segment of $\Sigma_\mathfrak{a}^{\mathrm{V}0}$ or $\Sigma_\mathfrak{a}^{\mathrm{S}0}$ within $D_\alpha$, then
\begin{equation}
\mathbf{J}^\mathbf{E}(k)=\dot{\mathbf{O}}_\mathfrak{a}(k)\mathbf{J}(k)\dot{\mathbf{O}}_\mathfrak{a}(k)^{-1},\quad k\in (\Sigma_\mathfrak{a}^{\mathrm{V}0}\cup\Sigma_\mathfrak{a}^{\mathrm{S}0})\cap D_\alpha,
\end{equation}
and similarly
\begin{equation}
\mathbf{J}^\mathbf{E}(k)=\dot{\mathbf{O}}_\mathfrak{b}(k)\mathbf{J}(k)\dot{\mathbf{O}}_\mathfrak{b}(k)^{-1},\quad k\in (\Sigma_\mathfrak{b}^{\mathrm{V}0}\cup\Sigma_\mathfrak{b}^{\mathrm{S}0})\cap D_\beta,
\end{equation}
since the inner parametrix has no jump.  In all four cases we have shown that $\mathbf{J}(k)-\mathbb{I}$ is exponentially small in the limit $\eps\downarrow 0$, so even though each conjugating factor amplifies this by $\eps^{-1/6}$, we still have decay beyond all orders in $\eps$.

If $k$ lies on the band $\Sigma^\mathrm{B}$ within either disk, then both $\mathbf{O}(k)$ and the (inner) parametrix $\dot{\mathbf{O}}(k)$ are discontinuous.  If the band edge abuts a void, a calculation shows that
\begin{equation}
\mathbf{J}^\mathbf{E}(k)=\dot{\mathbf{O}}_{\mathfrak{a},\mathfrak{b}-}^\mathrm{V}(k)
\begin{bmatrix}Y^\eps(k)+e^{-2\tau(k)/\eps} & e^{-2\myi\phi_\mathrm{B}(t)/\eps}e^{2(\Delta(k;t)-\tau(k))/\eps}\\(Y^\eps(k)-1)e^{2\myi\phi_\mathrm{B}(t)/\eps} e^{-2\Delta(t;k)/\eps} & Y^\eps(k)\end{bmatrix}
\dot{\mathbf{O}}_{\mathfrak{a},\mathfrak{b}-}^\mathrm{V}(k)^{-1}.
\end{equation}
The central factor is an exponentially small perturbation of $\mathbb{I}$ because $Y^\eps(k)-1$ is exponentially small and both $\tau(k)$ and $\tau(k)-\Delta(k;t)$ are strictly positive.  The conjugating factors amplify this by $\eps^{-1/3}$, but this remains beyond all orders small.  Similarly, if the band edge abuts a saturated region, then
\begin{equation}
\mathbf{J}^\mathbf{E}(k)=\dot{\mathbf{O}}_{\mathfrak{a},\mathfrak{b}-}^\mathrm{S}(k)
\begin{bmatrix}Y^\eps(k) & (1-Y^\eps(k))e^{-2\myi\phi_\mathrm{B}(t)/\eps}e^{2(\Delta(k;t)-\tau(k))/\eps}\\
-e^{2\myi\phi_\mathrm{B}(t)/\eps}e^{-2\Delta(k;t)/\eps} & Y^\eps(k)+e^{-2\tau(k)/\eps}\end{bmatrix}
\dot{\mathbf{O}}_{\mathfrak{a},\mathfrak{b}-}^\mathrm{S}(k)^{-1},
\end{equation}
and again $\mathbf{J}^\mathbf{E}-\mathbb{I}$ is small beyond all orders despite the conjugating factors algebraically large size.

Finally, consider $k\in \Sigma^\mathrm{B}$ outside both disks.  Again both $\mathbf{O}(k)$ and the 
(outer) parametrix $\dot{\mathbf{O}}(k)$ have jump discontinuities across this segment, and a calculation shows that
\begin{equation}
\mathbf{J}^\mathbf{E}(k)=\dot{\mathbf{O}}^{(\mathrm{out})}_-(k)
\begin{bmatrix} Y^\eps(k) & -e^{-2\myi\phi_\mathrm{B}(t)/\eps}e^{-2\Delta(k;t)/\eps}\\
e^{2\myi\phi_\mathrm{B}(t)/\eps}e^{2(\Delta(k;t)-\tau(k))/\eps} & Y^\eps(k)\end{bmatrix}
\dot{\mathbf{O}}^{(\mathrm{out})}_-(k)^{-1},
\end{equation}
which is exponentially close to $\mathbb{I}$ because the outer parametrix and its inverse are uniformly bounded away from $k=\alpha$ and $k=\beta$, while $Y^\eps(k)-1$ is exponentially small and 
$\Delta(k;t)$ and $\tau(k)-\Delta(k;t)$ are both strictly positive as $k$ lies on a compact subset of the interior of the band $\mathrm{B}$.

These considerations prove that $\mathbf{J}^\mathbf{E}(k)-\mathbb{I}$ is uniformly $\mathcal{O}((\log(\eps^{-1}))^{-1/2})$ on the $\eps$-independent contour $\Sigma^\mathbf{E}$.  Therefore,
in the $\eps\downarrow 0$ limit, Riemann-Hilbert Problem~\ref{rhp:small-norm} is a small-norm Riemann-Hilbert problem in the $L^2(\Sigma^\mathbf{E})$ sense.  This implies that $\mathbf{E}(k)$ exists for $\eps>0$ sufficiently small and has a convergent (because $\Sigma^\mathbf{E}$ is bounded) Laurent expansion for sufficiently large $|k|$ of the form 
\begin{equation}
\mathbf{E}(k)=\mathbb{I}+k^{-1}\mathbf{E}_1 + k^{-2}\mathbf{E}_2+\mathcal{O}(k^{-3}),\quad k\to\infty, 
\label{eq:E-series}
\end{equation}
with the first two moments satisfying
\begin{equation}
\mathbf{E}_1=\mathcal{O}((\log(\eps^{-1}))^{-1/2})\quad\text{and}\quad\mathbf{E}_2=\mathcal{O}((\log(\eps^{-1}))^{-1/2}),\quad \eps\downarrow 0.  
\label{eq:E-moment-bound}
\end{equation}

At last, we obtain the exact formula for $\tilde{\mathbf{M}}(k)$ valid for sufficiently large $|k|$:
\begin{equation}
\tilde{\mathbf{M}}(k)=\mathbf{N}(k)e^{\myi g(k)\sigma_3/\eps}=\mathbf{O}(k)e^{\myi g(k)\sigma_3/\eps}
=\mathbf{E}(k)\dot{\mathbf{O}}^{(\mathrm{out})}(k)e^{\myi g(k)\sigma_3/\eps}.
\end{equation}
Taking into account the Laurent expansion $g(k;t)=k^{-1}g_1(t)+k^{-2}g_2(t)+\mathcal{O}(k^{-3})$ as $k\to\infty$, using \eqref{eq:dot-O-out-expansion} with \eqref{eq:dot-O-moments}, and using \eqref{eq:E-series} with \eqref{eq:E-moment-bound},
it therefore follows from \eqref{eq:tilde-q-qx-moments} that in the limit $\eps\downarrow 0$,
\begin{equation}
\begin{split}
\tilde{q}^\eps(0,t)&=\frac{\beta-\alpha}{2}e^{-2\myi\phi_\mathrm{B}(t)/\eps} +\mathcal{O}((\log(\eps^{-1}))^{-1/2})\\
\eps\tilde{q}^\eps_x(0,t)&=-\frac{1}{2}\myi(\beta^2-\alpha^2)e^{-2\myi\phi_\mathrm{B}(t)/\eps}+
\mathcal{O}((\log(\eps^{-1}))^{-1/2}),
\end{split}
\end{equation}
where we have used \eqref{eq:dot-O-out-expansion} and \eqref{eq:E-moment-bound}.  Since $\alpha=\mathfrak{a}(t)$ and $\beta=\mathfrak{b}(t)$, using \eqref{eq:a-b-define} and Lemma~\ref{lemma-phi-B-S} we obtain
\begin{equation}
\tilde{q}^\eps(0,t)=H(t)e^{\myi S(t)/\eps}+\mathcal{O}((\log(\eps^{-1}))^{-1/2})\quad\text{and}\quad
\eps\tilde{q}^\eps_x(0,t)=\myi U(t)H(t)e^{\myi S(t)/\eps}+\mathcal{O}((\log(\eps^{-1}))^{-1/2}).
\end{equation}
The error terms are uniform on compact subintervals of $t\in (0,+\infty)\setminus\{t_\mathfrak{a},t_\mathfrak{b}\}$.   Indeed, the $L^\infty(\Sigma^\mathbf{E})$ estimate on $\mathbf{J}^\mathbf{E}(k)-\mathbb{I}$ fails only as $t\to 0$ or $t\to\infty$ (in which case the band $B$ shrinks to a point), as $t\to t_\mathfrak{a}$ (in which case the void or saturated region on the left of the band shrinks to a point), and as $t\to t_\mathfrak{b}$ (in which case the void or saturated region on the right of the band shrinks to a point).  Furthermore, by a slight modification of the preceding arguments in which the disks $D_\alpha$ and $D_\beta$ are allowed to be slightly off-center from the points $\alpha$ and $\beta$, it can be shown that for $t$ in a given compact subinterval of $(0,+\infty)\setminus\{t_\mathfrak{a},t_\mathfrak{b}\}$, Riemann-Hilbert Problem~\ref{rhp:small-norm} may be formulated with a finite number of different contours $\Sigma^\mathbf{E}$.  Therefore the norm of the Cauchy projection operator $\mathcal{C}_-^{\Sigma^\mathbf{E}}:L^2(\Sigma^\mathbf{E})\to L^2(\Sigma^\mathbf{E})$ is uniformly bounded (independent of $\eps$, of course) as a maximum over a finite number of values (see \cite[Proposition 3 of \S3.6, and \S4.6]{BuckinghamM13} for further information and more details about this procedure).  This completes the proof of Theorem~\ref{theorem:boundary-condition-recover}.

\subsubsection{Proof of Corollary~\ref{corollary:plane-wave}}
\label{sec:plane-wave}
The main idea of the proof is to observe that 
the construction of $g(k;t)$ for $x=0$ and $t>0$ presented in \S\ref{sec:g-function} may be subjected to continuation for small $x>0$.  Consider the prospect of continuation of the solution $(\alpha,\beta)=(\alpha(0,t),\beta(0,t))=(\mathfrak{a}(t),\mathfrak{b}(t))$ of the equations  \eqref{eq:moments-genus-zero}
for $x=0$ to nonzero $x$ by means of the implicit function theorem.  This requires calculating the Jacobian of the system \eqref{eq:moments-genus-zero}.  We first establish the following.
\begin{lemma}
The functions $m_j(\alpha,\beta;x,t)$ for $j=1,2$ defined by \eqref{eq:moments-genus-zero}
satisfy the identities
\begin{equation}
\frac{\partial m_2}{\partial \alpha}=\frac{1}{2}m_1 +\alpha\frac{\partial m_1}{\partial\alpha}\quad
\text{and}\quad
\frac{\partial m_2}{\partial\beta}=\frac{1}{2}m_1 +\beta\frac{\partial m_1}{\partial\beta}.
\end{equation}
\label{lemma-moment-derivatives}
\end{lemma}
\begin{proof}
Firstly, one establishes the related identities
\begin{equation}
\frac{\partial I_2}{\partial\alpha}=\frac{1}{2}I_1+\alpha\frac{\partial I_1}{\partial\alpha}\quad\text{and}\quad
\frac{\partial I_2}{\partial\beta}=\frac{1}{2}I_1+\beta\frac{\partial I_1}{\partial\beta}
\end{equation}
by replacing the integrals in \eqref{eq:Ip-formula} in a standard way by contour integrals over contours locally independent of $\alpha$ and $\beta$ that are bounded away from these points,
and then differentiating under the integral sign.  Then one eliminates $I_1(\alpha,\beta)$ and $I_2(\alpha,\beta)$ in favor of $m_1(\alpha,\beta;x,t)$ and $m_2(\alpha,\beta;x,t)$ using \eqref{eq:moments-genus-zero} to finish the proof.
\end{proof}
Whenever $(\alpha,\beta)$ satisfy the equations \eqref{eq:moments-genus-zero}, as is the case for $x=0$ and $t>0$ with $(\alpha,\beta)=(\mathfrak{a}(t),\mathfrak{b}(t))$ by Lemma~\ref{lemma:endpoints-x-zero}, the Jacobian is easily calculated with the help of Lemma~\ref{lemma-moment-derivatives}:
\begin{equation}
\frac{\partial m_1}{\partial\alpha}\frac{\partial m_2}{\partial\beta}-
\frac{\partial m_1}{\partial\beta}\frac{\partial m_2}{\partial\alpha}=(\beta-\alpha)\frac{\partial m_1}{\partial \alpha}\frac{\partial m_1}{\partial\beta}.
\end{equation}
When $x=0$ and $t>0$, we have $\beta-\alpha=2H(t)>0$.  It can be shown that in this situation also
\begin{equation}
\frac{\partial m_1}{\partial \alpha}(\mathfrak{a}(t),\mathfrak{b}(t);0,t)\neq 0\quad\text{and}\quad
\frac{\partial m_1}{\partial\beta}(\mathfrak{a}(t),\mathfrak{b}(t);0,t)\neq 0,\quad t\in (0,+\infty)\setminus\{t_\mathfrak{a},t_\mathfrak{b}\}.
\end{equation}
Indeed, it turns out that $m_{1\alpha}=0$ if and only if the function $\Delta'(k)$ or $\tau'(k)-\Delta'(k)$ vanishes as $k\downarrow \alpha$  to higher order than $(k-\alpha)^{1/2}$ depending on whether $(k_\mathfrak{a},\alpha)$ is a void $V$ or a saturated region $S$.  Similarly, $m_{1\beta}=0$ detects higher-order vanishing of either $\Delta'(k)$ or $\tau'(k)-\Delta'(k)$ at $k=\beta$.  It is easy to see from the explicit construction of $g$ given in \S\ref{sec:g-function} that as long as $t\neq t_\mathfrak{a}$ or $t_\mathfrak{b}$, then $m_{1\alpha}\neq 0$ and $m_{1\beta}\neq 0$.
Therefore, the Jacobian is nonzero, and so the following is true.
\begin{lemma}
Suppose $t_0>0$ and $t\neq t_\mathfrak{a}$, $t\neq t_\mathfrak{b}$, so that the complex phase function $g(k;t_0)$ is well-defined in a particular configuration, VBV, VBS, SBV, or SBS.  Then there is a neighborhood of the point $(0,t_0)$ in the $(x,t)$-plane in which there is a unique solution $(\alpha(x,t),\beta(x,t))$
of the equations \eqref{eq:moments-genus-zero} formulated for the same configuration that satisfies the boundary condition $(\alpha(0,t),\beta(0,t))=(\mathfrak{a}(t),\mathfrak{b}(t))$.  The solution is differentiable and the partial derivatives satisfy the identities \eqref{eq:Whitham-1}.
\label{lemma-endpoints-x-small}
\end{lemma}
\begin{proof}
It only remains to establish the identities \eqref{eq:Whitham-1}, but this is done simply by implicit differentiation of the equations $m_1(\alpha(x,t),\beta(x,t);x,t)=m_2(\alpha(x,t),\beta(x,t);x,t)=0$
with respect to $x$ and $t$, using Lemma~\ref{lemma-moment-derivatives} with $m_1=m_2=0$.
\end{proof}
This result allows us to define a candidate complex phase function $g=g(k;x,t)$ given $\alpha$ and $\beta$ as described in \S\ref{sec:g-function}.  Moreover, it is easy to see that the inequalities, guaranteed for $x=0$ in the particular configuration of voids and saturated regions valid for $t=t_0$
according to Lemma~\ref{lemma:inequalities-x-zero}, persist for $g(k;x,t)$ in the same configuration for small $x\neq 0$.  The existence of an appropriate complex phase function $g$ was the only essential ingredient in the proof of Theorem~\ref{theorem:boundary-condition-recover}.  Therefore,
the same result holds, with $\alpha(x,t)$ and $\beta(x,t)$ taking the place of $\mathfrak{a}(t)$ and $\mathfrak{b}(t)$ respectively, and the proof of Corollary~\ref{corollary:plane-wave} is complete.

\section{Concluding Remarks}
\label{sec:concluding-remarks}
Here we collect together a few additional remarks, some indicating natural generalizations of our results, and some pointing toward directions for future research.
\begin{remark}
The methodology of the complex phase function as described in \S\ref{sec:g-function} is capable of detecting and describing asymptotic behavior of $\tilde{q}^\eps(x,t)$ that neither corresponds to  the vacuum domain nor the plane-wave domain.  
A typical scenario would be that as one tunes $(x,t)$ through the plane-wave domain by continuing the complex phase function $g$ along some path, at some point $(x_0,t_0)$ one of the inequalities associated with the complex phase function $g=g(k;x,t)$ fails for some $k\in (k_\mathfrak{a},k_\mathfrak{b})$ (a less generic scenario involves the Jacobian of the system \eqref{eq:moments-genus-zero} vanishing).  Further continuation becomes impossible without resorting to the device of increasing $N$, where we recall that $N+1$ is the number of bands in $(k_\mathfrak{a},k_\mathfrak{b})$.  Thus one is witnessing the birth of a new band, void, or saturated region from the point $k$ as one continues beyond $(x_0,t_0)$.  Once $N>0$, instead of an asymptotic formula like \eqref{eq:plane-wave}, one arrives at an asymptotic formula written in terms of hyperelliptic functions of genus $N$, a formula that exhibits rapid variations of the amplitude on space and time scales proportional to $\eps$.  Thus, if $N>0$,
$\tilde{q}^\eps(x,t)$ no longer resembles a slowly-modulated plane wave, but rather has a more complicated microstructure.  The point $(x_0,t_0)$ therefore lies along the common boundary between domains corresponding to $N=0$ and (in the next simplest case) $N=1$.  In general, one expects the quarter plane to be tiled with $\eps$-independent domains corresponding to various values of $N\ge 0$, with the only general statement being that the positive $t$-axis abuts the (plane-wave) domain with $N=0$ and the positive $x$-axis abuts the vacuum domain (which is not associated with any value of $N$).  These phase transitions may in principal be computed, with the details depending on the particular boundary data given at $x=0$ and $t>0$.  We leave such considerations for future studies.  Note, however, that under the assumptions in force in this paper, it is impossible for a dispersive shock wave region with $N>0$ that is generated from a gradient catastrophe in the solution of the dispersionless system \eqref{eq:dispersion-less-NLS} at some time $(x,t)=(x_c,t_c)$ to  reach the boundary $x=0$ in finite time, since such behavior would be inconsistent with Theorem~\ref{theorem:boundary-condition-recover}.
\label{remark-multiphase}
\myendrmk
\end{remark}

\begin{remark}
\label{remark-imaginary-amplify}
We wish to briefly explain the key obstruction in our opinion that arises in the semiclassical analysis of  Riemann-Hilbert Problem~\ref{rhp-original} posed in \S\ref{sec:Introduction} relative to a cross contour $\Sigma=\mathbb{R}\cup \myi\mathbb{R}$, and hence our actual motivation in introducing the modified approach described in \S\ref{sec:Formulation}.

A calculation starting from \eqref{eq:gprime} in the VBV configuration shows that $g'(k;t)\to 0$ as $t\downarrow 0$ (for $x=0$), and hence from \eqref{eq:g-gprime-integrate} we obtain $g(k;0)=0$ (this is consistent with our approach for $x>0$ and $t=0$, which basically assumed $g(k)=0$).  Therefore for $x=0$ and $t>0$ we have
\begin{equation}
g(t;k)=\int_0^tg_t(k;s)\,ds.
\end{equation}
Using \eqref{eq:gt} and assuming that $k$ is purely imaginary gives
\begin{equation}
\Im\{g(t;k)\}=\int_0^t\Im\{g_t(k;s)\}\,ds = \int_0^t\Im\{(U(s)-2k)r(k;s)\}\,ds,\quad k\in \myi\mathbb{R}.
\label{eq:Img}
\end{equation}
Note that
\begin{equation}
(U(t)-2k)r(k;t)=-2k^2+\frac{1}{2}U(t)^2+H(t)^2-H(t)^2U(t)k^{-1}+\mathcal{O}(k^{-2}),\quad k\to\infty,
\end{equation}
which together with \eqref{eq:Img} implies that $\Im\{g(t;k)\}>0$ holds whenever $k$ is positive imaginary and of sufficiently large magnitude, because $H(t)^2U(t)>0$ for $t>0$ by Assumption~\ref{assumption:data}.  

Consider now the effect of introducing $g$ not into Riemann-Hilbert Problem~\ref{rhp-M-tilde} for $\tilde{\mathbf{M}}(k)$ having (a subset of ) $\mathbb{R}$ as the jump contour, but rather into Riemann-Hilbert Problem~\ref{rhp-original} for $\mathbf{M}(k)$ having $\Sigma=\mathbb{R}\cup \myi\mathbb{R}$ as the jump contour.  In the latter problem one has a lower-triangular jump matrix on the positive imaginary axis given by \eqref{eq:exact-jump-positive-imaginary}.  Thus, if one introduces the complex phase function $g(k;t)$ by means of a formula analogous to \eqref{eq:N-from-g}, the resulting jump condition for $\mathbf{N}(k)$ on the positive imaginary axis (oriented toward the origin) would read
\begin{equation}
\mathbf{N}_+(k)=\mathbf{N}_-(k)\begin{bmatrix}1 & 0\\-\Gamma_0(k)e^{2\myi\theta(k;0,t)/\eps}e^{-2\myi g(k;t)/\eps} & 1\end{bmatrix},\quad \arg(k)=\frac{\pi}{2}.
\end{equation}
The factor $e^{2\myi\theta(k;0,t)/\eps}$ is purely oscillatory for imaginary $k$, and by Proposition~\ref{prop:noturningpoints} we have an upper bound for $|\Gamma_0(k)|$ that decays to zero with $\eps$.  However, we now see that the factor $e^{-2\myi g(k;t)/\eps}$ is exponentially large in the limit $\eps\downarrow 0$ for sufficiently large $|k|$ with $\arg(k)=\pi/2$.  Therefore, the same complex phase function $g$ that works so well to control the modified problem for $\tilde{\mathbf{M}}(k)$ produces uncontrollable errors if used to study $\mathbf{M}(k)$ itself.  This difficulty originates in the jump discontinuities along the imaginary axis, which appear to be negligible before $g$ is introduced but that seem impossible to either neglect afterwards or include in a parametrix without explicit beyond-all-orders information about $\Gamma_0(k)$.  
\myendrmk
\end{remark}

\begin{remark}
As Corollary~\ref{corollary:plane-wave} is formulated, the points $t=t_\mathfrak{a}$ and $t=t_\mathfrak{b}$ appear to present an obstruction to continuation of the plane-wave approximation
to positive $x$ from the boundary.  While the details are not easy to explain, the fact is that the complex phase function $g$ can indeed be continued away from the boundary near such points, which become curves $t=t_\mathfrak{a}(x)$ and $t=t_\mathfrak{b}(x)$ for $x>0$ along which one of the band endpoints is fixed to an endpoint of the interval $[k_\mathfrak{a},k_\mathfrak{b}]$ (see, for example, \cite[\S4.3.2]{BuckinghamMMemoirs}).  There still remains, however, a technical issue in that it appears that a new type of inner parametrix is required near the corresponding  point in the $k$-plane, and to our knowledge this has not been worked out.  Although we are therefore prevented from proving convergence to a plane wave for $(x,t)$ near the curves $t=t_\mathfrak{a}(x)$ and $t=t_\mathfrak{b}(x)$, one can check that the plane wave formulae that are indeed valid as given in Corollary~\ref{corollary:plane-wave} on either side of these curves actually match on the curves themselves.  Thus one does not expect any new leading-order asymptotic behavior for $\tilde{q}^\eps(x,t)$ along these curves, which nonetheless present an obstruction to rigorous analysis.  
\myendrmk
\end{remark}

\begin{remark}
The reader will observe that the very slow rate of decay of the error terms in our results (proportional to $(\log(\eps^{-1}))^{-1/2}$) is in all cases due to the truncation introduced in \S\ref{sec:Formulation} of the jump matrix on $\mathbb{R}$ to one whose difference from $\mathbb{I}$ is compactly supported in the interval $[k_\mathfrak{a},k_\mathfrak{b}]$.  Indeed this truncation leads to dominant contributions to the error generated from small neighborhoods of the endpoints of this support interval.  
This suggests that it may be possible to further modify the formula for $\tilde{\Gamma}(k)$ near these two points with the aim of reducing the magnitude of these local contributions to the error without having any other significant effect.  
\myendrmk
\end{remark}

\begin{remark}
\label{remark:nonzero-IC}
In our view, the incorporation of nonzero initial conditions together with a non-linearizable boundary condition presents substantial new difficulties.  Physically, one expects a nontrivial interaction between the initial and boundary data, and from the ``hyperbolic'' point of view of the dispersionless nonlinear Schr\"odinger system \eqref{eq:dispersion-less-NLS} one may expect competing influences at a given point $(x,t)$ from boundary and initial data propagating along characteristics.  Mathematically, the initial and boundary conditions get mixed up in the complicated jump matrix on the negative real axis as indicated in \eqref{eq:exact-jump-negative}.  The fact that at points $k<0$ where both $\gamma(k)$ and $\Gamma(k)$ are nonzero the elements of the jump matrix are differences of oscillatory exponentials in the semiclassical approximation of $\gamma$ and $\Gamma$ makes analysis by steepest descent methods quite challenging, and we hope to consider such issues in the future.
\myendrmk
\end{remark}

\begin{remark}
We may consider several ways to generalize Assumption~\ref{assumption:data} while still maintaining the form \eqref{eq:DirichletDataForm} of the Dirichlet boundary data.  
\begin{itemize}
\item One could drop the condition $U(t)>2H(t)$ in favor of the weaker condition $U(t)>H(t)$.  In this situation the boundary $x=0$ is still a spacelike curve for the hyperbolic dispersionless defocusing nonlinear Schr\"odinger system \eqref{eq:dispersion-less-NLS}, so one has a local solution of this approximate system near the boundary for all $t>0$.  On the other hand, it seems that it is not possible to recover the boundary data from the Riemann-Hilbert problem in this case as pointed out in Remark~\ref{remark:kb-negative}.  Is this a merely technical issue, or does this imply that a dispersive shock wave generated for positive $x$ reaches the boundary at some later time $t$, ruining the local plane-wave microstructure?
\item One could go further and drop the condition $U(t)>H(t)$.  If some bicharacteristics of the hyperbolic system \eqref{eq:dispersion-less-NLS} point outside of the domain at the boundary $x=0$, is it still possible for the approximate Dirichlet-to-Neumann map given in Definition~\ref{definition:D-to-N} to be valid?  Perhaps should one expect a kind of boundary layer to form near $x=0$?
\item Finally, one could even drop the inequality \eqref{eq:Sprimebound}.  In this case the phase gradient $u$ is apparently imaginary at the boundary.  What does this mean?
\end{itemize}
It would also be very interesting to consider the nonlinear Schr\"odinger equation in the focusing case, in which case the analogue of the dispersionless system \eqref{eq:dispersion-less-NLS} is a quasilinear system of elliptic type, so all intuition involving propagation along characteristics is lost.
\myendrmk
\end{remark}

\section{Acknowledgements}
P.~D.~Miller was partially supported by the National Science Foundation under grant number DMS-1206131 and by a fellowship from the Simons Foundation, grant number 267106.  Part of this work was done during two visits Miller made to Fudan University in Shanghai, China, and both authors are grateful for the support Fudan University offered.  The authors' collaboration began while Z.~Qin visited the University of Michigan for an extended period under the support of the China Scholarship Council. Z.~Qin was also supported by a grant from the Key Laboratory of Mathematics for Nonlinear Science, Fudan University.

\appendix

\section{Proof of Proposition~\ref{prop:noturningpoints}}
\label{sec:Appendix-noturningpoints}
We begin with a key lemma.
\begin{lemma}[Bound on $\Re\{\lambda\}$]
\label{lem:eigenvaluebranch}
Suppose that Assumption~\ref{assumption:data} holds and that either $\Im\{k\}>0$ and $\Re\{k\}\le 0$, or that $k<k_\mathfrak{a}$, or that $k_\mathfrak{b}<k\le 0$.  Then there is an eigenvalue $\lambda$ of the matrix $\mathbf{B}(t;k)$ defined by \eqref{eq:B-matrix} that depends smoothly on $t>0$ 
and satisfies $\Re\{\lambda\}\le 0$, $\forall t>0$. 
\end{lemma}
\begin{proof}
The fact that $\lambda$ can be chosen to be a smooth 
function of $t$ follows from Lemma~\ref{lem:turningpoints} since there are no turning points and $\lambda^2$ is a smooth 
function of $t>0$.  If $k$ is real and $k<k_\mathfrak{a}$ or $k>k_\mathfrak{b}$, then $\lambda^2$ is real and strictly negative for all $t>0$, so both eigenvalue functions $\lambda=\lambda(t)$ will be purely imaginary for $t>0$.  On the other hand, if $\Im\{k\}>0$ but $\Re\{k\}\le 0$, then the lemma can be proven by showing that $\lambda^2$ is \emph{not} negative real for any $t>0$.  
Let $k_\mathrm{r}:=\Re\{k\}$ and $k_\mathrm{i}:=\Im\{k\}$, and suppose that $\lambda^2$ is real
for some $\kr\le 0$, some $\ki>0$, and some $t>0$.  Using \eqref{eq:lambdasquared}, we therefore have
\begin{equation}
\Im\{\lambda^2\} = 4k_\mathrm{i}\left[4k_\mathrm{r}k_\mathrm{i}^2 +[2k_\mathrm{r}-U(t)](H(t)^2-k_\mathrm{r}[2k_\mathrm{r}+U(t)])\right]=0.
\label{eq:Im-lambda-squared}
\end{equation}
If $\kr=0$, then as also $\ki>0$, \eqref{eq:Im-lambda-squared} implies that $U(t)H(t)^2=0$, which is a contradiction with $t>0$ by Assumption~\ref{assumption:data}.
On the other hand, if $\kr<0$, then we may solve \eqref{eq:Im-lambda-squared} for $\ki^2$ and hence from \eqref{eq:lambdasquared} we obtain
\begin{equation}
\begin{split}
\Re\{\lambda^2\}&=\left(\tfrac{1}{2}[2k_\mathrm{r}-U(t)]^2-2k_\mathrm{i}^2\right)\left(2H(t)^2-\tfrac{1}{2}[2k_\mathrm{r}+U(t)]^2+2k_\mathrm{i}^2\right)+4k_\mathrm{i}^2[2k_\mathrm{r}-U(t)][2k_\mathrm{r}+U(t)]\\
&=\frac{1}{4k_\mathrm{r}^2}[4k_\mathrm{r}^2-H(t)^2]^2[4k_\mathrm{r}^2-U(t)^2].
\end{split}
\end{equation}
But, at the same time \eqref{eq:Im-lambda-squared} implies the inequality
\begin{equation}
[2k_\mathrm{r}-U(t)](H(t)^2-k_\mathrm{r}[2k_\mathrm{r}+U(t)])> 0,
\label{eq:inequality-one}
\end{equation}
which can be written in the equivalent form (because $\kr<0$ and $U(t)>0$ by Assumption~\ref{assumption:data})
\begin{equation}
U(t)^2-4k_\mathrm{r}^2<\frac{U(t)-2k_\mathrm{r}}{k_\mathrm{r}}H(t)^2<0.
\label{eq:inequality-two}
\end{equation}
Obviously, \eqref{eq:inequality-two} implies that $\Re\{\lambda^2\}>0$.  Hence $\lambda^2<0$ is not possible for any $t>0$.
\end{proof}

We denote by $\lambda=\Lambda(t;k)$ the eigenvalue branch characterized by Lemma~\ref{lem:eigenvaluebranch}, so that $\Re\{\Lambda(t;k)\}\le 0$ for all $t>0$.  If $k\in \mathcal{Q}^\mathrm{II}_\delta\cap\mathbb{R}$, then $\Lambda(t;k)\in \myi\mathbb{R}$ is ambiguous up to a sign and we choose the sign so that $\Lambda:\mathbb{R}_+\times \mathcal{Q}^\mathrm{II}_\delta\to\mathbb{C}$ is continuous.  The matrix $\mathbf{R}_0(t;k)$ given by 
\begin{equation}
\mathbf{R}_0(t;k):=\left(\frac{1}{2}(2k-U(t))(2k+U(t))+\myi\Lambda(t;k)\right)\myi\sigma_3-(2k-U(t))H(t)\sigma_1
\end{equation}
is an eigenvector matrix for the coefficient matrix $\mathbf{B}(t;k)$ defined by \eqref{eq:B-matrix}, i.e., the identity  $\mathbf{B}(t;k)\mathbf{R}_0(t;k)=
\Lambda(t;k)\mathbf{R}_0(t;k)\sigma_3$ holds.  Since $\Lambda(t;k)\neq 0$ for $k\in\mathcal{Q}^\mathrm{II}_\delta$ and $t>0$ by Lemma~\ref{lem:turningpoints}, the eigenvalues are distinct and therefore $\mathbf{R}_0(t;k)$ is invertible as long as its columns are both nonzero.  But $\Re\{-(2k-U(t))H(t)\}=-(2\Re\{k\}-U(t))H(t)\ge 2U(t)H(t)$
as $\Re\{k\}\le 0$.  It follows from Assumption~\ref{assumption:data} that the off-diagonal entries of $\mathbf{R}_0(t;k)$ are nonzero for all $t>0$ and therefore $\mathbf{R}_0(t;k)$ has nonzero columns for all $k\in\mathcal{Q}^\mathrm{II}_\delta$ and for all $t>0$, so for such $k$ and $t$ we have
$\det(\mathbf{R}_0(t;k))\neq 0$.  

Moreover, $\det(\mathbf{R}_0(t;k))$ is bounded away from zero on $\mathbb{R}_+\times\mathcal{Q}^\mathrm{II}_\delta$.  Indeed, one can show that
$\Lambda(t;k)=-2\myi k^2 + \mathcal{O}(k)$ as $k\to\infty$ uniformly for $\pi/2\le\arg(k)\le\pi$ and $t>0$, and it follows that $\det(\mathbf{R}_0(t;k))=16k^4 + \mathcal{O}(k^3)$ also holds as $k\to\infty$ with the same uniform nature of the error term.  On the other hand if $|k|\le L$ for some (large) $L>0$, then $k$ lies in a compact subset of $\mathcal{Q}^\mathrm{II}_\delta$, so since $\det(\mathbf{R}_0(t;k))$ is a continuous nonzero function of its arguments it follows that it is uniformly bounded away from zero for all such $k$ and $t$ lying in any compact subset of $(0,\infty)$.  It remains to analyze $\det(\mathbf{R}_0(t;k))$ in the limits $t\downarrow 0$ and $t\uparrow +\infty$, in which $H(t)\to 0$ by Assumption~\ref{assumption:data}.  
Since
$\Lambda(t;k)\to -\myi\tfrac{1}{2}(4k^2-U_0^2)$ as $t\downarrow 0$ and $\Lambda(t;k)\to -\myi\tfrac{1}{2}(4k^2-U_\infty^2)$ as $t\uparrow +\infty$, we have $\mathbf{R}_0(t;k)\to\myi(4k^2-U_0^2)\sigma_3$
as $t\downarrow 0$ and $\mathbf{R}_0(t;k)\to \myi (4k^2-U_\infty^2)\sigma_3$ as $t\uparrow +\infty$.
Both limiting values are nonzero for $k\in\mathcal{Q}^\mathrm{II}_\delta$, and the convergence is uniform for $|k|\le L$, so the argument is complete.  

Let $N(t;k)$ be the continuous function defined by $N(t;k)^2=\det(\mathbf{R}_0(t;k))$ and the asymptotic condition (which selects an unambiguous branch of the square root) that $N(t;k)\to 4k^2-U_\infty^2$ as $t\to +\infty$.  By a homotopy argument taking the function $H(t)$ to zero by scaling it is easy to see that $N(0;k)=4k^2-U_0^2$.  Then $N(t;k)^{-1}$ is a uniformly bounded nonvanishing function for $t>0$ and $k\in \mathcal{Q}^\mathrm{II}_\delta$ that satisfies
$N(t;k)^{-1}=\mathcal{O}(k^{-2})$ as $k\to\infty$ in $\mathcal{Q}^\mathrm{II}_\delta$ uniformly for $t>0$.
Let an eigenvector matrix $\mathbf{R}(t;k)$ be defined for $t>0$ and $k\in\mathcal{Q}^\mathrm{II}_\delta$ by
\begin{equation}
\mathbf{R}(t;k):=\frac{1}{N(t;k)}\mathbf{R}_0(t;k).
\end{equation}
Of course $\det(\mathbf{R}(t;k))\equiv 1$.  Also, it is easy to see that
\begin{equation}
\lim_{t\downarrow 0}\mathbf{R}(t;k)=\lim_{t\uparrow +\infty}\mathbf{R}(t;k)=\myi\sigma_3.
\label{eq:Rlimits}
\end{equation}
Let the first column $\mathbf{f}(t;k)$ of the matrix $\mathbf{F}(t;k)$ satisfying \eqref{eq:Fsystem}
and \eqref{eq:Fnorm} be written in the form 
\begin{equation}
\mathbf{f}(t;k)=\mathbf{R}(t;k)\mathbf{g}(t;k)
\end{equation}
for some new unknown vector function $\mathbf{g}(t;k)$.  It is clear from \eqref{eq:Fnorm} and
\eqref{eq:Rlimits} that
\begin{equation}
\lim_{t\to +\infty}\mathbf{g}(t;k)e^{\myi (4k^2t + S(t))/(2\epsilon)}=-\myi\begin{bmatrix}1\\0\end{bmatrix}, 
\end{equation}
and from \eqref{eq:AstarF}--\eqref{eq:BstarF} and \eqref{eq:Rlimits} that $A_0(k^*)^*$ and
$B_0(k^*)^*$ can be expressed in terms of $\mathbf{g}(t;k)$ by
\begin{equation}
A_0(k^*)^*=\myi e^{\myi S(0)/(2\epsilon)}g_1(0;k)\quad\text{and}\quad
B_0(k^*)^*=-\myi e^{-\myi S(0)/(2\epsilon)}g_2(0;k).
\label{eq:AstarBstarg}
\end{equation}
Moreover, a direct calculation using \eqref{eq:Fsystem} shows that $\mathbf{g}(t;k)$ satisfies
the differential equation
\begin{equation}
\eps \frac{d\mathbf{g}}{dt}(t;k)=\Lambda(t;k)\sigma_3\mathbf{g}(t;k) - \eps \mathbf{R}(t;k)^{-1}\frac{d\mathbf{R}}{dt}(t;k)\mathbf{g}(t;k)=\Lambda(t;k)\sigma_3\mathbf{g}(t;k)-\eps
\rho(t;k)\sigma_2\mathbf{g}(t;k),
\end{equation}
where 
\begin{equation}
\rho(t;k):=\frac{1}{\myi N(t;k)^2}\left[R_{0,11}(t;k)\frac{dR_{0,21}}{dt}(t;k)-
R_{0,21}(t;k)\frac{dR_{0,11}}{dt}(t;k)\right].
\end{equation}
Assumption~\ref{assumption:data} implies that 
$H'$ and $U'$ are absolutely integrable on $\mathbb{R}_+$, and this in turn implies that $\rho(\cdot;k)\in L^1(\mathbb{R}_+)$ for each $k\in\mathcal{Q}^\mathrm{II}_\delta$.  Moreover there is a constant $C_\rho$ depending only on the functions $H$ and $U$ satisfying Assumption~\ref{assumption:data} such that
\begin{equation}
\|\rho(\cdot;k)\|_1:=\int_0^{+\infty}|\rho(t;k)|\,dt \le\frac{C_\rho}{1+|k|},\quad k\in\mathcal{Q}^\mathrm{II}_\delta.
\label{eq:rL1}
\end{equation}
Making the rescaling
\begin{equation}
\mathbf{g}(t;k)=-\myi e^{-\myi S(0)/(2\eps)}e^{\myi M(k)/(2\eps)}\mathbf{h}(t;k)\exp\left(\frac{1}{\eps}\int_0^t\Lambda(s;k)\,ds\right),
\end{equation}
where
\begin{equation}
M(k):=\int_0^\infty\left[2i\Lambda(s;k)-4k^2+U(s)^2+2H(s)^2\right]\,ds
\end{equation}
results in the system
\begin{equation}
\eps\frac{d\mathbf{h}}{dt}(t;k)=\begin{bmatrix}0 & 0\\0 & -2\Lambda(t;k)\end{bmatrix}
\mathbf{h}(t;k) -\eps \rho(t;k)\sigma_2\mathbf{h}(t;k),
\label{eq:hsystem}
\end{equation}
and the boundary condition
\begin{equation}
\lim_{t\to +\infty}\mathbf{h}(t;k)=\begin{bmatrix}1\\0\end{bmatrix}.
\label{eq:hnorm}
\end{equation}
Note that $M(k)$ is well-defined due to the asymptotic behavior of $\Lambda(t;k)$ for large positive $t$ and that the condition \eqref{eq:hnorm} is a consequence of \eqref{eq:SprimeUH}.
Also, from \eqref{eq:AstarBstarg} we have that
\begin{equation}
A_0(k^*)^*=e^{\myi M(k)/(2\eps)} h_1(0;k)\quad\text{and}\quad
B_0(k^*)^*=-e^{-\myi S(0)/\eps}e^{\myi M(k)/(2\eps)}h_2(0;k).
\label{eq:AstarBstarh}
\end{equation}

Now we solve \eqref{eq:hsystem} and \eqref{eq:hnorm}.  Writing $h_1(t;k)=1+y(t;k)$ and $h_2(t;k)=z(t;k)$ and introducing an appropriate integrating factor, 
\eqref{eq:hsystem} takes the form
\begin{equation}
\frac{dy}{dt}(t;k)=\myi \rho(t;k)z(t;k)\quad\text{and}\quad
\frac{d}{dt}\left[e^{2\varphi(0,t;k)/\eps}z(t;k)\right]= -\myi e^{2\varphi(0,t;k)/\eps}\rho(t;k)y(t;k)-\myi e^{2\varphi(0,t;k)/\eps}\rho(t;k),
\label{eq:hsystem-rewrite}
\end{equation}
where $\varphi$ is defined as
\begin{equation}
\varphi(t_0,t_1;k):=\int_{t_0}^{t_1}\Lambda(s;k)\,ds.
\end{equation}
Note that since $\Re\{\Lambda(s;k)\}\le 0$ by Lemma~\ref{lem:eigenvaluebranch}, the factor $e^{2\varphi(0,t;k)/\eps}$ remains bounded as $t\to +\infty$, so we may build in the boundary conditions \eqref{eq:hnorm} on $\mathbf{h}(t;k)$ by
integrating  equations \eqref{eq:hsystem-rewrite} from $t$ to $+\infty$:
\begin{equation}
\begin{split}
y(t;k)&=-\myi \int_t^{+\infty}\rho(t_1;k)z(t_1;k)\,dt_1\\
z(t;k)&=\myi\int_t^{+\infty}e^{2\varphi(t,t_1;k)\eps}\rho(t_1;k)y(t_1;k)\,dt_1 + F^\eps(t;k),\quad F^\eps(t;k):=
\myi\int_t^{+\infty}e^{2\varphi(t,t_1;k)/\eps}\rho(t_1;k)\,dt_1.
\end{split}
\label{eq:integralequationsystem}
\end{equation}
This system of integral equations is equivalent to 
the differential equations \eqref{eq:hsystem} and boundary conditions \eqref{eq:hnorm}.

To solve \eqref{eq:integralequationsystem}, substitute the first equation into the second and exchange the order of integration to get a closed equation for $z(t;k)$:
\begin{equation}
z(t;k)=\int_t^{+\infty}K^\eps(t,t_1;k)z(t_1;k)\,dt_1 +F^\eps(t;k),
\end{equation}
where the kernel is given by
\begin{equation}
K^\eps(t,t_1;k):=\rho(t_1;k)\int_t^{t_1} e^{2\varphi(t,t_2;k)/\eps}\rho(t_2;k)\,dt_2.
\end{equation}
Again, since $\Re\{\Lambda(s;k)\}\le 0$, we have the $\eps$-independent estimate
\begin{equation}
|K^\eps(t,t_1;k)|\le |\rho(t_1;k)|\int_t^{t_1} |\rho(t_2;k)|\,dt_2\le \|\rho(\cdot;k)\|_1 |\rho(t_1;k)|,\quad t\le t_1.
\end{equation}
We first seek $z(\cdot;k)$ in the Banach space $X$ of continuous functions for $t\ge 0$ decaying to zero as $t\to +\infty$, equipped with the supremum norm.  Note that $F^\eps(\cdot;k)\in X$ because $\rho(\cdot;k)\in L^1(\mathbb{R}_+)$.  Defining a sequence of iterates by setting $z_0(t;k)\equiv 0$ and then
\begin{equation}
z_{n+1}(t;k):=\int_t^{+\infty}K^\eps(t,t_1;k)z_n(t_1;k)\,dt_1 + F^\eps(t;k),\quad n\ge 0,
\end{equation}
one easily sees that $z_n(t;k)$ takes the form of a sum of multiple integrals:
\begin{equation}
z_n(t;k)=\sum_{m=0}^n S_m(t;k),
\end{equation}
where $S_0(t;k):= F^\eps(t;k)$ and for $m\ge 1$,
\begin{equation}
S_m(t;k):=\int_t^{+\infty}\int_{t_1}^{+\infty}\cdots\int_{t_{m-1}}^{+\infty}K^\eps(t,t_1;k)K^\eps(t_1,t_2;k)\cdots
K^\eps(t_{m-1},t_m;k)F^\eps(t_m;k)\,dt_m\,dt_{m-1}\cdots dt_1.
\end{equation}
Obviously, for all $m\ge 0$ we have
\begin{equation}
\begin{split}
|S_m(t;k)|&\le \|\rho(\cdot;k)\|_1^m\|F^\eps\|_\infty \int_t^{+\infty}\int_{t_1}^{+\infty}\cdots
\int_{t_{m-1}}^{+\infty}|\rho(t_1;k)||\rho(t_2;k)|\cdots |\rho(t_m;k)|\,dt_m\,dt_{m-1}\cdots dt_1\\
&=\frac{\|\rho(\cdot;k)\|_1^m\|F^\eps\|_\infty}{m!}\left[\int_t^{+\infty}|\rho(t_1;k)|\,dt_1\right]^m\\
&\le\frac{\|\rho(\cdot;k)\|_1^{2m}}{m!}\|F^\eps\|_\infty,
\end{split}
\end{equation}
where $\|F\|_\infty$ denotes the norm in $X$:
\begin{equation}
\|F\|_\infty:=\sup_{t\ge 0}|F(t)|.
\end{equation}
It follows by comparison with the exponential series that the infinite series
\begin{equation}
z(t;k)=\lim_{n\to\infty}z_n(t;k)=\sum_{m=0}^\infty S_m(t;k)
\end{equation}
converges uniformly on $\mathbb{R}_+$ and hence as each partial sum vanishes as $t\to +\infty$ it converges in $X$.  Moreover, 
\begin{equation}
\|z(\cdot;k)\|_\infty \le \sum_{m=0}^\infty \frac{\|\rho(\cdot;k)\|_1^{2m}}{m!}\|F^\eps\|_\infty =e^{\|\rho(\cdot;k)\|_1^2} \|F^\eps\|_\infty .
\label{eq:zsup}
\end{equation}
It then follows from the first equation of \eqref{eq:integralequationsystem} that $y(\cdot;k)$ also lies in $X$, and that
\begin{equation}
\|y(\cdot;k)\|_\infty \le \|\rho(\cdot;k)\|_1\|z(\cdot;k)\|_\infty\le \|\rho(\cdot;k)\|_1e^{\|\rho(\cdot;k)\|_1^2}\|F^\eps\|_\infty.
\label{eq:ysup}
\end{equation}
Once it is known that both $y$ and $z$ lie in $X$, it follows directly from the integral equations
\eqref{eq:integralequationsystem} that both functions are continuously differentiable, so $h_1(t;k):=1+y(t;k)$ and $h_2(t;k):=z(t;k)$ constitute the unique classical solution of the differential equations \eqref{eq:hsystem} subject to the boundary conditions \eqref{eq:hnorm}.

It remains to estimate $\|F^\eps\|_\infty$.  The part of Assumption~\ref{assumption:data} that
describes the local behavior of $U$ and $H$ near $t=0$ implies that $\rho(t;k)$ is bounded by $t^{-1/2}$ near $t=0$.  This integrable singularity determines the rate of decay of $F^\eps(t;k)$ for small $t$.  While integration by parts shows that $F^\eps(t;k)=\mathcal{O}(\eps)$ for $t$ bounded away from zero, letting $t\downarrow 0$ results in a slower uniform rate of decay, namely $\mathcal{O}(\eps^{1/2})$.  Therefore, $\|F^\eps\|_\infty=\mathcal{O}(\eps^{1/2})$ under the conditions on $U$ and $H$ in Assumption~\ref{assumption:data}.  Taking into account that $\Lambda$ scales as $|k|^2$ for large $k$, while $\rho$ scales as $|k|^{-1}$ for large $k$, a more precise statement is that
for some constant $C_F>0$, the inequality
\begin{equation}
\|F^\eps\|_\infty\le\frac{C_F\eps^{1/2}}{1+|k|^3}
\label{eq:Fsup}
\end{equation}
holds for $k\in \mathcal{Q}_\delta^\mathrm{II}$ and all $\eps>0$ sufficiently small.

Combining \eqref{eq:rL1}, \eqref{eq:zsup}, \eqref{eq:ysup}, and \eqref{eq:Fsup}, 
we see that (since the exponential factors are bounded),
\begin{equation}
|y(0;k)|=\mathcal{O}\left(\frac{\eps^{1/2}}{1+|k|^4}\right)\quad\text{and}\quad
|z(0;k)|=\mathcal{O}\left(\frac{\eps^{1/2}}{1+|k|^3}\right)
\end{equation}
both hold for $k\in \mathcal{Q}_\delta^\mathrm{II}$ and $\eps>0$.  Then, from $h_1(0;k)=1+y(0;k)$ and $h_2(0;k)=z(0;k)$ and \eqref{eq:AstarBstarh} we have
\begin{equation}
A_0(k^*)^*=e^{\myi M(k)/(2\eps)}\left[1+\mathcal{O}\left(\frac{\eps^{1/2}}{1+|k|^4}\right)\right]\quad\text{and}\quad
B_0(k^*)^*=-e^{-\myi S(0)/\eps}e^{\myi M(k)/(2\eps)}\mathcal{O}\left(\frac{\eps^{1/2}}{1+|k|^3}\right).
\end{equation}
In particular, it follows that $d_0(k):=A_0(k^*)^*$ has no zeros for $k\in\mathcal{Q}_\delta^{\mathrm{II}}$ if $\eps>0$ is sufficiently small, and that for $\eps>0$ fixed any zeros of $d_0(k)$ in $\mathcal{Q}_\delta^\mathrm{II}$ must lie in a bounded subset (depending only on the functions $H$ and $U$ satisfying Assumption~\ref{assumption:data}).  Also, it is clear that \eqref{eq:Gammabound} holds true.  This completes the proof of Proposition~\ref{prop:noturningpoints}.

\section{Proof of Proposition~\ref{prop:two-turning-points}}
\label{sec:Appendix-two-turning-points}
\subsection*{Langer transformation to a perturbed Airy equation}
Returning to \eqref{eq:Fsystem} subject to the boundary condition \eqref{eq:Fnorm}, we first consider making a gauge transformation; that is we consider a given invertible matrix $\mathbf{G}(t;k)$ and transform \eqref{eq:Fsystem} to a new unknown matrix $\mathbf{Y}(t;k)$ by the substitution
\begin{equation}
\mathbf{F}(t;k)=\mathbf{G}(t;k)\mathbf{Y}(t;k),
\label{eq:gauge1}
\end{equation}
which transforms \eqref{eq:Fsystem} into the form
\begin{equation}
\eps\frac{d\mathbf{Y}}{dt}(t;k)=\left[\mathbf{G}(t;k)^{-1}\mathbf{B}(t;k)\mathbf{G}(t;k) - 
\eps\mathbf{G}(t;k)^{-1}\frac{d\mathbf{G}}{dt}(t;k)\right]\mathbf{Y}(t;k).
\end{equation}
The main idea of the method is to try to choose $\mathbf{G}(t;k)$ so that the leading term on the right-hand side takes a particularly simple form that can be the basis for a perturbation expansion.
Unfortunately, conjugation by $\mathbf{G}$ does not allow the determinant to be changed; however an additional scalar factor can be introduced by making a change of the independent variable.  So let $z=z(t)$ be a smooth strictly monotone transformation of the independent variable (hence invertible with smooth inverse map $t=t(z)$) , and let $\overline{\mathbf{Y}}(z;k)=\mathbf{Y}(t(z);k)$.  The desired scalar factor then comes from the chain rule:
\begin{equation}
\eps\frac{d\overline{\mathbf{Y}}}{dz}(z;k)=\left(\frac{dt}{dz}\mathbf{G}(t;k)^{-1}\mathbf{B}(t;k)\mathbf{G}(t;k)-\eps\frac{dt}{dz}\mathbf{G}(t;k)^{-1}\frac{d\mathbf{G}}{dt}(t;k)\right)\overline{\mathbf{Y}}(z;k).
\label{eq:overlineYeqn}
\end{equation}
The target form for the leading matrix coefficient on the right-hand side is (for a problem with a single turning point in the $t$-interval of interest) the coefficient matrix of the first-order form of the Airy equation; therefore one tries to set
\begin{equation}
\frac{dt}{dz}\mathbf{G}(t;k)^{-1}\mathbf{B}(t;k)\mathbf{G}(t;k) =\mathbf{A}(z):=\begin{bmatrix}0 & 1\\z & 0\end{bmatrix}.
\label{eq:leadingterm}
\end{equation}
Taking determinants of both sides of \eqref{eq:leadingterm} leads to a differential equation for $z=z(t)$:
\begin{equation}
z\left(\frac{dz}{dt}\right)^2 =-\det(\mathbf{B}(t;k))=\lambda^2=(2k-U(t))^2(k-\mathfrak{a}(t))(\mathfrak{b}(t)-k).
\label{eq:zODE}
\end{equation}
For $k_\mathfrak{a}<k<k_\mathfrak{b}$ and assuming that $k\neq k_0$ and $k\neq k_\infty$, the right-hand side has simple roots at the turning points $t=t_\pm(k)$. 
We consider building two different transformations, denoted $z=z(t)=z_\pm(t;k)$, mapping neighborhoods of the turning points $t=t_\pm(k)$ to corresponding intervals of $z$.   If $z=z_\pm(t;k)$ is to be a smooth invertible transformation, it is necessary that $z_\pm(t_\pm(k);k)=0$, which fixes
the integration constant in each case.  The solution to \eqref{eq:zODE} is obtained (by separating the variables) as follows:
\begin{equation}
z_\pm(t;k)=\mp\sgn(t-t_\pm(k))\left|\frac{3}{2}\int_{t_\pm(k)}^t(U(s)-2k)\sqrt{|(k-\mathfrak{a}(s))(\mathfrak{b}(s)-k)|}\,ds
\right|^{2/3}.
\label{eq:zplusminus}
\end{equation}
This formula\footnote{The nonlinear mapping $t\mapsto z_\pm(t;k)$ of the independent variable, along with the linear gauge transformation  \eqref{eq:gauge1} and its higher-order correction \eqref{eq:YtoW}, constitute the Langer transformation of the system \eqref{eq:Fsystem} in a neighborhood of the turning point $t=t_\pm(k)$.} defines $z_+(t;k)$  in the interval $t_-(k)<t<+\infty$, and it defines $z_-(t;k)$ in the interval $0<t<t_+(k)$.  Moreover, $z_-(t;k)$ is monotone increasing, while $z_+(t;k)$ is monotone decreasing on their respective domains.  It is obvious that within their intervals of definition, $z_\pm(t;k)$ have exactly one more continuous derivative than do $H(\cdot)$ and $U(\cdot)$, as long as $t$ is bounded away from the corresponding turning point $t_\pm(k)$.  However, a local analysis of $z_\pm(t;k)$ for $t$ near $t=t_\pm(k)$ shows that $z_\pm(t;k)$ has only the same number of continuous derivatives as do $H(\cdot)$ and $U(\cdot)$ at the turning point.  Given that $H$ and $U$ are analytic according to  Assumption~\ref{assumption:data}, we conclude that $z_\pm(t;k)$ are real analytic monotone functions in their respective domains of definition.  Note that even though \eqref{eq:zplusminus} appears to allow $z_+(t;k)$ to be defined also for $t\le t_-(k)$ and for $z_-(t;k)$ to be defined also for $t\ge t_+(k)$, there will be an essential loss of smoothness of $z_\pm(t;k)$ at the ``other'' turning point $t_\mp(k)$,
so in fact our approach will be to cover the positive $t$-axis $t>0$ with the two overlapping intervals $(0,t_+(k))$ and $(t_-(k),+\infty)$ and hence use two different changes of independent variable to obtain the desired smoothness.  (Moreover, it is easy to check that $z_-(t;k)$ as defined by \eqref{eq:zplusminus} fails to satisfy \eqref{eq:zODE} for $t>t_+(k)$ and that $z_+(t;k)$ as defined by \eqref{eq:zplusminus} fails to satisfy \eqref{eq:zODE} for $t<t_-(k)$.)  Finally, we record here that
the image of $(t_-(k),+\infty)$ under $z_+(\cdot;k)$ is the interval $(-\infty,z_+(t_-(k);k))$ where $z_+(t_-(k);k)>0$ while the image of $(0,t_+(k))$ under $z_-(\cdot;k)$ is the interval $(z_-(0;k),z_-(t_+(k);k))$ where $z_-(0;k)<0<z_-(t_+(k);k)$.  Also, $z_+(t;k)\sim -t^{2/3}$ as $t\to +\infty$.

With $z_\pm(t;k)$ determined by \eqref{eq:zplusminus}, we can solve \eqref{eq:leadingterm} for the gauge transformation matrix $\mathbf{G}(t;k)$.  Using the fact that $\mathbf{B}(t;k)^2=\lambda^2\mathbb{I} = -\det(\mathbf{B}(t;k))\mathbb{I}$, we obtain the general solution in the form
\begin{equation}
\mathbf{G}(t;k)=\begin{bmatrix}B_{11}(t;k)p_1(t;k) + B_{12}(t;k)p_2(t;k) & z'(t;k)p_1(t;k)\\
B_{21}(t;k)p_1(t;k) + B_{22}(t;k)p_2(t;k) & z'(t;k)p_2(t;k)\end{bmatrix},
\label{eq:Gform}
\end{equation}
where $p_1(t;k)$ and $p_2(t;k)$ are (at this moment) arbitrary functions of $t$.

Whether $t\in (0,t_+(k))$ and $z=z_-(t;k)$ or whether $t\in (t_-(k),+\infty)$ and $z=z_+(t;k)$, due to \eqref{eq:zplusminus} and \eqref{eq:Gform}, equation \eqref{eq:overlineYeqn} takes the form
\begin{equation}
\eps\frac{d\overline{\mathbf{Y}}}{dz}(z;k)=\left(\mathbf{A}(z)+\eps\mathbf{P}(t;k)\right)\overline{\mathbf{Y}}(z;k),
\label{eq:overlineYeqn-II}
\end{equation}
where
\begin{equation}
\mathbf{P}(t;k):=-\frac{1}{z'(t;k)}\mathbf{G}(t;k)^{-1}\frac{d\mathbf{G}}{dt}(t;k).
\label{eq:Pmatrixformula}
\end{equation}
Of course we think of $\mathbf{P}(t;k)$ as a function of $z$ by means of the invertible transformation defined by \eqref{eq:zplusminus}.

It turns out that the error term $\eps\mathbf{P}(t;k)$ in the coefficient matrix is too large to be controllable directly when $\eps\ll 1$ due to the corresponding factor of $\eps$ on the left-hand side of \eqref{eq:overlineYeqn-II}.  However, the terms proportional to $\eps$ in the coefficient matrix can be removed by an explicit near-identity transformation.  For some $\eps$-independent matrix $\mathbf{H}=\mathbf{H}(t;k)$ to be determined, consider the effect of making the substitution
\begin{equation}
\overline{\mathbf{Y}}(z;k)=\left(\mathbb{I}+\eps\mathbf{H}(t;k)\right)\mathbf{W}(z;k)
\label{eq:YtoW}
\end{equation}
in equation \eqref{eq:overlineYeqn-II}.  In order to obtain an equivalent differential equation for $\mathbf{W}(z;k)$, it will be necessary to invert $\mathbb{I}+\eps\mathbf{H}(t;k)$, and while this can always be accomplished pointwise for sufficiently small $\eps$ by Neumann series, it is especially convenient if the inversion can be carried out explicitly.  Hence we assume at this point that $\mathbf{H}(t;k)$ is a nilpotent matrix of the general form
\begin{equation}
\mathbf{H}(t;k)=\myi\sigma_2\mathbf{h}(t;k)\mathbf{h}(t;k)^{\mathsf{T}} = \begin{bmatrix}
h_1(t;k)h_2(t;k) &h_2(t;k)^2\\ - h_1(t;k)^2 & - h_1(t;k)h_2(t;k)\end{bmatrix}
\label{eq:Hform}
\end{equation}
where $\mathbf{h}(t;k)$ is a general vector function of $t$ and $k$.
Therefore, $\mathbf{H}(t;k)^2=\mathbf{0}$, and it follows that the Neumann series for the inverse truncates:  $(\mathbb{I}+\eps\mathbf{H}(t;k))^{-1}=\mathbb{I}-\eps\mathbf{H}(t;k)$.  Under \eqref{eq:YtoW} and \eqref{eq:Hform}, equation \eqref{eq:overlineYeqn-II} becomes
\begin{equation}
\eps\frac{d\mathbf{W}}{dz}(z;k)=\left(\mathbf{A}(z) +
\eps\mathbf{C}_1(t;k) +\eps^2\mathbf{C}_2(t;k) +\eps^3\mathbf{C}_3(t;k)\right)
\mathbf{W}(z;k).
\end{equation}
where
\begin{equation}
\begin{split}
\mathbf{C}_1(t;k)&:=
[\mathbf{A}(z(t;k)),\mathbf{H}(t,k)]+\mathbf{P}(t;k),\\
\mathbf{C}_2(t;k)&:=
[\mathbf{P}(t,k),\mathbf{H}(t;k)]-\mathbf{H}(t;k)\mathbf{A}(z(t;k))\mathbf{H}(t;k)-
\frac{1}{z'(t;k)}\frac{d\mathbf{H}}{dt}(t;k),\\
\mathbf{C}_3(t;k)&:= 
\frac{1}{z'(t;k)}\mathbf{H}(t;k)\frac{d\mathbf{H}}{dt}(t;k)-\mathbf{H}(t;k)\mathbf{P}(t;k)\mathbf{H}(t;k).
\end{split}
\end{equation}
Here $[\mathbf{A},\mathbf{B}]:=\mathbf{A}\mathbf{B}-\mathbf{B}\mathbf{A}$ denotes the matrix commutator.

The question now arises as to how the arbitrary functions $p_j(t;k)$ and $h_j(t;k)$ for $j=1,2$ can be chosen to ensure that $\mathbf{C}_1(t;k)=\mathbf{0}$.  In fact, the condition $\mathbf{C}_1(t;k)=\mathbf{0}$ is equivalent to the two conditions on the matrix $\mathbf{P}(t;k)$ (really, on the two functions $p_j(t;k)$ for $j=1,2$):
\begin{equation}
P_{21}(t;k)+z(t;k)P_{12}(t;k)=0\quad\text{and}\quad P_{11}(t;k)+P_{22}(t;k)=0
\label{eq:Pconditions}
\end{equation}
and the two equations relating $h_j(t;k)$, $j=1,2$ to the functions $p_j(t;k)$, $j=1,2$ in the matrix $\mathbf{P}(t;k)$:
\begin{equation}
h_1(t;k)h_2(t;k)=\frac{1}{2}P_{12}(t;k)\quad\text{and}\quad h_1(t;k)^2+z(t;k)h_2(t;k)^2=-P_{22}(t;k).
\label{eq:hconditions}
\end{equation}
It turns out to be consistent to assume, in addition to \eqref{eq:Pconditions}
that
\begin{equation}
P_{12}(t;k)=0,
\label{eq:P12zero}
\end{equation}
in which case \eqref{eq:Pconditions} implies that also
\begin{equation}
P_{21}(t;k)=0.
\label{eq:P21zero}
\end{equation}
Then, conditions \eqref{eq:hconditions} require that either $h_1(t;k)=0$ or $h_2(t;k)=0$.  Taking\begin{equation}
h_2(t;k)=0,
\label{eq:h2zero}
\end{equation}
conditions \eqref{eq:hconditions} reduce to
\begin{equation}
h_1(t;k)^2=-P_{22}(t;k)=P_{11}(t;k).
\end{equation}
It follows that the matrix $\mathbf{H}(t;k)$ can be expressed in terms of $\mathbf{P}(t;k)$ as
\begin{equation}
\mathbf{H}(t;k)=\begin{bmatrix}0 & 0\\
P_{22}(t;k) & 0\end{bmatrix},
\label{eq:HformII}
\end{equation}
while $\mathbf{P}(t;k)$ has the form
\begin{equation}
\mathbf{P}(t;k)=-P_{22}(t;k)\sigma_3.
\label{eq:PformDiagonal}
\end{equation}
Assuming that $P_{22}(t;k)$ is differentiable with respect to $t$, it follows from \eqref{eq:HformII} and \eqref{eq:PformDiagonal} that
\begin{equation}
\mathbf{C}_2(t;k)=\begin{bmatrix}0 & 0\\Q(z(t;k);k)&0\end{bmatrix},
\end{equation}
where
\begin{equation}
Q(z(t;k);k):=P_{22}(t;k)^2-\frac{1}{z'(t;k)}\frac{dP_{22}}{dt}(t;k),
\label{eq:QP22define}
\end{equation}
and that 
\begin{equation}
\mathbf{C}_3(t;k)=\mathbf{0}.
\end{equation}
It follows that under the gauge transformation \eqref{eq:YtoW}, the differential equation \eqref{eq:overlineYeqn-II} becomes a perturbed Airy equation:
\begin{equation}
\epsilon\frac{d\mathbf{W}}{dz}(z;k)-\mathbf{A}(z)\mathbf{W}(z;k)=
\begin{bmatrix}0 & 0\\
\epsilon^2Q(z;k) & 0\end{bmatrix}\mathbf{W}(z;k).
\label{eq:perturbed-Airy}
\end{equation}

Let us consider finding $p_j(t;k)$, $j=1,2$, so that $P_{12}(t;k)=0$, $P_{21}(t;k)=0$, and also $P_{11}(t;k)+P_{22}(t;k)=0$.  These calculations are based on the formula \eqref{eq:Pmatrixformula} and the representation \eqref{eq:Gform}.  Assuming for the moment that $\det(\mathbf{G}(t;k))\neq 0$, it is easy to see that the condition $P_{11}(t;k)+P_{22}(t;k)=0$ implies that in fact $\det(\mathbf{G}(t;k))$ must be independent of $t$; as the conditions on the elements of $\mathbf{P}(t;k)$ are all linear in $p_j(t;k)$, $j=1,2$, we will assume that they are chosen so that
$\det(\mathbf{G}(t;k))=1$.  It is then easy to see from \eqref{eq:Pmatrixformula} and \eqref{eq:Gform} that the condition $P_{12}(t;k)=0$ implies that the ratio $p_2(t;k)/p_1(t;k)$ is independent of $t$; we therefore write
\begin{equation}
p_2(t;k)=cp_1(t;k)
\end{equation}
where $c$ is independent of $t$.  With the further assumption that $p_1(t;k)\neq 0$, one checks that because $U(t)-2k>0$ for all $t>0$ when $k\in (k_\mathfrak{a},k_\mathfrak{b})$, the condition $P_{21}(t;k)=0$ implies that
\begin{equation}
\frac{d}{dt}\left(\frac{2H(t)+\myi c(2k+U(t))}{-\myi (2k+U(t))+2cH(t)}\right)=0.
\end{equation}
This equation is obviously solved by choosing either $c=\myi$ or $c=-\myi$.  Finally, we return to the condition $\det(\mathbf{G}(t;k))=1$.  This condition yields the identity
\begin{equation}
p_1(t;k)^2=
\left[2z'(t;k)(U(t)-2k)(H(t)\mp(k+\tfrac{1}{2}U(t)))\right]^{-1},\quad c=\pm\myi.
\end{equation}
Recalling the definitions of $\mathfrak{a}(t)$ and $\mathfrak{b}(t)$, this can be written in the form
\begin{equation}
p_1(t;k)^2=\begin{cases}\left[2z'(t;k)(U(t)-2k)(\mathfrak{b}(t)-k)\right]^{-1},&\quad c=\myi,\\
\left[2z'(t;k)(U(t)-2k)(k-\mathfrak{a}(t))\right]^{-1},&\quad c=-\myi.
\end{cases}
\label{eq:p1-squared}
\end{equation}
In each case, $0<t<t_+(k)$ and $t_-(k)<t<+\infty$, we choose the value of $c=\pm\myi$ so that the turning point contained in the corresponding open interval is \emph{not} a root of the factor $(\mathfrak{b}(t)-k)$ or $(k-\mathfrak{a}(t))$ in the denominator of $p_1(t;k)^2$.    For example, if
$t_-(k)$ satisfies $\mathfrak{b}(t_-(k))=k$, then $(k-\mathfrak{a}(t))$ will be nonzero for $t\in (0,t_+(k))$, so we choose $c=-\myi$.
A calculation using the fact that $\det(\mathbf{G}(t;k))=1$ then shows that
\begin{equation}
P_{22}(t;k)=\begin{cases}
\displaystyle 
-\frac{z''(t;k)}{2z'(t;k)^2}+\frac{1}{2z'(t;k)}\frac{d}{dt}\log((U(t)-2k)(\mathfrak{b}(t)-k)),&\quad c=\myi,\\
\displaystyle
-\frac{z''(t;k)}{2z'(t;k)^2}+\frac{1}{2z'(t;k)}\frac{d}{dt}\log((U(t)-2k)(k-\mathfrak{a}(t))),&\quad c=-\myi.
\end{cases}
\label{eq:P22formula}
\end{equation}
(It is not necessary to resolve the sign of $p_1(t;k)$ to determine $P_{22}(t;k)$ uniquely.)  The corresponding function $Q(z;k)$ is then determined from \eqref{eq:QP22define} and \eqref{eq:zplusminus}.

\subsection*{Solution of the perturbed Airy equations}
\begin{lemma}
Let $k_\mathfrak{a}<k<k_\mathfrak{b}$ with $k\neq k_0$ and $k\neq k_\infty$.  Then there exists a constant $K=K(k)$ such that for $z=z_-(t;k)\in (z_-(0;k),z_-(t_+(k);k))$, $|Q(z;k)|\le K$, while for $z=z_+(t;k)\in (-\infty,z_+(t_-(k);k))$, $Q(z;k)$ satisfies an estimate of the form
\begin{equation}
|Q(z;k)|\le \frac{K}{1+z^2}.
\end{equation}
\label{lemma:Q-bound}
\end{lemma}
\begin{proof}
Since $z_-'(\cdot;k)$ and $z_+'(\cdot;k)$ are nonzero on $(0,t_+(k))$ and $(t_-(k),+\infty)$ respectively, it follows that $P_{22}(t;k)$ is in each case an analytic function on the corresponding interval.  It then follows from \eqref{eq:QP22define} that the same is true of $Q(z(t);k)$.  Since $z=z_\pm(t;k)$ is in each case a real analytic bijection, $Q(z;k)$ is smooth on either $(z_-(0;k),z_-(t_+(k);k))$ or $(-\infty,z_+(t_-(k);k))$.  Therefore it only remains to prove that $Q(z;k)=\mathcal{O}(z^{-2})$ as $z\to -\infty$ in the case when $z=z_+(t;k)$.  But from \eqref{eq:P22formula} and Assumption~\ref{assumption:data} it follows that $|P_{22}(t;k)|\sim t^{-2/3}$ and $|P_{22}'(t;k)|\sim t^{-5/3}$ as $t\to +\infty$, so the desired decay follows from \eqref{eq:QP22define} by composition with $z_+(t;k)\sim -t^{2/3}$ as $t\to +\infty$. 
\end{proof}

We now use the estimate of $Q$ recorded in Lemma~\ref{lemma:Q-bound} to solve the perturbed Airy equation \eqref{eq:perturbed-Airy} separately in the intervals
$z=z_-(t;k)\in (z_-(0;k),z_-(t_+(k);k))$ and $z=z_+(t;k)\in (-\infty,z_+(t_-(k);k))$.  The idea in both cases is the same:  consider the perturbation proportional to $\eps^2$ on the right-hand side as a forcing term, and convert the differential equation into an integral equation with the help of the fundamental matrix
\begin{equation}
\mathbf{W}_0(z):=\begin{bmatrix}\mathrm{Ai}(\eps^{-2/3}z) & \mathrm{Bi}(\eps^{-2/3}z)\\
\eps^{1/3}\mathrm{Ai}'(\eps^{-2/3}z) & \eps^{1/3}\mathrm{Bi}'(\eps^{-2/3}z)\end{bmatrix}
\end{equation}
of the unforced problem (we refer to \cite{DLMF} for the definitions and properties of the Airy functions $\mathrm{Ai}(\cdot)$ and $\mathrm{Bi}(\cdot)$).  Indeed, setting 
\begin{equation}
\mathbf{W}(z;k)=\mathbf{W}_0(z)\mathbf{U}(z;k)
\label{eq:W-to-U}
\end{equation}
in \eqref{eq:perturbed-Airy} where $\mathbf{U}(z;k)$ denotes a new matrix unknown,
we obtain the equivalent equation (after canceling a factor of $\eps$)
\begin{equation}
\frac{d\mathbf{U}}{dz}(z;k)=\eps Q(z;k)\mathbf{W}_0(z)^{-1}\begin{bmatrix}0 & 0\\1 & 0\end{bmatrix}\mathbf{W}_0(z)\mathbf{U}(z;k).
\end{equation}
Using the Wronskian identity $\mathrm{Ai}(\cdot)\mathrm{Bi}'(\cdot)-\mathrm{Bi}(\cdot)\mathrm{Ai}'(\cdot)=1/\pi$ to invert $\mathbf{W}_0(z)$, we may write this system in the form
\begin{equation}
\frac{d\mathbf{U}}{dz}(z;k)=\eps^{2/3}\pi Q(z;k)\begin{bmatrix}
-\mathrm{Ai}(\eps^{-2/3}z)\mathrm{Bi}(\eps^{-2/3}z) & -\mathrm{Bi}(\eps^{-2/3}z)^2\\
\mathrm{Ai}(\eps^{-2/3}z)^2 & \mathrm{Ai}(\eps^{-2/3}z)\mathrm{Bi}(\eps^{-2/3}z)\end{bmatrix}
\mathbf{U}(z;k).
\end{equation}
Assuming $\mathbf{U}(z;k)$ is known at some value $z_0$, we integrate to obtain the Volterra equation
\begin{equation}
\mathbf{U}(z;k)=\mathbf{U}(z_0;k) +\eps^{2/3}\pi\int_{z_0}^z
Q(\zeta;k)
\begin{bmatrix}
-\mathrm{Ai}(\eps^{-2/3}\zeta)\mathrm{Bi}(\eps^{-2/3}\zeta) & -\mathrm{Bi}(\eps^{-2/3}\zeta)^2\\
\mathrm{Ai}(\eps^{-2/3}\zeta)^2 & \mathrm{Ai}(\eps^{-2/3}\zeta)\mathrm{Bi}(\eps^{-2/3}\zeta)\end{bmatrix}\mathbf{U}(\zeta;k)\,d\zeta.
\label{eq:Volterra-U}
\end{equation}
This equation may be analyzed with the help of a weight function $\omega$ defined by
\begin{equation}
\omega(s):=\begin{cases}1,&\quad s\le 0\\
e^{-4s^{3/2}/3},&\quad s>0.\end{cases}
\end{equation}
The weight $\omega$ is strictly positive, continuous, and nonincreasing on $\mathbb{R}$.  Also, there is a constant $C>0$ such that the three inequalities
\begin{equation}
|\mathrm{Ai}(s)\mathrm{Bi}(s)|\le\frac{C}{|s|^{1/2}},\quad
|\mathrm{Ai}(s)|^2\omega(s)^{-1}\le \frac{C}{|s|^{1/2}},\quad\text{and}\quad
|\mathrm{Bi}(s)|^2\omega(s)\le\frac{C}{|s|^{1/2}}\quad\text{hold for all $s\in\mathbb{R}$.}
\label{eq:Airy-Inequalities}
\end{equation}
We now introduce the weight $\omega$ into \eqref{eq:Volterra-U} in two different ways.  First, let
\begin{equation}
\mathbf{U}(z;k)=\begin{bmatrix}\omega(\eps^{-2/3}z)^{-1} & 0\\0 & 1\end{bmatrix}\mathbf{V}_>(z;k)
\label{eq:U-V-greater}
\end{equation}
define a new unknown $\mathbf{V}_>(z;k)$.  In terms of $\mathbf{V}_>(z;k)$, the Volterra equation \eqref{eq:Volterra-U} becomes the equivalent equation
\begin{equation}
(1-\mathcal{K}_>)\mathbf{V}_>(z;k)
=\mathbf{R}_>(z;k),\quad \mathcal{K}_>\mathbf{V}(z):=\int_{z_0}^z
\mathbf{K}_>(z,\zeta;k)\mathbf{V}(\zeta)\,d\zeta
\label{eq:Volterra-greater}
\end{equation}
where the matrix kernel $\mathbf{K}_>(z,\zeta;k)$ is given by
\begin{equation}
\mathbf{K}_>(z,\zeta;k):=\eps^{2/3}\pi Q(\zeta;k)\begin{bmatrix}-\omega(\eps^{-2/3}z)\omega(\eps^{-2/3}\zeta)^{-1}\mathrm{Ai}(\eps^{-2/3}\zeta)\mathrm{Bi}(\eps^{-2/3}\zeta) & -\mathrm{Bi}(\eps^{-2/3}\zeta)^2\omega(\eps^{-2/3}z)\\
\mathrm{Ai}(\eps^{-2/3}\zeta)^2 \omega(\eps^{-2/3}\zeta)^{-1}& \mathrm{Ai}(\eps^{-2/3}\zeta)\mathrm{Bi}(\eps^{-2/3}\zeta)
\end{bmatrix},
\label{eq:K-greater}
\end{equation}
and where the forcing term is the matrix function
\begin{equation}
\mathbf{R}_>(z;k):=\begin{bmatrix}\omega(\eps^{-2/3}z)\omega(\eps^{-2/3}z_0)^{-1} & 0\\0 & 1\end{bmatrix}\mathbf{V}_>(z_0;k).
\label{eq:R-greater}
\end{equation}
We will use this form when $z>z_0$ in which case $\omega(\eps^{-2/3}z)\le\omega(\eps^{-2/3}\zeta)$
holds for the first row of the kernel.
Alternately, let
\begin{equation}
\mathbf{U}(z;k)=\begin{bmatrix}1 & 0\\0 & \omega(\eps^{-2/3}z)\end{bmatrix}\mathbf{V}_<(z;k)
\label{eq:U-V-less}
\end{equation}
define a new unknown $\mathbf{V}_<(z;k)$.  In terms of $\mathbf{V}_<(z;k)$, the Volterra equation \eqref{eq:Volterra-U} becomes
\begin{equation}
(1-\mathcal{K}_<)\mathbf{V}(z;k)
=\mathbf{R}_<(z;k),\quad\mathcal{K}_<\mathbf{V}(z):=\int_z^{z_0}\mathbf{K}_<(z,\zeta;k)\mathbf{V}(\zeta)\,d\zeta
\label{eq:Volterra-less}
\end{equation}
where the matrix kernel $\mathbf{K}_<(z,\zeta;k)$ is given by
\begin{equation}
\mathbf{K}_<(z,\zeta;k):=
-\eps^{2/3}\pi Q(\zeta;k)\begin{bmatrix}-\mathrm{Ai}(\eps^{-2/3}\zeta)\mathrm{Bi}(\eps^{-2/3}\zeta) & -\mathrm{Bi}(\eps^{-2/3}\zeta)^2\omega(\eps^{-2/3}\zeta)\\
\mathrm{Ai}(\eps^{-2/3}\zeta)^2 \omega(\eps^{-2/3}z)^{-1}& \omega(\eps^{-2/3}z)^{-1}\omega(\eps^{-2/3}\zeta)\mathrm{Ai}(\eps^{-2/3}\zeta)\mathrm{Bi}(\eps^{-2/3}\zeta)
\end{bmatrix},
\label{eq:K-less}
\end{equation}
and the matrix-valued forcing term is
\begin{equation}
\mathbf{R}_<(z;k):=\begin{bmatrix}1 & 0\\0 & \omega(\eps^{-2/3}z_0)\omega(\eps^{-2/3}z)^{-1}\end{bmatrix}\mathbf{V}_<(z_0;k).
\label{eq:R-less}
\end{equation}
We will use this form when $z<z_0$ in which case $\omega(\eps^{-2/3}z)^{-1}\le\omega(\eps^{-2/3}\zeta)^{-1}$ holds for the second row of the kernel.
Let $t_0$ be a fixed number in the open interval $(t_-(k),t_+(k))$.  We first use \eqref{eq:U-V-greater}--\eqref{eq:K-greater} along with the mapping $t\mapsto z=z_+(t;k)$ to analyze $\mathbf{F}(t;k)$ for $t\in [t_0,+\infty)$.  Then, we use \eqref{eq:U-V-less}--\eqref{eq:K-less} along with the mapping $t\mapsto z=z_-(t;k)$ to analyze $\mathbf{F}(t;k)$ for $t\in (0,t_0]$.

\subsubsection*{Analysis of $\mathbf{F}(t;k)$ for $t_0\le t<+\infty$}
Under the real-analytic bijection $t\mapsto z=z_+(t;k)$ given by \eqref{eq:zplusminus}, the interval $t_0\le t<+\infty$ is identified with the interval $-\infty<z\le z_+(t_0;k)$ where $z_+(t_0;k)>0$, and the mapping reverses orientation, i.e., $z_+'(t;k)<0$.  We begin by resolving the only remaining indeterminacy in the Langer transformation by first noting that in the formula \eqref{eq:p1-squared} we must take $c=\myi\sgn(k^2-k_\infty^2)$ (because $\mathfrak{a}(t_+(k))=k$ for $k_\mathfrak{a}<k<k_\infty<0$ while $\mathfrak{b}(t_+(k))=k$ for $k_\infty<k<k_\mathfrak{b}<0$).  Taking into account that $z_+'(t;k)<0$, we then choose the sign of the square root to obtain $p_1(t;k)$ as
\begin{equation}
p_1(t;k)^{-1}=\begin{cases}\myi\sqrt{-2z_+'(t;k)(U(t)-2k)(\mathfrak{b}(t)-k)},&\quad k_\mathfrak{a}<k<k_\infty\\
\myi\sqrt{-2z_+'(t;k)(U(t)-2k)(k-\mathfrak{a}(t))},&\quad k_\infty<k<k_\mathfrak{b}.
\end{cases}
\label{eq:p1-greater}
\end{equation}
In both cases, the positive square root is meant, and $p_1(t;k)$ is a purely imaginary nonvanishing analytic function of $t$ in the interval $t_0\le t<+\infty$.  The definition \eqref{eq:p1-greater} unambiguously determines the matrices $\mathbf{G}(t;k)=\mathbf{G}_>(t;k)$ and $\mathbf{H}(t;k)=\mathbf{H}_>(t;k)$.  Then, combining the transformations \eqref{eq:gauge1}, \eqref{eq:YtoW}, and \eqref{eq:W-to-U} with \eqref{eq:U-V-greater}, we obtain the exact relation linking $\mathbf{F}(t;k)$ and $\mathbf{V}_>(z;k)$ for $z=z_+(t;k)$:
\begin{equation}
\mathbf{V}_>(z;k)=\begin{bmatrix}\omega(\eps^{-2/3}z) & 0\\0 & 1\end{bmatrix}
\mathbf{W}_0(z)^{-1}(\mathbb{I}+\eps\mathbf{H}_>(t;k))^{-1}\mathbf{G}_>(t;k)^{-1}
\mathbf{F}(t;k),\quad t\ge t_0.
\label{eq:V-greater-F-exact}
\end{equation}
Since $\omega(\eps^{-2/3}z_+(t;k))=1$ for $t>t_+(k)$ (which implies $z_+(t;k)<0$), 
from the asymptotic normalization condition \eqref{eq:Fnorm} on $\mathbf{F}(t;k)$, we expect the following limit to exist for each $\eps>0$:
\begin{equation}
\mathbf{V}_>(-\infty;k):=\lim_{t\to +\infty}\mathbf{W}_0(z_+(t;k))^{-1}(\mathbb{I}+\eps\mathbf{H}_>(t;k))^{-1}\mathbf{G}_>(t;k)^{-1}e^{-\myi(4k^2t+S(t))\sigma_3/(2\eps)}.
\end{equation}
This limit does indeed exist, and it may be computed with the help of (i) the estimate $\mathbf{H}_>(t;k)=\mathcal{O}(t^{-2/3})$, (ii) the asymptotic relations
\begin{equation}
z_+(t;k)=-3^{2/3}|k^2-k_\infty^2|^{2/3}t^{2/3}(1+\mathcal{O}(t^{-1})),\quad t\to +\infty,
\end{equation}
\begin{equation}
\frac{2}{3}(-z_+(t;k))^{3/2}=2|k^2-k_\infty^2|t +\ell(k) + o(1),\quad t\to +\infty,
\end{equation}
where $\ell(k)$ is defined by \eqref{eq:ell-define}
and
\begin{equation}
z_+'(t;k)=-\frac{2}{3^{1/3}}|k^2-k_\infty^2|^{2/3}t^{-1/3}(1+\mathcal{O}(t^{-1})),\quad t\to +\infty,
\end{equation}
and (iii) known asymptotic formulae for Airy functions and their derivatives for large negative $z$.  The result is:
\begin{equation}
\mathbf{V}_>(-\infty;k)=\myi\eps^{-1/6}\sqrt{\frac{\pi}{2}}e^{\myi\pi\sgn(k^2-k_\infty^2)\sigma_3/4}
\begin{bmatrix}1 & -1\\-1 & -1\end{bmatrix}e^{\myi (\ell(k)\sgn(k^2-k_\infty^2)-\tfrac{1}{2}S_\infty)\sigma_3/\eps},
\label{eq:V-greater-infinity}
\end{equation}
where $S_\infty$ is defined by \eqref{eq:S-infty}.

We now consider the Volterra integral equation \eqref{eq:Volterra-greater} with $z_0=-\infty$, on the space of matrix-valued functions defined on $-\infty<z<z_+(t_0;k)$ equipped with the supremum norm (and based on some matrix norm).  It is a simple consequence of Lemma~\ref{lemma:Q-bound}, the inequality $\omega(\eps^{-2/3}z)\le\omega(\eps^{-2/3}\zeta)$ holding for $z\ge \zeta$ and the estimates \eqref{eq:Airy-Inequalities} that the Volterra integral operator $\mathcal{K}_>$ with matrix kernel $\mathbf{K}_>(z,\zeta;k)$ is bounded on this space with operator norm $\mathcal{O}(\eps)$, because
\begin{equation}
\int_{-\infty}^{z_+(t_0;k)}\frac{d\zeta}{|\zeta|^{1/2}(1+\zeta^2)}<\infty.
\end{equation}
It follows easily that $1-\mathcal{K}_>$ is invertible for sufficiently small $\eps$, and that the operator norm of $(1-\mathcal{K}_>)^{-1}-1$ is $\mathcal{O}(\eps)$.  Since $\omega(\eps^{-2/3}z_0)=1$ for $z_0=-\infty$, it is easy to see that $\mathbf{R}_>(z;k)$ defined by \eqref{eq:R-greater} is a bounded continuous function on $(-\infty,z_+(t_0;k))$.  Therefore we learn that ($\|\cdot\|$ is a matrix norm)
\begin{equation}
\sup_{-\infty<z<z_+(t_0;k)}\|\eps^{1/6}(\mathbf{V}_>(z;k)-\mathbf{R}_>(z;k))\|=\mathcal{O}(\eps),\quad \eps\to 0,
\label{eq:V-greater-control}
\end{equation}
because $\eps^{1/6}\mathbf{R}_>(-\infty;k)$ is bounded independently of $\eps$ according to 
\eqref{eq:V-greater-infinity} and the fact that $0<\omega(\eps^{-2/3}z)\le 1$.
It follows that we may solve \eqref{eq:V-greater-F-exact} for $\mathbf{F}(t;k)$, set
$t=t_0$ and take into account that $z_+(t_0;k)>0$ to obtain an asymptotic formula for $\mathbf{F}(t_0;k)$; here we need to use the fact that $\omega(\eps^{-2/3}z_+(t_0;k))=e^{-4z_+(t_0;k)^{3/2}/(3\eps)}$ along with asymptotic expansions of Airy functions for large positive arguments.
The key point is that
\begin{equation}
\mathbf{W}_0(z_+(t_0;k))\begin{bmatrix}\omega(\eps^{-2/3}z_+(t_0;k))^{-1} & 0\\0 & 1\end{bmatrix}=\frac{\eps^{1/6}}{\sqrt{\pi}}e^{2z_+(t_0;k)^{3/2}/(3\eps)}z_+(t_0;k)^{-\sigma_3/4}\left(\begin{bmatrix}
\tfrac{1}{2} &1\\
-\tfrac{1}{2}&1\end{bmatrix}+\mathcal{O}(\eps)\right),\quad \eps\to 0.
\end{equation}
Combining this with \eqref{eq:V-greater-F-exact}, \eqref{eq:V-greater-infinity}, and \eqref{eq:V-greater-control} gives
\begin{equation}
\mathbf{F}(t_0;k)=e^{2z_+(t_0;k)/(3\eps)}
\left(\mathbf{Z}_>(k)+\mathcal{O}(\eps)\right)e^{\myi (\ell\sgn(k^2-k_\infty^2)-\tfrac{1}{2}S_\infty)\sigma_3/\eps},
\quad\eps\to 0,
\label{eq:F-at-t0}
\end{equation}
where
\begin{equation}
\mathbf{Z}_>(k):=-\frac{\myi}{\sqrt{2}}e^{-\myi\pi\sgn(k^2-k_\infty^2)/4}\mathbf{G}_>(t_0;k)z_+(t_0;k)^{-\sigma_3/4}
\begin{bmatrix}1 & 1\\1 & 1\end{bmatrix}.
\end{equation}

\subsubsection*{Analysis of $\mathbf{F}(t;k)$ for $0<t\le t_0$}
Now, we consider the interval $0<t\le t_0$, which is mapped by the real-analytic bijection $t\mapsto z_-(t;k)$ to the interval $z_-(0;k)<z\le z_-(t_0;k)$ with orientation preserved, i.e., $z_-'(t;k)>0$.  To fully determine the Langer transformation in this interval, we need to take $c=\myi\sgn(k^2-k_0^2)$ (because $\mathfrak{a}(t_-(k))=k$ for $k_\mathfrak{a}<k<k_0$ while $\mathfrak{b}(t_-(k))=k$ for $k_0<k<k_\mathfrak{b}$).  Choosing a sign for the square root of $p_1(t;k)^2$ we then obtain
\begin{equation}
p_1(t;k)^{-1}=\begin{cases}\sqrt{2z_-'(t;k)(U(t)-2k)(\mathfrak{b}(t)-k)},&\quad k_\mathfrak{a}<k<k_0\\
\sqrt{2z_-'(t;k)(U(t)-2k)(k-\mathfrak{a}(t))},&\quad k_0<k<k_\mathfrak{b}.
\end{cases}
\label{eq:p1-less}
\end{equation}
Therefore, $p_1(t;k)$ is a strictly positive nonvanishing analytic function of $t$ for $0<t\le t_0$.  This choice determines the matrices $\mathbf{G}(t;k)$ and $\mathbf{H}(t;k)$, here denoted $\mathbf{G}_<(t;k)$ and $\mathbf{H}_<(t;k)$.  Then, according to \eqref{eq:gauge1}, \eqref{eq:YtoW}, \eqref{eq:W-to-U}, and \eqref{eq:U-V-less}, the exact relation between $\mathbf{F}(t;k)$ and $\mathbf{V}_<(z;k)$ for $z=z_-(t;k)$ is
\begin{equation}
\mathbf{V}_<(z;k)=\begin{bmatrix}1 & 0\\0 & \omega(\eps^{-2/3}z)^{-1}\end{bmatrix}
\mathbf{W}_0(z)^{-1}(\mathbb{I}+\eps\mathbf{H}_<(t;k))^{-1}\mathbf{G}_<(t;k)^{-1}\mathbf{F}(t;k),\quad 0<t\le t_0.
\label{eq:V-less-F-exact}
\end{equation}
If we consider in particular $t=t_0$, then $\mathbf{F}(t_0;k)$ is given by \eqref{eq:F-at-t0}, so using the fact that $z_-(t_0;k)>0$ to simplify the weight $\omega$ and using asymptotic formula for Airy functions of large positive arguments, we obtain the following:
\begin{equation}
\mathbf{V}_<(z_-(t_0;k);k)=\eps^{-1/6}e^{\tfrac{2}{3}(z_-(t_0;k)^{3/2}+z_+(t_0;k)^{3/2})/\eps}\left(\mathbf{Z}_<(t_0;k)+\mathcal{O}(\eps)\right)e^{\myi (\ell(k)\sgn(k^2-k_\infty^2)-\tfrac{1}{2}S_\infty)\sigma_3/\eps},
\label{eq:V-less-t0}
\end{equation}
where
\begin{equation}
\mathbf{Z}_<(t_0;k):=-\myi\sqrt{\frac{\pi}{2}}e^{-\myi\pi\sgn(k^2-k_\infty^2)/4}\begin{bmatrix}1 & -1\\\tfrac{1}{2}&\tfrac{1}{2}\end{bmatrix}
z_-(t_0;k)^{\sigma_3/4}\mathbf{G}_<(t_0;k)^{-1}\mathbf{G}_>(t_0;k)z_+(t_0;k)^{-\sigma_3/4}
\begin{bmatrix}1&1\\1&1\end{bmatrix}.
\end{equation}
We claim that the explicit terms in the formula \eqref{eq:V-less-t0} are independent of $t_0\in (t_-(k),t_+(k))$.  Indeed, we have the following results.
\begin{lemma}
Suppose that $k\in (k_\mathfrak{a},k_\mathfrak{b})$ with $k\neq k_0$ and $k\neq k_\infty$, and that $t_-(k)<t_0<t_+(k)$.  Then:
\begin{equation}
\frac{2}{3}z_-(t_0;k)^{3/2}+\frac{2}{3}z_+(t_0;k)^{3/2}=\tau(k),
\label{eq:tau-define}
\end{equation}
where $\tau:(k_\mathfrak{a},k_\mathfrak{b})\to\mathbb{R}_+$ is defined by \eqref{eq:tau-define-1}, and 
\begin{equation}
\mathbf{Z}_<(t_0;k)=-\sqrt{2\pi}e^{-\myi\pi\sgn(k^2-k_0^2)/4}
\begin{bmatrix}1 & 1\\
0 & 0\end{bmatrix}.
\label{eq:Z-less-formula}
\end{equation}
In particular, $\mathbf{Z}_<$ is independent of $t_0$ and piecewise constant in $k\in (k_\mathfrak{a},k_\mathfrak{b})$.
\label{lemma-t0-independent}
\end{lemma}
\begin{proof}
The proof of \eqref{eq:tau-define} follows directly from \eqref{eq:zplusminus}.  To prove 
\eqref{eq:Z-less-formula} one first computes the product $\mathbf{G}_>(t_0;k)z_+(t_0;k)^{-\sigma_3/4}$ from \eqref{eq:Gform} using 
$z(t;k)=z_+(t;k)$ and \eqref{eq:p1-greater} along with $p_2=cp_1$ with $c=\myi\sgn(k^2-k_\infty^2)$  and then one computes $\mathbf{G}_<(t_0;k)z_-(t_0;k)^{-\sigma_3/4}$ from \eqref{eq:Gform} using 
$z(t;k)=z_-(t;k)$ and \eqref{eq:p1-less} along with $p_2=cp_1$ with $c=\myi\sgn(k^2-k_0^2)$.  The resulting formulae may be simplified using the positive one-fourth root of \eqref{eq:zODE}.  The result then follows by a direct calculation.
\end{proof}
Henceforth, we will write $\mathbf{Z}_<(k)$ for $\mathbf{Z}_<(t_0;k)$ in light of this calculation.  We now may consider the Volterra integral equation \eqref{eq:Volterra-less} with $z_0=z_-(t_0;k)$ and
$z<z_0$.  Fixing a matrix norm $\|\cdot\|$ and defining the corresponding supremum norm over the interval $z_-(0;k)<z<z_-(t_0;k)$, it follows from Lemma~\ref{lemma:Q-bound}, the inequality $\omega(\eps^{-2/3}z)^{-1}\le \omega(\eps^{-2/3}\zeta)^{-1}$ holding for $z\le\zeta$ and the estimates \eqref{eq:Airy-Inequalities} that the Volterra integral operator $\mathcal{K}_<$ with matrix kernel $\mathbf{K}_<(z,\zeta;k)$ is bounded with operator norm $\mathcal{O}(\eps)$ because
\begin{equation}
\int_{z_-(0;k)}^{z_-(t_0;k)}\frac{d\zeta}{|\zeta|^{1/2}}<\infty.
\end{equation}
Hence $1-\mathcal{K}_<$ is invertible for sufficiently small $\eps$, and the norm of $(1-\mathcal{K}_<)^{-1}-1$ is $\mathcal{O}(\eps)$.  Therefore,
\begin{equation}
\sup_{z_-(0;k)<z<z_-(t_0;k)}\|\eps^{1/6}e^{-\tau(k)/\eps}\left(\mathbf{V}_<(z;k)-\mathbf{R}_<(z;k)\right)\|=\mathcal{O}(\eps),\quad\eps\to 0,
\label{eq:V-less-control}
\end{equation}
because $\eps^{1/6}e^{-\tau(k)/\eps}\mathbf{R}_<(z;k)$ is bounded independently of $\eps$ according to \eqref{eq:V-less-t0} and taking into account the fact that $\omega(\eps^{-2/3}z_-(t_0;k))\omega(\eps^{-2/3}z)^{-1}\le 1$.

Finally, we set $t=0$ and solve \eqref{eq:V-less-F-exact} for $F(0;k)$ taking into account that 
$\omega(\eps^{-2/3}z_-(0;k))=1$ because $z_-(0;k)<0$.  Using asymptotic expansions of Airy functions for large negative arguments, and using the fact that the second row of $\mathbf{R}_<(z_-(0;k);k)$ is exponentially small compared with the first row due to the factor $\omega(\eps^{-2/3}z_-(t_0;k))$ as $z_-(t_0;k)>0$, we obtain
\begin{equation}
\mathbf{F}(0;k)=-e^{\tau(k)/\eps}\left(e^{\tfrac{2}{3}\myi\sgn(k^2-k_0^2)(-z_-(0;k))^{3/2}\sigma_3/\eps}
\begin{bmatrix}1 & 1\\1 & 1\end{bmatrix}e^{\myi(\ell(k)\sgn(k^2-k_\infty^2)-\tfrac{1}{2}S_\infty)\sigma_3/\eps}+\mathcal{O}(\eps)\right),\quad\eps\to 0
\end{equation}
also with the help of the identity
\begin{equation}
\mathbf{G}_<(0;k)(-z_-(0;k))^{-\sigma_3/4}=\frac{1}{\sqrt{2}}\begin{bmatrix} \myi\sgn(k^2-k_0^2) & -1\\1 & -\myi\sgn(k^2-k_0^2)\end{bmatrix}.
\end{equation}

This completes the rigorous asymptotic calculation of $\mathbf{F}(0;k)$ by the method of Langer transformations.  From \eqref{eq:T-and-F} we then obtain $\mathbf{T}_0(0;k)$ for small $\eps$,
and then from \eqref{eq:basic-transforms} and \eqref{eq:conjugate-transforms} we complete the proof of Proposition~\ref{prop:two-turning-points}.

\end{document}